\documentclass[11pt]{article}
\usepackage[top=1in, bottom=1in, left=1in, right=1in]{geometry}

\usepackage[linktocpage,colorlinks,linkcolor=blue,anchorcolor=blue,citecolor=blue,urlcolor=blue,pagebackref]{hyperref}
\usepackage{euscript,mdframed}
\usepackage{microtype,todonotes,relsize}
\usepackage{amsmath,amsthm, amssymb}
\usepackage{algpseudocode}
\usepackage{bbding}
\usepackage{pifont}
\usepackage{diagbox}
\usepackage{graphicx}

\usepackage{accents}

\usepackage{comment,url,graphicx,relsize}
\usepackage{amssymb,amsfonts,amsmath,amsthm,amscd,dsfont,mathrsfs,mathtools,nicefrac, bm}
\usepackage{float,psfrag,epsfig,color,xcolor,url}
\usepackage{epstopdf,bbm,mathtools,enumitem}
\usepackage{subfigure}
\usepackage[ruled,vlined]{algorithm2e}
\usepackage{tablefootnote}
\graphicspath{{Plots/}}
\usepackage{diagbox}
\usepackage{tikz}
\usetikzlibrary{calc,patterns,angles,quotes}
\usepackage{makecell}
\usepackage{multirow}
\usepackage{booktabs}

\newcommand{\st}{\text{s.t. }}

\newcommand{\E}{\mathbb{E}}

\providecommand{\argmin}{\mathop\mathrm{arg min}}

\newcommand{\proj}{\mathsf{proj}}

\newcommand{\pcd}{\phi^{\text{cd}}}
\newcommand{\R}{\mathbb{R}}
\newcommand{\p}{\partial}

\ifdefined\nonewproofenvironments\else
\ifdefined\ispres\else
\renewenvironment{proof}{\noindent\textbf{Proof.}\hspace*{.3em}}{\qed\\}
\newenvironment{proof-sketch}{\noindent\textbf{Proof Sketch}
  \hspace*{0.em}}{\qed\bigskip\\}
\newenvironment{proof-idea}{\noindent\textbf{Proof Idea}
  \hspace*{0.em}}{\qed\bigskip\\}
\newenvironment{proof-of-lemma}[1][{}]{\noindent\textbf{Proof of Lemma {#1}.}
  \hspace*{0.em}}{\qed\\}
\newenvironment{proof-of-corollary}[1][{}]{\noindent\textbf{Proof of Corollary {#1}.}
  \hspace*{0.em}}{\qed\\}
\newenvironment{proof-of-theorem}[1][{}]{\noindent\textbf{Proof of Theorem {#1}.}
  \hspace*{0.em}}{\qed\\}
\newenvironment{proof-attempt}{\noindent\textbf{Proof Attempt}
  \hspace*{0.em}}{\qed\bigskip\\}

\fi
\newtheorem{theorem}{Theorem}[section]
\newtheorem{lemma}{Lemma}[section]
\newtheorem{corollary}{Corollary}[section]
\newtheorem{proposition}{Proposition}[section]
\newtheorem{assumption}{Assumption}[section]
\newtheorem{remark}{Remark}[section]
\newtheorem{example}{Example}[section]

\newtheorem{definition}{Definition}[section]

\usetikzlibrary{calc}

\allowdisplaybreaks
\usepackage[authoryear,round]{natbib}
\renewcommand*{\backref}[1]{\ifx#1\relax \else Page #1 \fi}
\renewcommand*{\backrefalt}[4]{%
  \ifcase #1 \footnotesize{(Not cited.)}%
  \or        \footnotesize{(Cited on page~#2.)}%
  \else      \footnotesize{(Cited on pages~#2.)}%
  \fi
}

\newcommand*{\colorboxed}{}
\def\colorboxed#1#{%
  \colorboxedAux{#1}%
}
\newcommand*{\colorboxedAux}[3]{%
  \begingroup
    \colorlet{cb@saved}{.}%
    \color#1{#2}%
    \boxed{%
      \color{cb@saved}%
      #3%
    }%
  \endgroup
}
\numberwithin{equation}{section}
\usepackage{xcolor} 
\usepackage{marginnote}

\newcommand{\todol}[2][]{{%
 \let\marginpar\marginnote
 \reversemarginpar
 \renewcommand{\baselinestretch}{0.8}%
 \todo[color=yellow]{#2}}}

\newcommand{\ouralgo}{\texttt{SCiNBiO}}

\title{A Correspondence-Driven Approach for Bilevel Decision-making with Nonconvex Lower-Level Problems}

\author{Xiaotian Jiang, Jiaxiang Li, Jiawen Bi, Mingyi Hong, Shuzhong Zhang\footnote{X.~Jiang, J.~Bi, and S.~Zhang are with Department of Industrial and System Engineering, University of Minnesota, Minneapolis, MN, USA. E-mails: \texttt{jian0851@umn.edu}, \texttt{bi000050@umn.edu}, \texttt{zhangs@umn.edu}. J.~Li and M.~Hong are with Department of Electrical and Computer Engineering, University of Minnesota, Minneapolis, MN, USA. E-mails: \texttt{li003755@umn.edu}, \texttt{mhong@umn.edu}.}}

\begin{document}
\maketitle

\begin{abstract}
    We consider bilevel optimization problems with general nonconvex lower-level objectives and show that the classical hyperfunction-based formulation is unsettled, since the global minimizer of the lower-level problem is generally unattainable. To address this issue, we propose a \emph{correspondence-driven hyperfunction} $\pcd$. In this formulation, the follower is modeled not as a \emph{rational agent} always attaining a global minimizer, but as an \emph{algorithm-based bounded rational agent} whose decisions are produced by a fixed algorithm with initialization and step size. Since $\pcd$ is generally discontinuous, we apply Gaussian smoothing to obtain a smooth approximation $\pcd_\xi$, then show that its value and gradient converge to those of $\pcd$. In the nonconvex setting, we identify that bifurcation phenomena, which arise when $g(x,\cdot)$ has a degenerate stationary point, pose a key challenge for hyperfunction-based methods. This is especially the case when $\pcd_\xi$ is solved using gradient methods. To overcome this challenge, we analyze the geometric structure of the bifurcation set under some weak assumptions. Building on these results, we design a biased projected \texttt{SGD}-based algorithm \ouralgo{} to solve $\pcd_\xi$ with a cubic-regularized Newton lower-level solver. We also provide convergence guarantees and oracle complexity bounds for the upper-level. Finally, we connect bifurcation theory from dynamical systems to the bilevel setting and define the notion of fold bifurcation points in this setting. Under the assumption that all degenerate stationary points are fold bifurcation points, we analyze the Hessian behavior of the lower-level problem $g(x,\cdot)$ at its stationary points when the upper-level parameter $x$ lies in a neighborhood of the bifurcation set. We further characterize the manifold structure of the bifurcation set and establish the oracle complexity of \ouralgo{} for the lower-level problem.
\end{abstract}

\section{Introduction}

Bilevel optimization has emerged as a powerful modeling framework in various machine learning and operations research applications, including hyperparameter optimization \citep{maclaurin2015gradient,franceschi2018bilevel}, meta-learning \citep{finn2017model, franceschi2018bilevel}, reinforcement learning \citep{zeng2024demonstrations,hong2023ttsa} and more recently machine unlearning \citep{reisizadeh2025blur} and large language model alignment \citep{li2024getting,li2025joint}. A bilevel problem involves two nested optimization tasks, where the solution of the upper-level problem depends on the solution of the lower-level problem. In this work, we consider the following bilevel problem with an unconstrained and possibly nonconvex lower-level structure:
\begin{align}\label{eq:original_problem}
&\min_{x\in\mathcal{X}\subset\mathbb{R}^n}\ f(x,y^\ast(x))\\
&\st\quad y^\ast(x)\in S(x):=\argmin_{y\in\R^m}\ g(x,y),\nonumber
\end{align}
where $f:\R^n\to\R$ is called the upper-level objective and $g:\R^m\to\R$ is called the lower-level objective; $\mathcal{X}\subseteq \mathbb{R}^n$ is the feasible set for the upper-level problem. In game theory, the bilevel problem can be thought of as a two-player (Stackelberg) game, the lower-level and upper-level problems are also called the follower and the leader, respectively; see e.g.~\cite{kohli2012optimality,liu2018pessimistic}. 

When the lower-level problem satisfies certain regularity conditions, such as convexity or the Polyak-Łojasiewicz (PL) condition, two main approaches are commonly used to solve bilevel optimization problems: hypergradient-based methods \citep{ghadimi2018approximation,yang2021provably,hong2023two} and value-function-based methods~\citep{OnMomentum-BasedGradientMethods,huang2024optimal,Onpenaltybasedbilevelgradientdescentmethod,lu2026solving}. Hypergradient-based methods aim to optimize the hyperfunction defined by
\begin{align}\label{eq:hyperfunction}
\phi(x) := \min_{y^\ast(x) \in \mathcal{S}(x)} f(x, y^\ast(x)).
\end{align}
To ensure that the hypergradient $\nabla_x \phi(x)$ is well-defined and computable, it is typically required that the lower-level objective $g(x,y)$ is strongly convex in $y$ for any $x$. Value-function-based methods, on the other hand, reformulate the bilevel problem as a single-level problem by introducing the value function constraint:
$$g(x,y)\leq v(x),\quad\text{where}\quad v(x):=\min_{y\in\R^m}g(x,y),$$
and then penalizing this constraint to form an unconstrained formulation:
\begin{align*}
\min_{x\in\mathcal{X},y\in\R^m} F_\gamma(x,y):=f(x,y)+\gamma(g(x,y)-v(x)),
\end{align*}
where $\gamma$ is a penalty parameter. Similarly, to guarantee differentiability and computability of the value function $v(x)$, strong convexity or the PL assumption on the lower-level objective $g(x,y)$ is often imposed.

When the lower-level problem is general nonconvex and lacks such assumptions, these methods break down. For hypergradient-based methods, the solution mapping $S(x)$ can be multi-valued or discontinuous, making $\nabla_x \phi(x)$ undefined. For value-function-based methods, evaluating $v(x)$ requires solving a nonconvex optimization problem to global optimality for each $x$, which is computationally prohibitive. We further argue in Section~\ref{sec:illnessofhyperfunction} that this difficulty is not merely algorithmic in nature. In fact, the classical bilevel formulation~\eqref{eq:original_problem} itself implicitly assumes that the follower is a \emph{rational agent} capable of attaining a global minimizer. In the nonconvex setting, this \emph{rational agent} assumption is unreasonable, rendering the classical hyperfunction~\eqref{eq:hyperfunction} unsettled from both computational and modeling perspectives.

To address this issue, in Section~\ref{section:CorrespondenceDrivenHyperfunction}, we introduce the notion of an \emph{algorithm-based bounded rational agent} for the follower, and correspondingly define the \emph{correspondence-driven hyperfunction}. Instead of assuming access to the entire global minimizer set $S(x)$, the follower applies a prescribed optimization method with step size schedule and initialization to generate the set of algorithmically attainable solutions. The \emph{correspondence-driven hyperfunction} then evaluates the leader’s objective based on these reachable solutions, making the bilevel formulation both more realistic and computationally tractable.

We then develop an algorithm \ouralgo{} (Smooth Correspondence-driven Nonconvex lower-level Bilevel Optimization; see Algorithm~\ref{alg:main}) that minimizes a smoothed version of the \emph{correspondence-driven hyperfunction}. Importantly, \ouralgo{} requires only very weak assumptions, which are significantly milder than those typically imposed in existing bilevel optimization frameworks (see Section~\ref{sec:genericassumption}), while still enabling provable convergence guarantees.

\subsection{Related work}

Recent studies on bilevel optimization (BLO) have relaxed the classical assumption that the lower-level objective $g(x, y)$ is strongly convex in $y$, allowing it to be merely convex \citep{liu2022general,chen2023bilevel,shen2024method} or even nonconvex. For nonconvex lower-level problems, some works still guarantee global optimality under conditions such as the Polyak-Łojasiewicz (PL) inequality, while others adopt weaker assumptions to handle more general objectives, such as Morse-type functions, where global optimality may not be attainable. Accordingly, we categorize existing nonconvex methods based on whether the lower-level global minimum is guaranteed by design.\\

\noindent\textbf{Lower-Level Global Optimality Guarantees:} A notable nonconvex assumption enabling lower-level global optimality is the Polyak-Łojasiewicz (PL) condition, under which every stationary point is a global minimizer and gradient-based methods can find a global optimum. The works \cite{OnMomentum-BasedGradientMethods,huang2024optimal} propose hypergradient-based methods and prove non-asymptotic convergence under a PL condition for the lower-level problem; they also assume that the set of optimal solutions to the lower-level problem is a singleton. \cite{Onpenaltybasedbilevelgradientdescentmethod} develops a value-function-based penalty method and establishes non-asymptotic convergence under a PL condition for the lower-level problem in $y$. In \cite{chen2024finding}, the authors design an algorithm that either requires the penalized hyperfunction to satisfy a PL condition in $y$, or assumes a PL condition in $y$ for the lower-level problem together with a singleton solution set, and proves non-asymptotic convergence under these assumptions. The study \cite{chen2025set} shows that, when the lower-level problem satisfies the PL condition in $y$, the optimistic hyper-objective function is weakly concave and the pessimistic hyper-objective function is weakly convex. A penalty-based approach is also taken by \cite{jiang2025beyond}, who establish non-asymptotic convergence under the assumption that the penalized problem is PL in $y$. In \cite{masiha2025superquantile}, the authors assume a local PL condition (the PL inequality holds in a neighborhood of each local minimizer with a uniform constant) together with certain regularity conditions. Under these assumptions, they show that global optimality can still be attained. They propose a smoothing approximation of the pessimistic hyperfunction, design an algorithm, and establish its non-asymptotic convergence. Because minimax problems are special cases of bilevel optimization, it is relevant that \cite{lu2025first} adopts a local KL assumption, weaker than the PL condition, and under this assumption develops a proximal gradient-based method with non-asymptotic convergence.
\\

\noindent\textbf{Without Lower-Level Global Optimality Guarantees:} Many recent studies relax lower-level assumptions from guaranteeing global optimality to weaker ones, under which only a stationary point can be obtained. All of the following works fall into this category. The work \cite{liu2022bome} designs a barrier-based algorithm and, under a local PL assumption (PL condition holds locally around each local minimizer with a uniform constant), proves its non-asymptotic convergence. In \cite{liu2024moreau}, a Moreau envelope-based reformulation is proposed, together with a penalty-based algorithm; under the assumption that the value function is weakly convex, the method is shown to converge to a stationary point of the Moreau envelope-based reformulation and achieves a non-asymptotic solution guarantee. In \cite{lutsp}, the author introduces a stochastic primal-dual method based on a two-sided smoothing of the Moreau envelope reformulation; under the assumption that the lower-level objective is weakly convex, the algorithm is proven to converge to an $\epsilon$-KKT point with a non-asymptotic convergence. The results in \cite{bolte2025bilevel} show that if the lower-level problem is Morse in $y$ for any $x$, then there exist finitely many disjoint lower-level local-minimizer curves parameterized by $x$. A hyperfunction can thus be defined for each curve, and an algorithm is designed that converges asymptotically to a stationary point of one such hyperfunction. \cite{xiao2025hybrid} proposes a hybrid algorithm that, without requiring strong assumptions, achieves asymptotic convergence to a KKT stationary point of the original problem. As a special case of bilevel optimization, the two-stage problem is also examined in \cite{lou2025decomposition}, where a constrained nonconvex lower-level problem is addressed via a log-barrier reformulation and an algorithm with asymptotic convergence guarantees is developed.

\subsection{Contributions}
In this work, we introduce a novel, more practical problem formulation along with its smoothed version, and design an algorithm for efficiently solving its smoothing. In particular, our main contributions can be summarized as follows:
\begin{itemize}
    \item We show that the classical hyperfunction definition $\phi(x)$ is unsettled in the nonconvex lower-level setting. It requires the lower-level follower to be a \emph{rational agent} whose decision $y^\ast(x)$ is always contained in the global minimizer set of $g(x,y)$ with respect to $y$. This \emph{rational agent} assumption makes optimizing the hyperfunction $\phi(x)$ intractable from the computational perspective and unreasonable from the modeling perspective (see Section~\ref{sec:illnessofhyperfunction}). We instead introduce a novel \emph{correspondence-driven hyperfunction} $\pcd(x)$. It replaces the lower-level \emph{rational agent} assumption with an \emph{algorithm-based bounded rational agent}, who generates its decision directly through a prescribed algorithm with a given initialization and step size. This makes the formulation more meaningful and computationally tractable. To overcome the discontinuity of $\pcd(x)$, we employ Gaussian smoothing to construct a smooth approximation $\pcd_\xi(x)$. We prove that, at points where $\pcd(x)$ is continuous, the function value of $\pcd_\xi(x)$ converges to that of $\pcd(x)$ as the smoothing parameter $\xi\to 0$; and at points where $\pcd(x)$ is differentiable, the gradient of $\pcd_\xi(x)$ converges to that of $\pcd(x)$ (see Section~\ref{sec:Correspondence_Driven_Reformulation_and_Its_Smoothing} for more details);
    \item We point out that some existing, relatively general assumptions still impose restrictions on the class of nonconvex problems, and introduce the notion of a prevalent assumption: one that holds with probability one after applying a specific type of arbitrarily small perturbation (e.g., adding a linear term) within a given function class. We prove that the assumption “for almost every $x$, the function $g(x,\cdot)$ is Morse” is prevalent (refer to Theorem~\ref{theorem:nowhere_dense}), and identify that the set of $x$ for which $g(x,\cdot)$ is not Morse, although of measure zero, represents a key difficulty for hyperfunction-based algorithms in the nonconvex lower-level setting (see Section~\ref{section:LimitationsofCommonRegularityAssumptions},\ \ref{sec:genericassumption}). We refer to this set as the bifurcation set (Definition~\ref{def:bifurcationpointset});
    \item We study the geometric structure of the bifurcation set in certain cases (see Section~\ref{sec:measureofbifurcationpoints}). In particular, when the lower-level problem is semi-algebraic, the bifurcation set admits a stratified manifold structure and has Minkowski dimension at most $n-1$, where $n$ is the dimension of the domain of $x$ (Theorem~\ref{theorem:poly_Minkowski});
    \item We design a biased projected \texttt{SGD}-based algorithm \ouralgo{} that estimates the gradient of $\pcd_\xi(x)$ via sampling, with the lower-level responses computed using the cubic-regularized Newton method. We establish its convergence and derive the oracle complexity of the upper-level problem (Theorem~\ref{thm:biased_sgd_conv}). As a key step, we show that when $g(x,\cdot)$ is Morse, the cubic-regularized Newton method exhibits a two-phase update behavior, and the iteration sequence eventually converges to a second-order stationary point of the lower-level problem (Lemma~\ref{lemma:approximationerror});
    \item We relate the notion of bifurcation points in the bilevel setting to those in dynamical systems, and introduce the concept of fold bifurcation points in the bilevel setting (Definition~\ref{def:fold}). Under the assumption that all bifurcation points are fold bifurcation points, we analyze the relationship between the strong convexity parameter of $g(x,\cdot)$ at a local minimizer $y$ and the distance from $x$ to the bifurcation set. This analysis yields the oracle complexity of the lower-level problem solved by \ouralgo{} (for more details see Section~\ref{sec:foldbifurcation}).
    \item Through experiments, we demonstrate the effectiveness of \ouralgo{}. First, we apply it to nonconvex-nonconcave minimax problems, which are a specific case of bilevel optimization with a nonconvex lower-level. The results show that \ouralgo{} converges from all random initializations, whereas \texttt{GDA} struggles with cyclic behaviors and fails to converge in several cases. This highlights the stability and reliability of \ouralgo{} in tackling challenging nonconvex problems. Second, we use \ouralgo{} for bilevel hyperparameter optimization, specifically tuning a nonnegative $\ell_2$-regularization parameter $\lambda$ for a neural network. In this setting, \ouralgo{} consistently outperforms \texttt{V-PBGD} and \texttt{BOME} in terms of both convergence speed and stability, demonstrating its superior performance in solving bilevel optimization problems with nonconvex lower-level objectives.
\end{itemize}

\section{Main difficulties}\label{sec:difficulty_hyperfunction}

In this section, we discuss the main difficulties in bilevel optimization with a nonconvex lower-level problem, affecting both problem formulation and algorithm design.

\subsection{Unsettledness of the Hyperfunction with Nonconvex Lower-level}\label{sec:illnessofhyperfunction}

Define $S(x)=\argmin_{y\in\mathcal{Y}}g(x,y)$ as the set of global optimizers of the lower-level problem. The hyperfunction is defined as
\begin{align}\label{eq:blo_optimistic}
    \phi(x):=\min_{y\in S(x)}f(x,y)=f(x,y^\ast(x)).
\end{align}
When the lower-level objective $g(x,y)$ is nonconvex in $y$, the definition of the hyperfunction is unsettled for two main reasons:
\begin{itemize}
    \item First, from a computational perspective, when the lower-level objective $g(x,y)$ is nonconvex in $y$, the set of global optimizers $S(x)$ is generally unattainable. As a result, the hyperfunction value, though well-defined in theory, cannot be evaluated in practice.
    \item Second, from a modeling perspective, particularly in game-theoretic settings, bilevel problems reflect a hierarchical structure where the leader selects $x$, and the follower reacts by choosing $y$ that minimizes a lower-level objective function $g(x, y)$ in response to the leader’s decision~\citep{labbe2016bilevel,Silvrio2022ABO}. This classical formulation implicitly treats the follower as a \emph{rational agent} who can always attain a global minimizer. Such an assumption is unrealistic, as it endows the follower with unlimited computational power capable of solving generally intractable nonconvex problems to global optimality. In practice, it is unreasonable to assume that followers possess such ``superrational capacity", and this assumption severely disconnects the model from actual decision-making behavior. Consequently, it is not reasonable for the leader to update $x$ under the premise that the follower always returns a global optimum.
\end{itemize}

To bypass this global optimality challenge, many approaches replace the requirement of having global minimum for the lower-level problem with local minima or stationary points. For example, in \cite{liu2024moreau,bolte2025bilevel,lou2025decomposition}, the authors relax the problem \eqref{eq:blo_optimistic} to the following problem
\begin{align}
    &\min_{x\in \mathcal{X},y\in \mathcal{Y}} f(x,y)\\
    &\st\quad y\in\arg\min\text{-loc}_{y\in\mathcal{Y}}\ g(x,\cdot),\nonumber
\end{align}
where $\arg\min\text{-loc}_{y\in\mathcal{Y}} g(x,y)$ means the set of local minimizes of $g(x,y)$. However, this alternative formulation does not address the fundamental difficulties discussed above from both computational and modeling perspectives. In practice, it is still impossible to find all local minima and select the one that yields the smallest value of the upper-level objective $f(x,\cdot)$. If the set of all local minimizers of the lower-level problem is further restricted to a subset of them, another difficulty arises: how should such a subset be selected, and which local minimizers are considered more “representative” than others?

\subsection{Limitations of Common Regularity Assumptions}\label{section:LimitationsofCommonRegularityAssumptions}

Due to the inherent difficulties highlighted in Section~\ref{sec:illnessofhyperfunction}, the literature commonly imposes some regularity conditions on the lower-level problem $g(x,y)$, which are required to hold for \textbf{every} $x$. Examples of such conditions include:
\begin{itemize}
    \item Morse-type parametric qualifications~\citep{bolte2025bilevel};
    \item Local Polyak-Łojasiewicz (PL) conditions~\citep{liu2022bome,masiha2025superquantile}.
\end{itemize}
These assumptions ensure that the set of local solutions to the lower-level problem exhibits favorable geometric structure. Let us discuss these conditions in more detail below.

\textbf{Morse-type parametric qualifications} include the following parameterized Morse property: for \textbf{every} $x\in\mathcal{X}$, any stationary point of $g(x,y)$ with respect to $y$ is non-degenerate, which means $\nabla_{yy}^2 g(x,y)$ is nonsingular. According to \citet[Proposition 3.6]{bolte2025bilevel}, there exists an integer $M \geq 0$ such that for each $x$, the function $g(x,\cdot)$ admits exactly $M$ non-degenerate critical points, denoted by $y_1(x), \dots, y_M(x)$. As $x$ varies, the graph of each mapping $x \mapsto y_i(x)$, i.e., $\{(x,y_i(x)):x\in\mathcal{X}\}$, traces out an $n$-dimensional manifold embedded in $\mathcal{X} \times \mathcal{Y}$. This structure allows the definition of $M$ distinct hyperfunctions $\phi_i(x):=f(x,y_i^\ast(x))$, and \texttt{SMBG} Algorithm proposed in~\cite{bolte2025bilevel} is designed to compute a local minimum of one such $\phi_i(x)$~\citep[Theorem 4.2]{bolte2025bilevel}. 

\textbf{Local PL condition} requires the existence of a constant $\mu>0$, independent of $x$, such that for \textbf{every} $x\in\mathcal{X}$, the function $g(x,\cdot)$ satisfies the following inequality: for any connected component $\mathcal{M}'(x)$ of the set $\mathcal{M}(x)$ of local minima, there exists an open neighborhood $\mathcal{N}(\mathcal{M}'(x))\supset\mathcal{M}'(x)$ such that
$$\forall y\in \mathcal{N}(\mathcal{M}'(x)),\quad g(x,y)-\min_{y'\in\mathcal{N}(\mathcal{M}'(x))}g(x,y')\leq\frac{1}{2\mu}\|\nabla_y g(x,y)\|^2.$$
This condition, combined with additional regularity assumptions~\citep[Definition 2.2]{masiha2025superquantile}, ensures that for each fixed $x$, the set of local minimizers of $g(x,\cdot)$ forms a connected manifold~\citep[Proposition 2.1]{masiha2025superquantile}. 

However, these assumptions, which are required to hold for \textbf{every} $x$, are in fact strong. Both the Morse condition and the local PL condition impose nontrivial structural constraints on the function $g(x,y)$. In what follows, we demonstrate that neither assumption can generally be expected to hold for \textbf{every} $x$.

In Theorem~\ref{theorem:nowhere_dense}, we show that for any smooth function, an arbitrarily small linear perturbation can yield a new function that is Morse in $y$ for \textbf{almost every} $x$; that is, all stationary points of $g(x,y)$ in $y$ are non-degenerate for \textbf{almost every} $x$. However, this \textbf{almost-everywhere} property does not imply the stronger \textbf{pointwise} Morse condition, which requires the function to be Morse for \textbf{every} $x$. The following counterexample illustrates this distinction: there exist functions that are Morse in $y$ for \textbf{almost every} $x$, yet no arbitrarily small perturbation, for which the changes in the function value and in its first and second derivatives are all sufficiently small, can make them Morse for \textbf{every} $x$.

\begin{figure}[htbp]
    \centering
    \begin{minipage}[t]{0.4\textwidth}
        \centering
        \includegraphics[width=\linewidth]{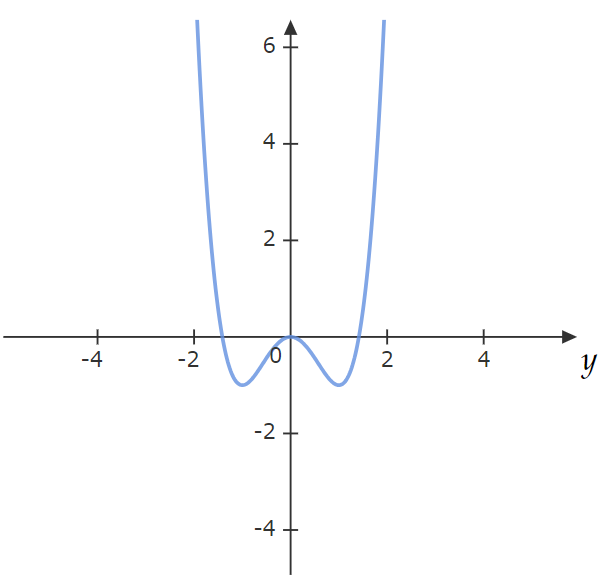}
        \caption{Graph of $g(0,y)$ with three non-degenerate stationary points}
        \label{fig:y^2}
    \end{minipage}
    \hfill
    \begin{minipage}[t]{0.4\textwidth}
        \centering
        \includegraphics[width=\linewidth]{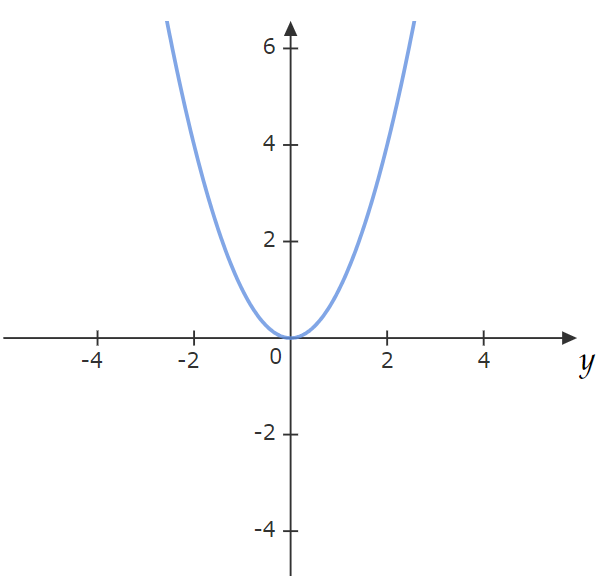}
        \caption{Graph of $g(1,y)$ with one non-degenerate stationary point}
        \label{fig:y^4-2y^2}
    \end{minipage}
\end{figure}

\begin{example}\label{example:almosteverymorse}
    Suppose the smooth function $g(x,y):\mathbb{R}\times\mathbb{R}\to\mathbb{R}$ satisfies $g(0,y)=y^4-2y^2$, and $g(1,y)=y^2$. As illustrated in Figures~\ref{fig:y^2} and~\ref{fig:y^4-2y^2}, $g(0,y)$ has three non-degenerate stationary points, and $g(1,y)$ has one non-degenerate stationary point. If $g$ is slightly perturbed to a new function $\widetilde{g}$, with the perturbation term having sufficiently small function value, for which the changes in the function value and in its first and second derivatives are all sufficiently small, then by the non-degeneracy of the stationary points of $g$ at $x=0$ and $x=1$, $\widetilde{g}$ will still have three non-degenerate stationary points near $x=0$ and one near $x=1$. We will show that no matter which sufficiently small perturbation we choose, there always exists at least one point $x\in (0,1)$ such that $\widetilde{g}(x,y)$ is not Morse in $y$.
    
    Assume, for the sake of contradiction, that $\widetilde{g}(x,y)$ has no degenerate stationary points for any $x\in[0,1]$. Note that $\widetilde{g}(x,y)$ is continuously differentiable. By the implicit function theorem, if $\nabla_y \widetilde{g}(\bar{x},\bar{y})=0$ and the Hessian $\nabla^2_{yy} \widetilde{g}(\bar{x},\bar{y})$ is non-singular, then there exists a neighborhood $I$ of $\bar{x}$ and a unique smooth function $y:I\to \mathbb{R}$ with $y(\bar{x})=\bar{y}$ such that $\nabla_y\widetilde{g}(x,y(x))=0$ for every $x\in I$. Since $[0,1]$ is compact, we can cover it with finitely many such intervals $I_1,...,I_k$ on each of which the implicit function theorem applies. Therefore, for any $\bar{x} \in [0,1]$ and stationary point $\bar{y}$ of $\widetilde{g}(\bar{x}, \cdot)$, we can construct a smooth curve $x \mapsto y(x)$ such that $y(\bar{x}) = \bar{y}$ and $\nabla_y \widetilde{g}(x, y(x)) = 0$ for every $x \in [0,1]$. Applying the above argument to the three non-degenerate stationary points of $\widetilde{g}(0,y)$, we obtain three smooth curves $y_1(x)$, $y_2(x)$, and $y_3(x)$ defined on $[0,1]$, such that for each $x \in [0,1]$, the point $y_i(x)$ is a stationary point of $\widetilde{g}(x, \cdot)$. By the local uniqueness guaranteed by the implicit function theorem, these curves cannot intersect, i.e., $y_i(x) \neq y_j(x)$ for $i \neq j$ and all $x \in [0,1]$. In particular, at $x = 1$, the function $\widetilde{g}(1,y)$ must have at least three distinct stationary points. However, $\widetilde{g}(1,y)$ has only one stationary point. This contradiction implies that there must exist some $x\in(0,1)$ at which $\widetilde{g}(x,y)$ admits a degenerate stationary point. 
\end{example}

Example~\ref{example:almosteverymorse} shows that the set of functions that are Morse in $y$ for \textbf{every} $x$ is not dense in the space of smooth functions.

One can also consider the assumption that there exists {\bf a global constant} $\mu>0$ such that $g(x,y)$ satisfies the local $\mu$-PL condition in $y$ for \textbf{every} $x$. This is a strong assumption, and the set of smooth functions satisfying it is not dense in the space of smooth functions. Consider again Example \ref{example:almosteverymorse}. At $x=0$, the function $g(0,y)$ has two non-degenerate local minimizers, whereas at $x=1$, $g(1,y)$ has one. This implies that as $x$ increases from $0$ to $1$, at least one of the local minimizers of $g(0,y)$ must disappear through a degeneracy. Let $\bar{x}\in(0,1)$ denote the smallest point at which degeneration occurs. Consider the smooth function $y : [0, \bar{x}) \to \mathbb{R}$ satisfying $\nabla_y g(x, y(x)) = 0$ for all $x \in [0, \bar{x})$, with initial value $y(0)$ equal to a non-degenerate local minimizer of $g(0,y)$. As $x$ approaches $\bar{x}$ from the left, the point $y(x)$ converges to a degenerate stationary point $y(\bar{x})$ of $g(\bar{x},y)$. For every $x\in[0,\bar{x})$, the point $y(x)$ is a locally strongly convex minimizer of $g(x,\cdot)$. As $x$ approaches $\bar{x}$ from the left, the local strong convexity constant at $y(x)$ tends to zero since $g(\bar{x},y)$ is degenerate at $y(\bar{x})$. Note that local strong convexity implies the local PL condition, and the PL constant $\mu(x)$ in this case coincides with the local strong convexity constant. It follows that $\mu(x)$ converges to $0$ as $x$ approaches $\bar{x}$ from the left. Therefore, no uniform constant $\mu>0$ can exist such that $g(x,y)$ satisfies the $\mu$-PL condition in $y$ for all $x\in[0,1]$.

\subsection{Beyond Pointwise Regularity: Prevalent Assumption and its Challenge}\label{sec:genericassumption}

We begin by introducing the notion of \emph{prevalence}, which provides a measure-theoretic way to describe properties that are common in infinite-dimensional spaces such as spaces of smooth functions.

\begin{definition}[Prevalence \citep{hunt1992prevalence}]
Let $\mathcal{V}$ be a topological vector space (e.g., the space of $C^r$ functions), and let $\mathcal{S} \subseteq \mathcal{V}$ be a Borel-measurable subset. We say that $\mathcal{S}$ is \emph{prevalent} if there exists a finite-dimensional subspace $P \subseteq \mathcal{V}$, called a \emph{probe set}, such that for every $v \in \mathcal{V}$, the set $\{p \in P : v + p \in \mathcal{S}\}$ has full Lebesgue measure in $P$.
\end{definition}

We say that a property is prevalent in $\mathcal{V}$ if the set $\mathcal{S}\subset\mathcal{V}$ of elements having this property is prevalent. Intuitively, this means there exists a way $P$ to perturb any element $v \in \mathcal{V}$ such that, after applying a random perturbation, the perturbed element $v+p$ satisfies the desired property with probability one.

To broaden the scope of the problem under consideration, we move beyond the relatively strong pointwise regularity assumptions such as requiring $g(x,y)$ to be Morse or local $\mu$-PL in $y$ for \textbf{every} $x$, where Morse function is formally defined as follow:
\begin{definition}[Morse function]
    $g(x,y)$ is said Morse in $y$ if every stationary point with respect to $y$ is non-degenerate, that is, whenever $\nabla_y g(x,y)=0$, it holds that $\det(\nabla^2_{yy} g(x,y))\ne 0$.
\end{definition}
We then adopt a weaker and prevalent assumption: 
\begin{assumption}\label{assumption:nowhere_dense}
    $g(x,y)$ is Morse in $y$ for \textbf{almost every} $x$.
\end{assumption}
The prevalence of this assumption can be established by the following Theorem~\ref{theorem:nowhere_dense}. This assumption holds with probability one for any smooth function under a random linear perturbation, thereby capturing a much broader class of problems.
\begin{theorem}\label{theorem:nowhere_dense}
    For any smooth function $g(x,y)$, we uniformly choose $a$ from $[-\nu,\nu]^m$, where $\nu>0$ can be any constant. We perturb $g(x,y)$ as follows
    $$\widetilde{g}_a(x,y):=g(x,y)+a^\top y.$$
    Then, with probability one (i.e., almost surely with respect to the random choice of $a$), $\widetilde{g}_a(x,y)$ is Morse in $y$ for \textbf{almost every} $x$.
\end{theorem}

\begin{proof}
    See Appendix~\ref{proof:theorem:nowhere_dense}.
\end{proof}

To understand why Theorem \ref{theorem:nowhere_dense} holds, we draw insight from differential topology. In particular, Sard's theorem tells us that for a smooth function, the set of critical values, meaning the values at which the Jacobian matrix fails to be of full rank, has measure zero. If we apply Sard's theorem to the gradient mapping of a function, we obtain a powerful conclusion: for almost all small perturbations, the resulting function becomes a Morse function. In particular, by applying Sard's theorem to the gradient map and combining it with Fubini's theorem, we show that the set of $x$ for which the perturbed function $\widetilde{g}_a(x,y)$ fails to be Morse has measure zero for almost every perturbation $a$.

It is worth noting that the notion of “prevalent” here does not refer to perturbing a single function $g(x,\cdot)$ to satisfy a property, but rather to perturbing the entire parametric family $\{g(x,\cdot)\}_{x\in\mathcal{X}}$ so that it satisfies the property. The specific distinction is presented in the following remark.

\begin{remark}[Two viewpoints on “prevalent”]\label{rmk:generic-vs-ae}
We introduce two ways to call the Morse-in-$y$ property “prevalent”.

\smallskip
\noindent\textbf{(i) Single-function viewpoint.}
Fix $x$ and regard $g(x,\cdot)$ as a function of $y$ only.
For any smooth $g$ one can add an arbitrarily small linear perturbation
$g(x,y)+a^\top y$.  
By Sard's theorem, this perturbed function is Morse with probability~1 (see proof of Theorem~\ref{theorem:nowhere_dense}).
Hence, “being Morse” is prevalent in the space of smooth functions of $y$.

\smallskip
\noindent\textbf{(ii) Parametric-family viewpoint.}
Now regard $g$ as a family $\{g(x,\cdot)\}_{x\in\mathcal X}$ of smooth functions,
parameterized by $x$.
A perturbation 
cannot simultaneously regularize every member of this family (see Example~\ref{example:almosteverymorse}).
Theorem~\ref{theorem:nowhere_dense} states that, after an arbitrarily small linear perturbation
$g(x,y)+a^\top y$, which can be viewed as a uniform perturbation applied to the entire function family $\{g(x,\cdot)\}_{x\in\mathcal{X}}$, the probability (over $a$) that $g(x,\cdot)$ is Morse for almost every $x$ is 1.
Thus, the condition
\[
g(x,y)\text{ is Morse in }y\text{ for almost every }x
\]
is a prevalent property of the function family $\{g(x,\cdot)\}_{x \in \mathcal{X}}$.
\end{remark}

Although the difference between assuming that $g(x, y)$ is Morse in $y$ for \textbf{every} $x$ and Assumption~\ref{assumption:nowhere_dense} lies only in a measure-zero subset, the impact of this relaxation is substantial. The set of parameter values where $g(x, \cdot)$ fails to be Morse, referred to as the bifurcation point set and formally defined below, may introduce significant theoretical and practical challenges. 
\begin{definition}[Bifurcation Point Set]\label{def:bifurcationpointset}
    Let $g:\mathcal{X}\times\R^m\to\R$ be a twice continuously differentiable function. We define the bifurcation point set as
    \begin{align*}
    \widetilde{X}:=\{x\in\mathcal{X}:g(x,\cdot)\text{ is not a Morse function}\}.
\end{align*}
\end{definition}
\begin{remark}\label{remark:close}
    Note that Morse is an open property for $x$, i.e., if $g(\bar{x},y)$ is Morse in $y$, then there exists an open neighborhood $U$ of $\bar{x}$ such that $g(x,y)$ is Morse in $y$ for any $x\in U$. This directly implies that $\widetilde{\mathcal{X}}$ is closed.
\end{remark}

The bifurcation point set gives rise to the following two challenges, from both theoretical and practical perspectives:\\

\noindent\textbf{Challenge from theoretical perspective:} The presence of the bifurcation point set makes the structure of the lower-level solutions highly difficult to analyze. Because the stationary points of $g(x,\cdot)$ may disappear or emerge as $x$ varies (see Example~\ref{example:almosteverymorse}), the conclusion in \citet[Section 2.2]{bolte2025bilevel}, which states that there exist finitely many disjoint stationary curves $y^{(1)}(x),\dots,y^{(M)}(x)$ (each being a stationary point of $g(x,\cdot)$ and varying smoothly with $x$), no longer holds. Such stationary curves may intersect at some $x$ or vanish.
New stationary curves may also emerge at some $x$. Analyzing the evolution of solution curves by studying the structure of the bifurcation point set is also extremely challenging. As we present in Section~\ref{sec:measureofbifurcationpoints}, the bifurcation point set often admits only a stratified manifold structure (see Figure~\ref{fig:bifurcation points} for a concrete example). Only under stronger assumptions on the bifurcation point set, such as the fold bifurcation concept introduced in Section~\ref{sec:foldbifurcation}, does it enjoy better geometric properties, for example, a manifold structure.\\

\noindent\textbf{Challenge from practical perspective:} Even if we relax the lower-level solution requirement from the global minimizer to the commonly used local minimizer~\citep{liu2024moreau,bolte2025bilevel,lou2025decomposition}, evaluating the hyperfunction value near the bifurcation point set is practically intractable. Specifically, from a double loop algorithm perspective, any practical algorithm can only run the lower-level solver for a finite number of iterations, producing an approximate solution $\hat{y}$. As a result, we do not have access to the true hyperfunction value \( \phi(x) \), but only to an approximate evaluation
\[
\hat{\phi}(x) := f(x, \hat{y}).
\]
To control the error
$$|\hat{\phi}(x)-\phi(x)|=|f(x,\hat{y})-f(x,y^\ast(x))|,$$
one must first bound $\|\hat{y}-y^\ast(x)\|$. For a fixed $x$ that does not lie on the bifurcation point set, by the Morse property, the lower-level problem is locally strongly convex in $y$ at each local minimizer. In this case, a suitable algorithm, like gradient descent, can make $\|\hat{y} - y^\ast(x)\|$ arbitrarily small within a finite number of iterations. In particular, in a neighborhood where $g(x,\cdot)$ has a strong convexity constant $m(x) > 0$, standard gradient descent error bounds give
$$\|\hat{y}-y^\ast(x)\|\leq C(1-\alpha \cdot m(x))^K,$$
for some constants $C,\alpha$ and iteration number $K$. However, when $x$ is close to the bifurcation point set and $y^\ast(x)$ is about to degenerate, the local strong convexity modulus $m(x)$ of the minimizer becomes small. This slows down the convergence rate of gradient-type methods and greatly increases the number of iterations needed to reach a target accuracy. In practice, since one does not know in advance how close a given $x$ is to the bifurcation point set, it is impossible to determine $m(x)$, and hence to estimate the number of iterations required to achieve sufficient accuracy. Consequently, the hyperfunction value cannot be practically computed in any neighborhood of the bifurcation point set.

\section{Correspondence-Driven Hyperfunction and Its Smoothing}\label{sec:Correspondence_Driven_Reformulation_and_Its_Smoothing}

We introduce in Section~\ref{section:CorrespondenceDrivenHyperfunction} a new definition of the hyperfunction called \emph{correspondence-driven hyperfunction} based on algorithmic trajectories, which replaces the classical \emph{rational agent} assumption of the lower-level with a \emph{algorithm-based bounded rational agent}. Unlike the classical \emph{rational agent}, who is assumed to always attain a global minimizer even for nonconvex problems, an \emph{algorithm-based bounded rational agent} generates its decision directly through the execution of a prescribed algorithm.

This \emph{correspondence-driven hyperfunction} is then smoothed via Gaussian convolution in Section~\ref{section:smoothing}, yielding a well-behaved approximation that serves as the objective in our subsequent algorithm design. Section~\ref{section:smoothingproperty} establishes several basic properties of this smoothed function.

\subsection{Definition of the Correspondence-Driven Hyperfunction}\label{section:CorrespondenceDrivenHyperfunction}

To address the unsettledness (see Section~\ref{sec:illnessofhyperfunction}) of the hyperfunction definition associated with the original problem \eqref{eq:original_problem}, we introduce a novel \emph{correspondence-driven hyperfunction} in this section.

Instead of modeling the follower as a \emph{rational agent} who always attains a global minimizer, we consider the follower to be an \emph{algorithm-based bounded rational agent}. Specifically, the follower’s decision is generated by executing a prescribed optimization method $\mathcal{M}$ with step size schedule $\boldsymbol{\eta}$ from a fixed initialization $y_0$. For any given upper-level variable $x$, this procedure yields a sequence, and $\hat{S}(x, y_0, \boldsymbol{\eta}, \mathcal{M})$ denotes its accumulation points, from which the follower selects its decision. Regarding the algorithm $\mathcal{M}$ and step size schedule $\boldsymbol{\eta}$, we impose the following mild assumption, which can be satisfied by a wide range of standard optimization methods, such as the gradient descent method.

\begin{assumption}\label{assumption:method}
    The accumulation points of the sequence generated using method $\mathcal{M}$ with step size schedule $\boldsymbol{\eta}$ for initial point $y_0$ and any upper-level decision variable $x$ are stationary points with respect to $y$. That is, the set of all limit points of the sequence are a subset of the lower-level stationary point set.
\end{assumption}

This \emph{algorithm-based bounded rational} perspective captures realistic decision-making behavior. Under limited computational resources, the follower uses a prescribed optimization method $\mathcal{M}$ with step size schedule $\boldsymbol{\eta}$ from $y_0$ to explore the solution space and obtains the set of practically attainable solutions $\hat{S}(x, y_0, \boldsymbol{\eta}, \mathcal{M})$. By following this optimization method $\mathcal{M}$, the agent makes the best decision it can achieve, without assuming access to the full set of global minimizers as in the fully rational case.

We define the \emph{correspondence-driven hyperfunction} as follows:

\begin{definition}[Correspondence-driven Hyperfunction]
Given an upper-level decision variable $x \in \mathcal{X}$, a fixed follower initialization $y_0 \in \mathcal{Y}$, a step size schedule $\boldsymbol{\eta} = \{\eta_t\}$, and an optimization method $\mathcal{M}$ satisfying Assumption~\ref{assumption:method}, let $y^{\mathrm{cd}}(x, y_0, \boldsymbol{\eta}, \mathcal{M})$ be defined as follows:
\begin{align}\label{eq:ycd}
y^{\mathrm{cd}}(x, y_0, \boldsymbol{\eta}, \mathcal{M})
    := \arg\min_{y \in \hat{S}(x, y_0, \boldsymbol{\eta}, \mathcal{M})} f(x,y).
\end{align}
The \emph{correspondence-driven hyperfunction} is defined as
\begin{align}\label{eq:generalreformulation_single}
    \phi^{\mathrm{cd}}(x; y_0, \boldsymbol{\eta}, \mathcal{M}) := f\big(x, y^{\mathrm{cd}}(x, y_0, \boldsymbol{\eta}, \mathcal{M})\big).
\end{align}
\end{definition}

For notational simplicity, we denote $\phi^{\text{cd}}(x; y_0, \boldsymbol{\eta}, \mathcal{M})$, $y^{\text{cd}}(x, y_0, \boldsymbol{\eta}, \mathcal{M})$ and $\hat{S}(x, y_0, \boldsymbol{\eta}, \mathcal{M})$ simply as $\phi^{\text{cd}}(x)$, $y^{\text{cd}}(x)$ and $\hat{S}(x)$ in the remainder of this work.

This \emph{correspondence-driven hyperfunction} resolves the essential challenges raised in Section~\ref{sec:illnessofhyperfunction} from both the computational and modeling perspectives.  
From the computational perspective, replacing the global minimizer set $S(x)$ with the algorithmically attainable set $\hat{S}(x)$ makes the hyperfunction tractable, since $\hat{S}(x)$ can be approximated in principle by running method $\mathcal{M}$ with step size schedule $\boldsymbol{\eta}$ from initialization $y_0$.  
From the modeling perspective, the leader no longer assumes that the follower is a \emph{rational agent} who always achieves a global optimum, but instead treats the follower as an \emph{algorithm-based bounded rational agent} with limited computational capacity, leading to a more realistic modeling of the bilevel interaction.

\begin{remark}
More generally, one may model the follower’s response by averaging over a uniform distribution of initial points. Given a compact set $\mathcal{Y}_0 \subseteq \R^m$, one can define
\begin{align}\label{eq:generalreformulation}
    \phi^{\text{cd}}(x; \mathcal{Y}_0, \boldsymbol{\eta}, \mathcal{M}) := 
\frac{1}{\operatorname{vol}(\mathcal{Y}_0)} 
\int_{y_0 \in \mathcal{Y}_0} f(x, y^{\text{cd}}(x, y_0, \boldsymbol{\eta}, \mathcal{M})) \, dy_0.
\end{align}
This extension captures the average behavior of the follower under randomized initializations and reduces to \eqref{eq:generalreformulation_single} when $\mathcal{Y}_0 = \{y_0\}$. Our proposed approach \ouralgo{} (Algorithm~\ref{alg:main}) can be extended to this setting as well; however, for simplicity of analysis, we focus on the single-initialization case in this paper.
\end{remark}

\subsection{Smoothing via Mollification}\label{section:smoothing}

The hyperfunction $\phi^{\text{cd}}(x)$ defined in \eqref{eq:generalreformulation_single} remains challenging to optimize, as it can be discontinuous in $x$: small perturbations in $x$ may cause the algorithm to converge to different solutions, resulting in abrupt changes in the value of $\phi^{\text{cd}}$. We illustrate this behavior with the following example.

\begin{example}\label{example:discontiunous}
    Consider the lower-level problem 
    $$g(x,y)=(y-x)^4-2(y-x)^2,$$
    whose graph shifts left or right as $x$ varies. Figure \ref{fig:discontinuity} illustrates the function for $x=0$. We fix the initialization $y_0=0$ and apply gradient descent with step size smaller than $1/L$, where $L$ is the Lipschitz constant of $\nabla_y g(x,y)$ in the compact sublevel set $\{y:g(x,y)\leq g(x,0)\}$. In this setting, the algorithm converges to different local minimizers depending on the value of $x$: for $x<0$, the trajectory converges to the right minimizer; for $x>0$, it converges to the left minimizer; and for $x=0$, the algorithm stagnates at the maximizer $y^{\text{cd}}(0)=0$. As a result, $y^{\text{cd}}(x)$ is discontinuous at $x=0$, and consequently, $\phi^{\text{cd}}(x)=f(x,y^{\text{cd}}(x))$ is also discontinuous at that point unless the upper-level objective $f$ is carefully designed to cancel the jump.
\end{example}

\begin{figure}
    \centering
    \includegraphics[width=0.4\linewidth]{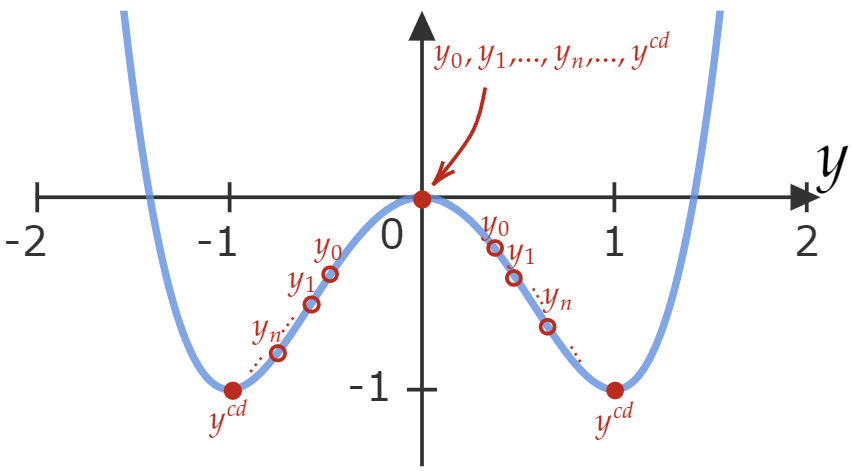}
    \caption{Illustration of the lower-level objective $g(x,y)=(y-x)^4-2(y-x)^2$ when $x=0$. The function has two symmetric minimizers and a saddle point at $y=0$. With fixed initialization $y_0=0$ and gradient descent, small changes in $x$ lead to different accumulation points, resulting in discontinuity in $y^{\text{cd}}(x)$ and hence in $\phi^{\text{cd}}(x)$.}
    \label{fig:discontinuity}
\end{figure}

The discontinuity of $\phi^{\text{cd}}(x)$ poses significant challenges for designing algorithms to optimize it directly. To address this, we consider a smoothed approximation of the hyperfunction, denoted by $\phi^{\text{cd}}_\xi(x)$, obtained by convolving $\phi^{\text{cd}}(x)$ with a smooth mollifier. Specifically, we use the Gaussian kernel
\begin{align}\label{eq:kernal}
    h_\xi(z)=\frac{1}{(2\pi\xi^2)^{n/2}}\exp(-\frac{\|z\|^2}{2\xi^2}),
\end{align}
and define the {\bf smoothed version of correspondence-driven hyperfunction}
\begin{align}\label{eq:smooth:hyper}
    \phi^{\text{cd}}_\xi(x):=(\phi^{\text{cd}}* h_\xi)(x)=\int_{\mathbb{R}^n}h_\xi(x-z)\phi^{\text{cd}}(z)dz.
\end{align}
Its gradient can be computed as
\begin{align}\label{eq:computeofgradieeent}
    \nabla_x\phi^{\text{cd}}_\xi(x)=\int_{\mathbb{R}^n}\nabla_xh_\xi(x-z)\phi^{\text{cd}}(z)dz.
\end{align}

This smoothing procedure produces a smooth approximation $\phi^{\text{cd}}_\xi(x)$, which approximates $\phi^{\text{cd}}(x)$ well in regions where $\phi^{\text{cd}}(x)$ is continuous. As illustrated in Figure \ref{fig:smoothing}, the smoothed function aligns well with the original function except near the discontinuity. Moreover, the minimizer of the smoothed function converges to $0$ from the left as the smoothing parameter $\xi$ tends to $0$.

\begin{figure}
    \centering
    \includegraphics[width=0.5\linewidth]{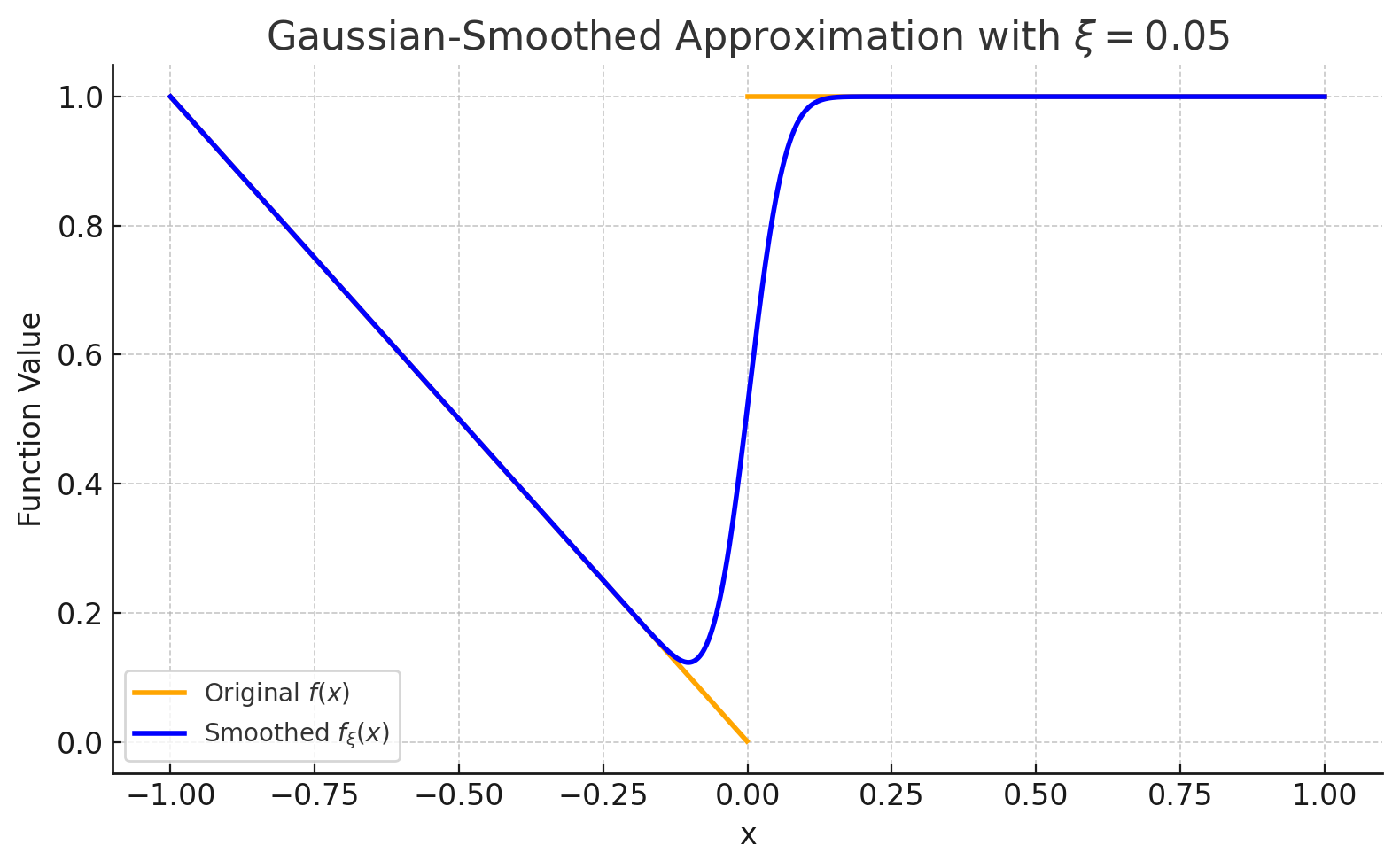}
    \caption{Gaussian smoothing of the piecewise-defined function $f(x)=-x$ for $x\leq 0$, and $f(x)=1$ for $x>0$, using a Gaussian kernel with $\xi=0.05$. The original function is discontinuous at $x=0$, while the smoothed approximation $f_\xi(x)$ is smooth and closely follows $f(x)$ away from the discontinuity. The minimizer of the smoothed function approximates the left-sided minimum of the original function.}
    \label{fig:smoothing}
\end{figure}

Before analyzing $\phi^{\text{cd}}_\xi(x)$ and designing the algorithm, we clarify two issues related to the well-definedness of the convolution. 

\begin{itemize}
    \item \textbf{The first issue} concerns the domain of definition: since the upper-level objective $f(x,y)$ is only defined for $x\in\mathcal{X}$, and hence the hyperfunction $\phi^{\text{cd}}(x)$ is only defined on $\mathcal{X}$. However, the convolution that defines $\phi^{\text{cd}}_\xi(x)$ formally integrates over the entire space $\mathbb{R}^n$, which requires a careful interpretation outside the domain $\mathcal{X}$;
    \item \textbf{The second issue} concerns the integrability of $\phi^{\text{cd}}(x)$: since this function is generally discontinuous, it is not immediately clear whether the convolution integral and its gradient are well-defined. In particular, the integrability of discontinuous functions requires careful analysis of the set of discontinuity points.
\end{itemize} 

To ensure that the convolution $\phi^{\text{cd}}_\xi(x)$ is well-defined and to provide a foundation for the theoretical developments that follow, we now introduce a set of assumptions.

\begin{assumption}\label{assumption:general1}
    The following hold:
    \begin{enumerate}
        \item $f(x,y)$ is continuous and $g(x,y)$ is twice continuously differentiable for any $x\in\mathcal{X}$ and $y\in\mathbb{R}^m$;\label{assumption:general1(1)}
        \item $\mathcal{X}$ is convex and compact;\label{assumption:general1(2)}
        \item $g(x,y)$ is Lipschitz smooth with constant $\overline{L}_g$ on $\mathcal{X}\times\R^m$;\label{assumption:general1(4)}
        \item $\nabla^2_{yy} g(x,y)$ is Lipschitz continuous with constant $\overline{\overline{L}}_g$ on $\mathcal{X}\times\R^m$;\label{assumption:general1(11)}
        \item $g(x,y)$ is lower bounded by $\underline{g}$ for any $x\in\mathcal{X}$ and $y\in\mathbb{R}^m$;\label{assumption:general1(5)}
        \item $\{y:g(x,y)\leq g(x,y_0)\}$ is compact and contained in a compact set $\mathcal{Y}\subset\R^m$ for any $x\in\mathcal{X}$;\label{assumption:general1(7)}
        \item $f(x,y)$ is Lipschitz continuous with constant $L_f$ on $\mathcal{X}\times\{y:g(x,y)\leq g(x,y_0)\}$.\label{assumption:general1(10)}
    \end{enumerate}
\end{assumption}
Assumptions \ref{assumption:general1}(\ref{assumption:general1(1)})-(\ref{assumption:general1(5)}) and (\ref{assumption:general1(10)}) are standard assumptions in the BLO literature. Assumption \ref{assumption:general1}(\ref{assumption:general1(7)}) is generally assumed in much of single-level optimization literature. Then we can directly obtain the following bounds by Assumptions~\ref{assumption:general1}(\ref{assumption:general1(2)})(\ref{assumption:general1(7)}), which state the compactness of $\mathcal{X}$ and the uniform compactness of the level sets $\{y : g(x, y) \leq g(x,y_0)\}$, respectively.

\begin{proposition}\label{prop:general} Suppose Assumption~\ref{assumption:general1} holds. There exist constants $\overline{g}_0$, $L_f$, and $\overline{f}$ such that the following hold
    \begin{enumerate} 
        \item $g(x,y_0)$ is upper bounded by $\overline{g}_0$ for any $x\in\mathcal{X}$, where $y_0$ is the initialization of the lower-level problem;\label{assumption:general1(6)}
        \item $|f(x,y)|$ is upper bounded by $\overline{f}$ for any $x\in \mathcal{X}$ and $y\in\{y:g(x,y)\leq g(x,y_0)\}$\label{assumption:general1(8)}.
    \end{enumerate}
\end{proposition}

About the lower-level method $\mathcal{M}$, step size schedule $\boldsymbol{\eta}$, and the initialization $y_0 \in \mathcal{Y}$, we need the following assumption:

\begin{assumption}\label{assumption:descent}
The lower-level optimization method $\mathcal{M}$, step size schedule $\boldsymbol{\eta}$, and initialization $y_0 \in \mathcal{Y}$ satisfy the following property: there exists an integer $K_0$ such that for any $x \in \mathcal{X}$, the sequence $\{y_k\}$ generated by $\mathcal{M}$ satisfies
$$g(x,y_k)\leq g(x,y_0)$$
for any $k\geq K_0$. Therefore, the following also holds
\begin{align}\label{eq:levelset}
    g(x, y^{\text{cd}}(x)) \leq g(x, y_0),
\end{align}
where $y^{\text{cd}}(x)$ is defined in \eqref{eq:ycd}. In the following, we always choose the lower-level iteration number $K \geq K_0$.
\end{assumption}

Assumption~\ref{assumption:descent} is mild and holds for a wide range of commonly used optimization methods (e.g., gradient descent, accelerated methods, or second-order methods) with properly selected step sizes. The relation \eqref{eq:levelset} implies that the output $y^{\text{cd}}(x)$ lies in the sublevel set
$\{y:g(x,y)\leq g(x,y_0))\}$. Combining this with Proposition \ref{prop:general}(\ref{assumption:general1(8)}), which ensures that $f(x,y)\leq \overline{f}$ for all $y$ in this sublevel set, we obtain the uniform bound
\begin{align}\label{eq:boundoff}
    \phi^{\text{cd}}(x)=f(x,y^{\text{cd}}(x))\leq \overline{f},\text{ for all }x\in\mathcal{X}.
\end{align}
This allows us to resolve \textbf{the first issue}, namely the mismatch between the domain of $\phi^{\text{cd}}(x)$, which is defined only on $\mathcal{X}$, and the domain required by the convolution, which integrates over the entire space $\mathbb{R}^n$. Specifically, we extend $\phi^{\text{cd}}(x)$ outside $\mathcal{X}$ by assigning it $\overline{f}$ (or equivalently modify the value of $f$ outside $\mathcal{X}$ to be $\overline{f}$), i.e., define the following extension:
\begin{align}\label{eq:extension}
    \widetilde{\phi}^{\text{cd}}(x)=\begin{cases}
        \phi^{\text{cd}}(x),\quad&\text{if}\ x\in\mathcal{X}\\
        \overline{f},\quad&\text{if}\ x\notin\mathcal{X}.
    \end{cases}
\end{align}
Under this extension, solving $\phi^{\text{cd}}(x)$ over $\mathcal{X}$ becomes equivalent to solving extended $\widetilde{\phi}^{\text{cd}}(x)$ over the whole space $\mathbb{R}^n$, since the minimizer must lie within $\mathcal{X}$. For notational simplicity, we continue to denote the extended function $\widetilde{\phi}^{\text{cd}}(x)$ by $\phi^{\text{cd}}(x)$.

To address \textbf{the second issue}, namely the integrability of $\phi^{\mathrm{cd}}(x)$, we partition $\mathcal{X}$ into the set of bifurcation points $\widetilde{\mathcal{X}}$ defined in \eqref{def:bifurcationpointset} and its complement $\mathcal{X}\setminus\widetilde{\mathcal{X}}$, and impose mild assumptions on each set separately. Leveraging the result that the discontinuity points of $y^{\mathrm{cd}}(x)$ within both sets have measure zero, we show that the integral is well-defined. This is motivated by a classical result in real analysis: the integral of a bounded function is well-defined as a Riemann integral if and only if its set of discontinuities has Lebesgue measure zero.

For the set $\widetilde{\mathcal{X}}$, we use Assumption~\ref{assumption:nowhere_dense}, which states that the bifurcation point set $\widetilde{\mathcal{X}}$ defined in \eqref{def:bifurcationpointset} has measure zero. For the remaining set $\mathcal{X} \setminus \widetilde{\mathcal{X}}$, we make the following assumption:

\begin{assumption}\label{assumption:discontinuous_point}
    On $\mathcal{X}\setminus\widetilde{\mathcal{X}}$, the set of points where $y^{\text{cd}}(x)$ is discontinuous has measure zero.
\end{assumption}

Assumption~\ref{assumption:discontinuous_point} is very mild. We show in Section~\ref{sec:MeasureZeroDiscontinuities} that this property holds in a representative setting where the solution mapping $y^{\text{cd}}(x)$ is generated by gradient descent with a suitable step size. From a practical perspective, small changes in $x$ can be seen as perturbations to the problem parameters. It is rarely observed that such perturbations cause abrupt changes in the behavior of solution trajectories produced by many algorithms (such as gradient methods and second-order methods). Therefore, it is reasonable to assume that the discontinuity points of $y^{\text{cd}}(x)$ are of measure zero.

Under Assumption \ref{assumption:nowhere_dense} and \ref{assumption:discontinuous_point}, $y^{\text{cd}}(x)$ is discontinuous only on a set of measure zero in $\mathcal{X}$. Since $\phi^{\text{cd}}(x)=f(x,y^{\text{cd}}(x))$ and the function $f$ is continuous, it follows that $\phi^{\text{cd}}(x)$ is also discontinuous only on a measure-zero subset of $\mathcal{X}$. Note that $\mathcal{X}$ is a closed convex set, and thus its boundary $\partial\mathcal{X}$ has Lebesgue measure zero~\cite[Theorem 1]{lang1986note}. The extension of $\phi^{\text{cd}}(x)$, i.e., setting $\phi^{\text{cd}}(x) = \overline{f}$ for $x \notin \mathcal{X}$, introduces new discontinuities only on a subset of $\partial\mathcal{X}$, since the function takes different values inside and outside $\mathcal{X}$. Therefore, the set of discontinuities of the extended $\phi^{\text{cd}}(x)$ on $\mathbb{R}^n$ remains a measure-zero set. By the classical criterion for Riemann integrability, it follows that $\phi^{\text{cd}}(x)$ is Riemann integrable on $\mathcal{X}$. Therefore, the second technical issue is resolved. In conclusion, the convolution $\phi^{\text{cd}}_\xi(x)$ is well-defined.

\begin{table}[h]
\centering
\renewcommand{\arraystretch}{1.4}
\begin{tabular}{|c|c|c|c|}
\hline
 & $\smash{\widetilde{\mathcal{X}}}$ & $\mathcal{X} \setminus \widetilde{\mathcal{X}}$ & $\mathbb{R}^n \setminus \mathcal{X}$ \\
\hline
Measure & zero & nonzero & nonzero \\
\hline
Continuity of $\phi^{\text{cd}}(x)$ & unknown & almost everywhere continuous & everywhere continuous \\
\hline
\end{tabular}
\caption{Measure and continuity properties of $\phi^{\text{cd}}(x)$ across different regions.}
\label{tab:cdcontinuity}
\end{table}

Having established that the smoothed hyperfunction $\pcd_\xi(x)$ is well-defined, we next discuss its properties in the following section.

\subsection{Properties of the Smoothed Hyperfunction \texorpdfstring{\eqref{eq:smooth:hyper}}{(Equation~\ref{eq:smooth:hyper})}}\label{section:smoothingproperty}

We now present several standard properties of the smoothed hyperfunction $\pcd_\xi(x)$, defined as the convolution of the function $\pcd(x)$ with the Gaussian kernel $h_\xi$.

\begin{proposition}\label{proposition:limit}
    Suppose Assumption~\ref{assumption:general1} holds. Then $\phi^{\text{cd}}_\xi(x)$ is smooth for any $\xi$. Furthermore, if $\phi^{\text{cd}}(x)$ is continuous at point $\bar{x}$, then we have the convergence of the function value $\lim_{\xi\to 0} \phi^{\text{cd}}_\xi(\bar{x})=\phi^{\text{cd}}(\bar{x})$. If $\pcd(x)$ is differentiable at point $\bar{x}$, then we have the convergence of gradient $\lim_{\xi\to 0} \nabla_x\phi^{\text{cd}}_\xi(\bar{x})=\nabla_x\phi^{\text{cd}}(\bar{x})$, and thus $\lim_{\xi\to 0}P_{\mathcal{X}}(x,\nabla_x\phi^{\text{cd}}_\xi,\beta)=P_{\mathcal{X}}(x,\nabla_x\phi^{\text{cd}},\beta)$ which are defined as follows:
    $$P_{\mathcal{X}}(x,\nabla_x\phi^{\text{cd}}_\xi,\beta):=\frac{x-\proj_\mathcal{X}(x-\beta\nabla_x\phi^{\text{cd}}_\xi)}{\beta},$$
    $$P_{\mathcal{X}}(x,\nabla_x\phi^{\text{cd}},\beta):=\frac{x-\proj_\mathcal{X}(x-\beta\nabla_x\phi^{\text{cd}})}{\beta}.$$
\end{proposition}

\begin{proof}
    See Appendix~\ref{proof:proposition:limit}.
\end{proof}

Proposition~\ref{proposition:limit} shows that Gaussian smoothing gives $\pcd_\xi(x)$ good regularity: it is smooth for all $\xi>0$, and it recovers both the function value and (proximal) gradient of $\pcd$ in the limit as $\xi$ approaches $0$, at any point where $\pcd$ is continuous or differentiable. This result ensures that the projected gradient-based methods applied to the smoothed problem can approximate the behavior of the original hyperfunction. The proof follows the classical theory of mollification in PDE~\citep{evans2022partial}, where a discontinuous or nonsmooth function is approximated by the convolution with a smooth kernel. In classical mollifier theory, a compactly supported smooth bump function (e.g., supported on the unit ball) is often used. Pointwise and gradient convergence are established using local continuity or differentiability inside the support, while the contribution from outside the support is exactly zero. In our case, the smoothing kernel is the Gaussian $h_\xi$, which is not compactly supported. Nevertheless, the key idea remains the same: most of the kernel mass is concentrated around the origin, and the contribution from the tail decays exponentially fast. This exponential decay plays the same role as compact support in standard mollifier arguments. Thus, the structure of the proof is essentially identical to the standard bump function case.

We next establish that $\pcd_\xi(x)$ has bounded gradient and is Lipschitz smooth. These properties are crucial for the convergence analysis in Theorem~\ref{thm:biased_sgd_conv}.

\begin{proposition}\label{proposition:boundedgradient}
    Suppose Assumption~\ref{assumption:general1} holds. We have the following bound on the gradient of $\phi^{\text{cd}}_\xi(x)$:
    $$\|\nabla_x\phi^{\text{cd}}_\xi(x)\|\leq\sqrt{\frac{2}{\pi}}\times \frac{\overline{f}}{\xi}.$$
\end{proposition}

\begin{proof}
    See Appendix~\ref{proof:proposition:boundedgradient}.
\end{proof}

\begin{proposition}\label{proposition:lipschitz}
    Suppose Assumption~\ref{assumption:general1} holds. $\phi^{\text{cd}}_\xi(x)$ is Lipschitz smooth with constant $\overline{L}_{\phi^{\text{cd}}_{\xi}}=\overline{f}/\xi^2$.
\end{proposition}

\begin{proof}
    See Appendix~\ref{proof:proposition:lipschitz}.
\end{proof}

Given the Lipschitz smoothness of $\phi^{\text{cd}}_\xi(x)$ and the explicit gradient expression in \eqref{eq:computeofgradieeent}, we can now employ gradient-based methods to optimize it.

\section{Algorithm and Convergence Results}\label{section:Algorithm}

We now turn to the algorithmic component of this work. As emphasized in Section~\ref{sec:Correspondence_Driven_Reformulation_and_Its_Smoothing}, our goal is to minimize the smoothed hyperfunction $\pcd_\xi(x)$ for a fixed smoothing parameter $\xi>0$. We consider the case that the follower, to obtain higher-quality solutions, adopts the cubic-regularized Newton method~\citep{nesterov2006cubic} as the lower-level method $\mathcal{M}$. Other second-order approaches, such as trust-region methods~\citep{conn2000trust,jiang2023beyond} or the homogeneous second-order methods~\citep{he2025homogeneous,zhang2025homogeneous}, can be analyzed in a similar manner.

Although, by the extension in \eqref{eq:extension}, the minimum of $\pcd$ is attained within $\mathcal{X}$, and Proposition~\ref{proposition:limit} ensures that, for small $\xi$, the minimum of $\pcd_\xi$ is also attained in $\mathcal{X}$, applying an unconstrained first-order method directly to $\pcd_\xi$ may still be problematic. Indeed, the extension renders $\pcd$ constant outside $\mathcal{X}$, making its gradient zero there, and Proposition~\ref{proposition:limit} further implies that, when $\xi$ is small, the gradient of $\pcd_\xi$ is also very small outside $\mathcal{X}$. As a result, an unconstrained method initialized outside $\mathcal{X}$ may terminate prematurely. For this reason, we consider the following constrained problem:
\begin{align}
    \min_{x\in\mathcal{X}}\ \pcd_\xi(x):=\int_{\mathbb{R}^n}h_\xi(x-z)\phi^{\text{cd}}(z)dz
\end{align}
where $\phi^{\text{cd}}$ is from \eqref{eq:generalreformulation_single} and $h_\xi$ is the parametrized Gauss kernel defined in \eqref{eq:kernal}.
\subsection{Sketch of the Proposed Algorithm}
Since $\pcd_\xi(x)$ is smooth (see Proposition~\ref{proposition:lipschitz}), we adopt a standard projected gradient-based method (see Algorithm \ref{alg:main}) to solve the problem. In principle, the exact gradient of $\pcd_\xi(x)$ takes the form
\begin{align}\label{eq:exactgradient}       
    \nabla_x\phi^{\text{cd}}_\xi(x)&=\int_{\mathbb{R}^n}\nabla_xh_\xi(x-z)\phi^{\text{cd}}(z)dz\\
    &=\int_{\mathbb{R}^n}\nabla_xh_\xi(x-z)f(z,y^{\text{cd}}(z))dz\nonumber\\
    &=\int_{\mathbb{R}^n}-\frac{x-z}{\xi^2}h_\xi(x-z)f(z,y^{\text{cd}}(z))dz\nonumber\\
    &=\frac{1}{\xi}\E_{u\sim\mathcal{N}(0,I_n)}[uf(x+\xi u,y^{\text{cd}}(x+\xi u))].\nonumber
\end{align}
Note that, as established in \eqref{eq:extension}, we have extended $\pcd$ outside $\mathcal{X}$, which is equivalently interpreted as redefining $f$ to take the constant value $\overline{f}$ outside $\mathcal{X}$, where $\overline{f}$ is from Proposition~\ref{prop:general}(\ref{assumption:general1(8)}). Thus, in the following, we work with $f$ under this redefinition.

In practice, we estimate this gradient via Monte Carlo sampling, which leads to a biased stochastic gradient descent (\texttt{SGD}) algorithm. Specifically, we draw i.i.d. sample points $u^{(1)},...,u^{(N)}\sim\mathcal{N}(0,I_n)$, and for each such sample $u^{(i)}$, if the perturbed point $x+\xi u^{(i)}\notin\mathcal{X}$, we directly assign the value $f(x+\xi u^{(i)},\hat{y}(x+\xi u^{(i)}))=\overline{f}$; otherwise, when $x+\xi u^{(i)}\in\mathcal{X}$, we run a fixed number $K$ of iterations of the lower-level algorithm $\mathcal{M}$ at $x+\xi u^{(i)}$ to obtain an approximate solution $\hat{y}(x+\xi u^{(i)})$, where $K$ is independent of $x$. We then form the following average
\begin{align}\label{eq:estimator}
    \hat{\nabla}^K_x\phi^{\text{cd}}_{\xi }(x):=\frac{1}{N\xi}\sum_{i=1}^Nu^{(i)}f(x+\xi u^{(i)},\hat{y}(x+\xi u^{(i)}))
\end{align}
as a Monte Carlo gradient estimator of the following inexact gradient
\begin{align}\label{eq:inexactgradient}
\nabla^K_x\phi^{\text{cd}}_\xi(x):=\int_{\mathbb{R}^n}\nabla_xh_\xi(x-z)f(z,\hat{y}(z))dz.
\end{align}
Since each $\hat{y}(x+\xi u^{(i)})$ only approximates the true correspondence‐driven point $y^{\text{cd}}(x+\xi u^{(i)})$, this estimator carries a bias. 

This smoothing-and-sampling technique is closely related to the zeroth-order methods studied in \cite{nesterov2017random}, which focus on optimizing continuous or smooth functions via randomized function evaluations. In contrast, our formulation applies this idea to a discontinuous function since $\pcd$ may be discontinuous (see Example~\ref{example:discontiunous}). 

\begin{algorithm*}[htbp]
\caption{\ouralgo{}\ (Smooth Correspondence-driven Nonconvex lower-level Bilevel Optimization)}
\label{alg:main}
\KwIn{initial point $\bar{x}$; total iterations $T$; step sizes $\{\beta_t\}$; smoothing parameter $\xi $; number of samples $\{N_t\}$; number of inner steps $\{K_t\}$;}
\For{$t \gets 0$ \KwTo $T-1$}{
    Sample $u^{(1)}, \dots, u^{(N_t)} \sim \mathcal{N}(0, I_n)$\;
    \For{$i \gets 1$ \KwTo $N_t$}{
    $\tilde{x}^{(i)} \gets x_t + \xi \, u^{(i)}$\;
    \uIf{$\tilde{x}^{(i)} \in \mathcal{X}$}{
            Run inner algorithm $\mathcal{M}$ on $g(\tilde{x}^{(i)},\cdot)$ for $K_t$ steps, initialized at $y_0$, with step size $\boldsymbol{\eta}$ to obtain $\hat{y}^{(i)}$\;
            $f^{(i)} \gets f(\tilde{x}^{(i)},\, \hat{y}^{(i)})$\;
        }
        \Else{
            $f^{(i)} \gets \overline{f}$\;
        }
}
    Compute the gradient estimator $\hat{\nabla}_x^{K_t} \pcd_{\xi }(x_t)$ given by \eqref{eq:estimator} and update $x_{t+1}$:\ 
    $\displaystyle \hat{\nabla}_x^{K_t} \pcd_{\xi }(x_t)
    \gets \frac{1}{N_t\,\xi } \sum_{i=1}^{N_t} u^{(i)} f^{(i)}$\;
    $x_{t+1} \gets \proj_{\mathcal{X}}(x_t - \beta_t\, \hat{\nabla}_x^{K_t} \pcd_{\xi }(x_t))$\ 
}
\end{algorithm*}

While the proposed algorithm provides a practical way to estimate the gradient of the smoothed objective $\phi^{\text{cd}}_\xi(x)$ via Monte Carlo sampling, the resulting estimator \eqref{eq:estimator} approximates an inexact gradient \eqref{eq:inexactgradient} that is itself biased relative to the true gradient. This inexactness arises from using approximate lower-level solutions $\hat{y}(x)$ instead of $y^{\text{cd}}(x)$. The resulting discrepancy between the exact and inexact gradients can be quantified by the following expression:

\begin{align}\label{eq:splitofintegral}
\nabla_x\phi^{\text{cd}}_\xi(x)-\nabla^K_x\phi^{\text{cd}}_\xi(x)=\int_{\R^n}\nabla_x h_\xi(x-z)[f(z,y^{\text{cd}}(z))-f(z,\hat{y}(z))]dz.
\end{align}

As stated in Section~\ref{sec:genericassumption}, a central difficulty arises when $x$ lies near the bifurcation point set $\widetilde{\mathcal{X}}$ defined in \eqref{def:bifurcationpointset}: with a fixed number of lower-level iterations $K$, it is generally impossible to ensure that the approximate solution $\hat{y}(x)$ is close to the correspondence-driven solution $y^{\text{cd}}(x)$. Consequently, for any fixed $K$, the integrand in equation~\eqref{eq:splitofintegral} cannot be uniformly small in absolute value across all $z$. To proceed further, we define the $\delta$-neighborhood of the bifurcation point set as
\begin{align}\label{eq:nbhdofbifurcationset}
    \widetilde{\mathcal{X}}_{\delta} := \{x \in \mathcal{X} : \mathrm{dist}(x, \widetilde{\mathcal{X}}) \leq \delta\}.
\end{align}
We then split the integral in \eqref{eq:splitofintegral} as follows:
\begin{align}\label{eq:splitofintegral2}
    \int_{\widetilde{\mathcal{X}}_\delta} \nabla_x h_\xi(x-z)[f(z, y^{\text{cd}}(z)) - f(z, \hat{y}(z))]\,dz 
    + \int_{\mathbb{R}^n \setminus \widetilde{\mathcal{X}}_\delta} \nabla_x h_\xi(x-z)[f(z, y^{\text{cd}}(z)) - f(z, \hat{y}(z))]\,dz.
\end{align}

The first term in \eqref{eq:splitofintegral2} is particularly difficult to estimate, since within $\widetilde{\mathcal{X}}_\delta$ we cannot reliably control the gap $\|\hat{y}(z) - y^{\text{cd}}(z)\|$. This is because the Hessian of the lower-level objective at $y^{\text{cd}}(z)$ may be degenerate, so even though the algorithm returns a point $\hat{y}(z)$ with a small gradient norm, it may still be far from the true solution. Our approach to addressing this challenge relies on the absolute continuity of the integral: as long as the measure of $\widetilde{\mathcal{X}}_\delta$ is sufficiently small and the integrand is bounded, the effect of the first term in \eqref{eq:splitofintegral2} can be made arbitrarily small.
Hence, it is not necessary to compute highly accurate solutions on $\widetilde{\mathcal{X}}_\delta$. By Proposition~\ref{assumption:general1}(\ref{assumption:general1(8)}), together with Assumption~\ref{assumption:descent}, it follows immediately that both $|f(z, y^{\text{cd}}(z))|$ and $|f(z, \hat{y}(z))|$ are bounded. Moreover, by the design of $h_\xi$, the integrand in the first term of \eqref{eq:splitofintegral2} is itself bounded. Therefore, the core remaining challenge is to estimate the measure of $\widetilde{\mathcal{X}}_\delta$. In Section~\ref{sec:measureofbifurcationpoints}, we will provide an upper bound on the measure of $\widetilde{\mathcal{X}}_\delta$.

To estimate the second term in \eqref{eq:splitofintegral2}, we will analyze how many lower-level iterations $K$ are required to ensure that $\|y^{\text{cd}}(x) - \hat{y}(x)\| \leq \rho$ for all $x \in \mathbb{R}^n \setminus \widetilde{\mathcal{X}}_\delta$ in Section~\ref{sec:Outer-loop Iteration Complexity Analysis}. 

With these two results, we can then view Algorithm~\ref{alg:main} as a biased projected \texttt{SGD} method, and will subsequently establish an upper-level problem oracle complexity bound for \ouralgo{}.

Finally, to obtain the lower-level problem oracle complexity of \ouralgo{}, we introduce the notion of fold bifurcation Section~\ref{sec:foldbifurcation} to endow the bifurcation points with additional structure. This also leads to a clearer geometric picture of the bifurcation set.

\subsection{Geometric Analysis near Bifurcation Points}\label{sec:measureofbifurcationpoints}

We now analyze the measure of $\widetilde{\mathcal{X}}_\delta$.  Note that under Assumption~\ref{assumption:nowhere_dense}, we have already assumed that the bifurcation point set $\widetilde{\mathcal{X}}$ defined in \eqref{def:bifurcationpointset} has Lebesgue measure zero. Moreover, since $\widetilde{\mathcal{X}}$ is closed by Remark~\ref{remark:close} and bounded as a subset of $\mathcal{X}$ by 
Assumption~\ref{assumption:general1}(\ref{assumption:general1(2)}), we know that $\widetilde{\mathcal{X}}$ is also compact. It is provable that the $\delta$-neighborhood of a compact, measure-zero set has vanishing measure as $\delta$ tends to $0$. To obtain more precise estimates on the rate at which this measure vanishes, we introduce the following definitions and assumption from geometric measure theory.

\begin{definition}[Covering number]
    Suppose $\mathcal{X}$ is a compact set of $\mathbb{R}^n$, the covering number $N(\mathcal{X},r)$ is the number of balls of radius $r$ required to cover $\mathcal{X}$.
\end{definition}

\begin{definition}[Minkowski dimension]
    Suppose $\mathcal{X}$ is a compact set of $\mathbb{R}^n$ with covering number $N(\mathcal{X},r)$. Then the upper and lower Minkowski dimension are defined as follows respectively:
    $$\overline{\mathrm{dim}}_{\mathrm{box}}(\mathcal{X}):=\limsup_{r\to 0}\frac{\log N(\mathcal{X},r)}{-\log(r)},$$
    $$\underline{\mathrm{dim}}_{\mathrm{box}}(\mathcal{X}):=\liminf_{r\to 0}\frac{\log N(\mathcal{X},r)}{-\log(r)}.$$
    If the limit exists (i.e., the limsup equals the liminf), we call it the Minkowski dimension of $\mathcal{X}$ and write
    $$\mathrm{dim}_{\mathrm{box}}(\mathcal{X}):=\lim_{r\to 0}\frac{\log N(\mathcal{X},r)}{-\log(r)}.$$
\end{definition}

Intuitively, the Minkowski dimension reflects how the number of small balls needed to cover the set grows as the ball radius shrinks.

\begin{assumption}\label{assumption:Minkowski_dimension}
    The upper Minkowski dimension of the bifurcation point set $\widetilde{\mathcal{X}}$ defined in \eqref{def:bifurcationpointset} is $d:=\overline{\mathrm{dim}}_{\mathrm{box}}(\widetilde{\mathcal{X}})<n$.
\end{assumption}

This assumption is mild: it is satisfied by all semi-algebraic functions $g(x,y)$ under Assumption~\ref{assumption:nowhere_dense}, as guaranteed by the following theorem .

\begin{theorem}\label{theorem:poly_Minkowski}
    Suppose $g(x,y)$ is a semi-algebraic function for $(x,y)$, and Assumption~\ref{assumption:nowhere_dense} holds, then the upper Minkowski dimension $d$ of $\widetilde{\mathcal{X}}$ is less than or equal to $n-1$.
\end{theorem}

\begin{proof}
    See Appendix~\ref{proof:theorem:poly_Minkowski}.
\end{proof}

The core of the proof of Theorem~\ref{theorem:poly_Minkowski} is that the bifurcation set $\widetilde{\mathcal{X}}$ corresponding to a semi-algebraic function $g(x,y)$ is itself a semi-algebraic set, and thus admits a stratified manifold structure. Since $\widetilde{\mathcal{X}}$ has Lebesgue measure zero, all manifolds in the stratification must have dimension at most $n-1$, which implies that the Minkowski dimension of $\widetilde{\mathcal{X}}$ is at most $n-1$. A concrete example of such a stratified manifold structure for $\widetilde{\mathcal{X}}$ is illustrated in Figure~\ref{fig:bifurcation points}.

\begin{figure}
        \centering
        \includegraphics[width=0.4\linewidth]{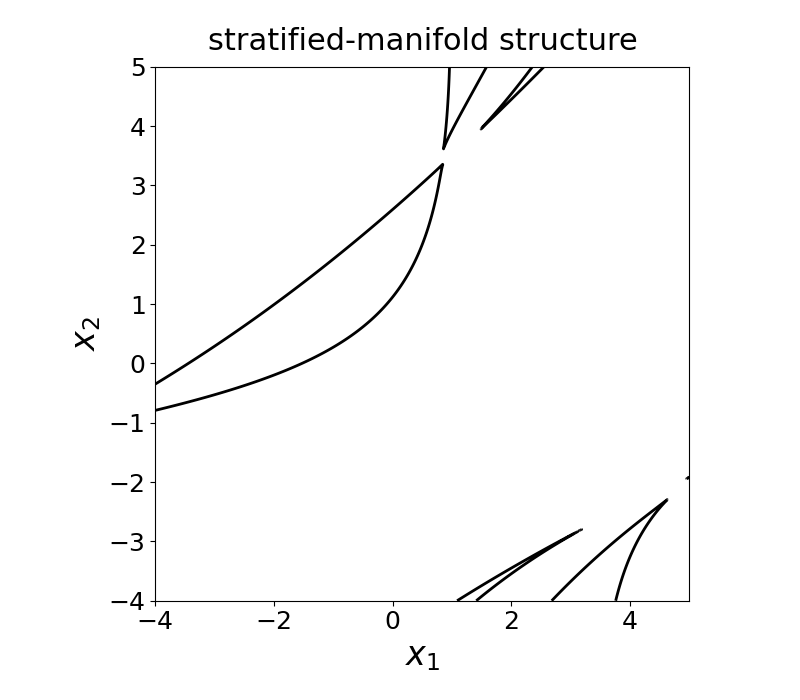}
        \caption{Visualization of the bifurcation point set $\widetilde{\mathcal{X}}$ of the function $g(x,y)=y^4+(x_1^2 - 5x_1x_2 + 2x_2^2 - 7x_1 + 8x_2 - 30)y^3
        +(x_1^2 - 3x_1x_2 + 4x_2^2 - 5x_1 + 2x_2 - 40)y^2
        +(x_1^2 - 5x_1x_2 + 2x_2^2 - 7x_1 + 8x_2 - 30)y.$ over the domain $[-4,5]^2$. The set exhibits a finite stratified manifold structure.}
        \label{fig:bifurcation points}
    \end{figure}

With the Minkowski dimension established, we can now estimate the measure of the $\delta$-neighborhood of the bifurcation set $\widetilde{\mathcal{X}}_\delta$ defined in \eqref{eq:nbhdofbifurcationset}. Let $\lambda(\delta)$ denote the Lebesgue measure of $\widetilde{\mathcal{X}}_\delta$. The goal is to bound $\lambda(\delta)$ as a function of $\delta$, using the Minkowski dimension provided by Assumption~\ref{assumption:Minkowski_dimension}.

\begin{lemma}\label{lem:measureofneighborhood}
    Suppose Assumption~\ref{assumption:Minkowski_dimension} holds. Then there exists a constant $C$ such that the measure of $\widetilde{\mathcal{X}}_\delta$ satisfies
    \begin{align}\label{equation:volume_of_balls}
        \lambda(\delta)\leq C\delta^{(n-d)/2}.
    \end{align}
\end{lemma}

\begin{proof}
    See Appendix~\ref{proof:lem:measureofneighborhood}.
\end{proof}

The core idea of the proof of Lemma~\ref{lem:measureofneighborhood} is illustrated in Figure \ref{figure:covering}. The black curve represents the bifurcation set $\widetilde{\mathcal{X}}$, while the blue curve marks the boundary of its $\delta$-neighborhood $\widetilde{\mathcal{X}}_\delta$. We begin by covering $\widetilde{\mathcal{X}}$ with $N(\widetilde{\mathcal{X}},\delta)$ balls of radius $\delta$, whose boundaries are shown as dashed circles. Then, around the same centers, we construct larger balls of radius $2\delta$, shown in red. It can be shown that these enlarged balls collectively cover the entire neighborhood $\widetilde{\mathcal{X}}_\delta$. Finally, applying the definition of the upper Minkowski dimension yields an upper bound on the covering number $N(\widetilde{\mathcal{X}}, \delta)$, from which we obtain the estimate for $\lambda(\delta)$, that is, the area of the blue tubular region. 

\begin{remark}
    In Lemma~\ref{lem:measureofneighborhood}, letting $\delta\to 0$, we obtain that the measure of $\widetilde{\mathcal{X}}$ is zero, which is precisely Assumption~\ref{assumption:nowhere_dense}. Hence Assumption~\ref{assumption:Minkowski_dimension} actually implies Assumption~\ref{assumption:nowhere_dense}. Moreover, Theorem~\ref{theorem:poly_Minkowski} shows that, when $g$ is semi-algebraic in $(x,y)$, Assumption~\ref{assumption:Minkowski_dimension} and Assumption~\ref{assumption:nowhere_dense} are equivalent.
\end{remark}

The estimate on the measure of $\widetilde{\mathcal{X}}_\delta$ allows us to control the first term in \eqref{eq:splitofintegral2}, which is essential for establishing Proposition~\ref{proposition:gradient_error}. As a key intermediate result, this proposition is indispensable for the proof of Theorem~\ref{thm:biased_sgd_conv}, which establishes the complexity of Algorithm~\eqref{alg:main} for the upper-level problem.

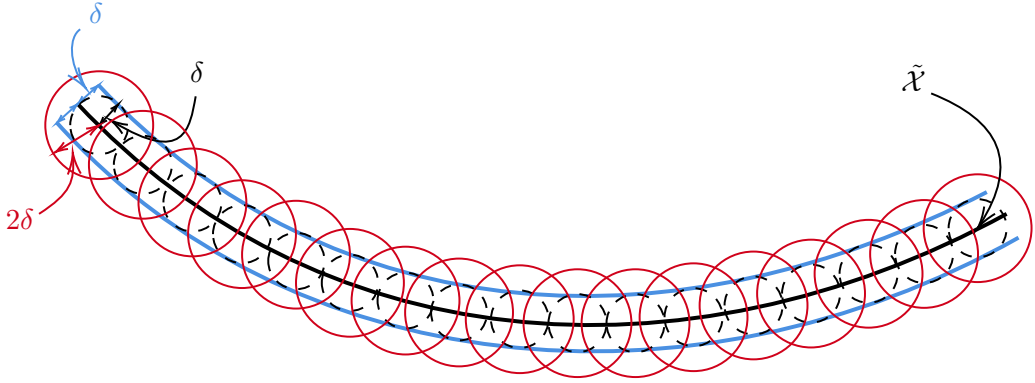
\begin{figure}
    \centering
    \tikzset{every picture/.style={line width=0.75pt}} 

\begin{tikzpicture}[x=0.75pt,y=0.75pt,yscale=-1,xscale=1]

\draw [line width=1.5]    (61.77,89.9) .. controls (181.03,217.3) and (365.27,235.4) .. (530.27,145.4) ;
\draw [color={rgb, 255:red, 74; green, 144; blue, 226 }  ,draw opacity=1 ][line width=1.5]    (72.27,80.4) .. controls (191.53,207.8) and (365.27,214.4) .. (520.27,134.4) ;
\draw [color={rgb, 255:red, 74; green, 144; blue, 226 }  ,draw opacity=1 ][line width=1.5]    (51.27,99.4) .. controls (167.27,228.4) and (369.27,251.4) .. (536.27,157.4) ;
\draw [color={rgb, 255:red, 74; green, 144; blue, 226 }  ,draw opacity=1 ]   (70.78,81.74) -- (63.25,88.56) ;
\draw [shift={(61.77,89.9)}, rotate = 317.86] [color={rgb, 255:red, 74; green, 144; blue, 226 }  ,draw opacity=1 ][line width=0.75]    (4.37,-1.32) .. controls (2.78,-0.56) and (1.32,-0.12) .. (0,0) .. controls (1.32,0.12) and (2.78,0.56) .. (4.37,1.32)   ;
\draw [shift={(72.27,80.4)}, rotate = 137.86] [color={rgb, 255:red, 74; green, 144; blue, 226 }  ,draw opacity=1 ][line width=0.75]    (4.37,-1.32) .. controls (2.78,-0.56) and (1.32,-0.12) .. (0,0) .. controls (1.32,0.12) and (2.78,0.56) .. (4.37,1.32)   ;
\draw  [dash pattern={on 4.5pt off 4.5pt}] (80,120.63) .. controls (80,112.55) and (86.55,106) .. (94.63,106) .. controls (102.72,106) and (109.27,112.55) .. (109.27,120.63) .. controls (109.27,128.72) and (102.72,135.27) .. (94.63,135.27) .. controls (86.55,135.27) and (80,128.72) .. (80,120.63) -- cycle ;
\draw  [color={rgb, 255:red, 0; green, 0; blue, 0 }  ,draw opacity=1 ][dash pattern={on 4.5pt off 4.5pt}] (58,100.63) .. controls (58,92.55) and (64.55,86) .. (72.63,86) .. controls (80.72,86) and (87.27,92.55) .. (87.27,100.63) .. controls (87.27,108.72) and (80.72,115.27) .. (72.63,115.27) .. controls (64.55,115.27) and (58,108.72) .. (58,100.63) -- cycle ;
\draw  [dash pattern={on 4.5pt off 4.5pt}] (388,191.63) .. controls (388,183.55) and (394.55,177) .. (402.63,177) .. controls (410.72,177) and (417.27,183.55) .. (417.27,191.63) .. controls (417.27,199.72) and (410.72,206.27) .. (402.63,206.27) .. controls (394.55,206.27) and (388,199.72) .. (388,191.63) -- cycle ;
\draw  [dash pattern={on 4.5pt off 4.5pt}] (130,155.63) .. controls (130,147.55) and (136.55,141) .. (144.63,141) .. controls (152.72,141) and (159.27,147.55) .. (159.27,155.63) .. controls (159.27,163.72) and (152.72,170.27) .. (144.63,170.27) .. controls (136.55,170.27) and (130,163.72) .. (130,155.63) -- cycle ;
\draw  [dash pattern={on 4.5pt off 4.5pt}] (157.87,169.2) .. controls (157.87,161.36) and (164.22,155) .. (172.07,155) .. controls (179.91,155) and (186.27,161.36) .. (186.27,169.2) .. controls (186.27,177.04) and (179.91,183.4) .. (172.07,183.4) .. controls (164.22,183.4) and (157.87,177.04) .. (157.87,169.2) -- cycle ;
\draw  [dash pattern={on 4.5pt off 4.5pt}] (184.87,180.2) .. controls (184.87,172.36) and (191.22,166) .. (199.07,166) .. controls (206.91,166) and (213.27,172.36) .. (213.27,180.2) .. controls (213.27,188.04) and (206.91,194.4) .. (199.07,194.4) .. controls (191.22,194.4) and (184.87,188.04) .. (184.87,180.2) -- cycle ;
\draw  [dash pattern={on 4.5pt off 4.5pt}] (213,189.17) .. controls (213,181.38) and (219.31,175.07) .. (227.1,175.07) .. controls (234.89,175.07) and (241.2,181.38) .. (241.2,189.17) .. controls (241.2,196.95) and (234.89,203.27) .. (227.1,203.27) .. controls (219.31,203.27) and (213,196.95) .. (213,189.17) -- cycle ;
\draw  [dash pattern={on 4.5pt off 4.5pt}] (240,195.17) .. controls (240,187.38) and (246.31,181.07) .. (254.1,181.07) .. controls (261.89,181.07) and (268.2,187.38) .. (268.2,195.17) .. controls (268.2,202.95) and (261.89,209.27) .. (254.1,209.27) .. controls (246.31,209.27) and (240,202.95) .. (240,195.17) -- cycle ;
\draw  [dash pattern={on 4.5pt off 4.5pt}] (270.2,198.57) .. controls (270.2,190.56) and (276.69,184.07) .. (284.7,184.07) .. controls (292.71,184.07) and (299.2,190.56) .. (299.2,198.57) .. controls (299.2,206.57) and (292.71,213.07) .. (284.7,213.07) .. controls (276.69,213.07) and (270.2,206.57) .. (270.2,198.57) -- cycle ;
\draw  [dash pattern={on 4.5pt off 4.5pt}] (299.2,200.57) .. controls (299.2,192.48) and (305.75,185.93) .. (313.83,185.93) .. controls (321.92,185.93) and (328.47,192.48) .. (328.47,200.57) .. controls (328.47,208.65) and (321.92,215.2) .. (313.83,215.2) .. controls (305.75,215.2) and (299.2,208.65) .. (299.2,200.57) -- cycle ;
\draw  [dash pattern={on 4.5pt off 4.5pt}] (328.47,200.57) .. controls (328.47,192.48) and (335.02,185.93) .. (343.1,185.93) .. controls (351.18,185.93) and (357.73,192.48) .. (357.73,200.57) .. controls (357.73,208.65) and (351.18,215.2) .. (343.1,215.2) .. controls (335.02,215.2) and (328.47,208.65) .. (328.47,200.57) -- cycle ;
\draw  [dash pattern={on 4.5pt off 4.5pt}] (357.73,197.57) .. controls (357.73,189.48) and (364.28,182.93) .. (372.37,182.93) .. controls (380.45,182.93) and (387,189.48) .. (387,197.57) .. controls (387,205.65) and (380.45,212.2) .. (372.37,212.2) .. controls (364.28,212.2) and (357.73,205.65) .. (357.73,197.57) -- cycle ;
\draw  [dash pattern={on 4.5pt off 4.5pt}] (105,139.63) .. controls (105,131.55) and (111.55,125) .. (119.63,125) .. controls (127.72,125) and (134.27,131.55) .. (134.27,139.63) .. controls (134.27,147.72) and (127.72,154.27) .. (119.63,154.27) .. controls (111.55,154.27) and (105,147.72) .. (105,139.63) -- cycle ;
\draw  [dash pattern={on 4.5pt off 4.5pt}] (417,185.63) .. controls (417,177.55) and (423.55,171) .. (431.63,171) .. controls (439.72,171) and (446.27,177.55) .. (446.27,185.63) .. controls (446.27,193.72) and (439.72,200.27) .. (431.63,200.27) .. controls (423.55,200.27) and (417,193.72) .. (417,185.63) -- cycle ;
\draw  [dash pattern={on 4.5pt off 4.5pt}] (446,175.63) .. controls (446,167.55) and (452.55,161) .. (460.63,161) .. controls (468.72,161) and (475.27,167.55) .. (475.27,175.63) .. controls (475.27,183.72) and (468.72,190.27) .. (460.63,190.27) .. controls (452.55,190.27) and (446,183.72) .. (446,175.63) -- cycle ;
\draw  [dash pattern={on 4.5pt off 4.5pt}] (474,165.63) .. controls (474,157.55) and (480.55,151) .. (488.63,151) .. controls (496.72,151) and (503.27,157.55) .. (503.27,165.63) .. controls (503.27,173.72) and (496.72,180.27) .. (488.63,180.27) .. controls (480.55,180.27) and (474,173.72) .. (474,165.63) -- cycle ;
\draw  [dash pattern={on 4.5pt off 4.5pt}] (501,152.63) .. controls (501,144.55) and (507.55,138) .. (515.63,138) .. controls (523.72,138) and (530.27,144.55) .. (530.27,152.63) .. controls (530.27,160.72) and (523.72,167.27) .. (515.63,167.27) .. controls (507.55,167.27) and (501,160.72) .. (501,152.63) -- cycle ;
\draw    (73.98,99.16) -- (80.92,91.54) ;
\draw [shift={(82.27,90.07)}, rotate = 132.35] [color={rgb, 255:red, 0; green, 0; blue, 0 }  ][line width=0.75]    (4.37,-1.32) .. controls (2.78,-0.56) and (1.32,-0.12) .. (0,0) .. controls (1.32,0.12) and (2.78,0.56) .. (4.37,1.32)   ;
\draw [shift={(72.63,100.63)}, rotate = 312.35] [color={rgb, 255:red, 0; green, 0; blue, 0 }  ][line width=0.75]    (4.37,-1.32) .. controls (2.78,-0.56) and (1.32,-0.12) .. (0,0) .. controls (1.32,0.12) and (2.78,0.56) .. (4.37,1.32)   ;
\draw    (121.27,86.07) .. controls (118.31,102.81) and (115.36,120.53) .. (79.91,99.07) ;
\draw [shift={(78.27,98.07)}, rotate = 31.87] [color={rgb, 255:red, 0; green, 0; blue, 0 }  ][line width=0.75]    (8.74,-2.63) .. controls (5.56,-1.12) and (2.65,-0.24) .. (0,0) .. controls (2.65,0.24) and (5.56,1.12) .. (8.74,2.63)   ;
\draw  [color={rgb, 255:red, 208; green, 2; blue, 27 }  ,draw opacity=1 ] (45.43,100.63) .. controls (45.43,85.61) and (57.61,73.43) .. (72.63,73.43) .. controls (87.66,73.43) and (99.83,85.61) .. (99.83,100.63) .. controls (99.83,115.66) and (87.66,127.83) .. (72.63,127.83) .. controls (57.61,127.83) and (45.43,115.66) .. (45.43,100.63) -- cycle ;
\draw  [color={rgb, 255:red, 208; green, 2; blue, 27 }  ,draw opacity=1 ] (67.43,120.63) .. controls (67.43,105.61) and (79.61,93.43) .. (94.63,93.43) .. controls (109.66,93.43) and (121.83,105.61) .. (121.83,120.63) .. controls (121.83,135.66) and (109.66,147.83) .. (94.63,147.83) .. controls (79.61,147.83) and (67.43,135.66) .. (67.43,120.63) -- cycle ;
\draw  [color={rgb, 255:red, 208; green, 2; blue, 27 }  ,draw opacity=1 ] (92.43,139.63) .. controls (92.43,124.61) and (104.61,112.43) .. (119.63,112.43) .. controls (134.66,112.43) and (146.83,124.61) .. (146.83,139.63) .. controls (146.83,154.66) and (134.66,166.83) .. (119.63,166.83) .. controls (104.61,166.83) and (92.43,154.66) .. (92.43,139.63) -- cycle ;
\draw  [color={rgb, 255:red, 208; green, 2; blue, 27 }  ,draw opacity=1 ] (117.43,155.63) .. controls (117.43,140.61) and (129.61,128.43) .. (144.63,128.43) .. controls (159.66,128.43) and (171.83,140.61) .. (171.83,155.63) .. controls (171.83,170.66) and (159.66,182.83) .. (144.63,182.83) .. controls (129.61,182.83) and (117.43,170.66) .. (117.43,155.63) -- cycle ;
\draw  [color={rgb, 255:red, 208; green, 2; blue, 27 }  ,draw opacity=1 ] (144.67,166) .. controls (144.67,150.98) and (156.84,138.8) .. (171.87,138.8) .. controls (186.89,138.8) and (199.07,150.98) .. (199.07,166) .. controls (199.07,181.02) and (186.89,193.2) .. (171.87,193.2) .. controls (156.84,193.2) and (144.67,181.02) .. (144.67,166) -- cycle ;
\draw  [color={rgb, 255:red, 208; green, 2; blue, 27 }  ,draw opacity=1 ] (171.87,180.2) .. controls (171.87,165.18) and (184.04,153) .. (199.07,153) .. controls (214.09,153) and (226.27,165.18) .. (226.27,180.2) .. controls (226.27,195.22) and (214.09,207.4) .. (199.07,207.4) .. controls (184.04,207.4) and (171.87,195.22) .. (171.87,180.2) -- cycle ;
\draw  [color={rgb, 255:red, 208; green, 2; blue, 27 }  ,draw opacity=1 ] (199.9,189.17) .. controls (199.9,174.14) and (212.08,161.97) .. (227.1,161.97) .. controls (242.12,161.97) and (254.3,174.14) .. (254.3,189.17) .. controls (254.3,204.19) and (242.12,216.37) .. (227.1,216.37) .. controls (212.08,216.37) and (199.9,204.19) .. (199.9,189.17) -- cycle ;
\draw  [color={rgb, 255:red, 208; green, 2; blue, 27 }  ,draw opacity=1 ] (226.9,195.17) .. controls (226.9,180.14) and (239.08,167.97) .. (254.1,167.97) .. controls (269.12,167.97) and (281.3,180.14) .. (281.3,195.17) .. controls (281.3,210.19) and (269.12,222.37) .. (254.1,222.37) .. controls (239.08,222.37) and (226.9,210.19) .. (226.9,195.17) -- cycle ;
\draw  [color={rgb, 255:red, 208; green, 2; blue, 27 }  ,draw opacity=1 ] (257.5,198.57) .. controls (257.5,183.54) and (269.68,171.37) .. (284.7,171.37) .. controls (299.72,171.37) and (311.9,183.54) .. (311.9,198.57) .. controls (311.9,213.59) and (299.72,225.77) .. (284.7,225.77) .. controls (269.68,225.77) and (257.5,213.59) .. (257.5,198.57) -- cycle ;
\draw  [color={rgb, 255:red, 208; green, 2; blue, 27 }  ,draw opacity=1 ] (286.63,200.57) .. controls (286.63,185.54) and (298.81,173.37) .. (313.83,173.37) .. controls (328.86,173.37) and (341.03,185.54) .. (341.03,200.57) .. controls (341.03,215.59) and (328.86,227.77) .. (313.83,227.77) .. controls (298.81,227.77) and (286.63,215.59) .. (286.63,200.57) -- cycle ;
\draw  [color={rgb, 255:red, 208; green, 2; blue, 27 }  ,draw opacity=1 ] (315.9,200.57) .. controls (315.9,185.54) and (328.08,173.37) .. (343.1,173.37) .. controls (358.12,173.37) and (370.3,185.54) .. (370.3,200.57) .. controls (370.3,215.59) and (358.12,227.77) .. (343.1,227.77) .. controls (328.08,227.77) and (315.9,215.59) .. (315.9,200.57) -- cycle ;
\draw  [color={rgb, 255:red, 208; green, 2; blue, 27 }  ,draw opacity=1 ] (345.17,197.57) .. controls (345.17,182.54) and (357.34,170.37) .. (372.37,170.37) .. controls (387.39,170.37) and (399.57,182.54) .. (399.57,197.57) .. controls (399.57,212.59) and (387.39,224.77) .. (372.37,224.77) .. controls (357.34,224.77) and (345.17,212.59) .. (345.17,197.57) -- cycle ;
\draw  [color={rgb, 255:red, 208; green, 2; blue, 27 }  ,draw opacity=1 ] (375.43,191.63) .. controls (375.43,176.61) and (387.61,164.43) .. (402.63,164.43) .. controls (417.66,164.43) and (429.83,176.61) .. (429.83,191.63) .. controls (429.83,206.66) and (417.66,218.83) .. (402.63,218.83) .. controls (387.61,218.83) and (375.43,206.66) .. (375.43,191.63) -- cycle ;
\draw  [color={rgb, 255:red, 208; green, 2; blue, 27 }  ,draw opacity=1 ] (404.43,185.63) .. controls (404.43,170.61) and (416.61,158.43) .. (431.63,158.43) .. controls (446.66,158.43) and (458.83,170.61) .. (458.83,185.63) .. controls (458.83,200.66) and (446.66,212.83) .. (431.63,212.83) .. controls (416.61,212.83) and (404.43,200.66) .. (404.43,185.63) -- cycle ;
\draw  [color={rgb, 255:red, 208; green, 2; blue, 27 }  ,draw opacity=1 ] (433.43,175.63) .. controls (433.43,160.61) and (445.61,148.43) .. (460.63,148.43) .. controls (475.66,148.43) and (487.83,160.61) .. (487.83,175.63) .. controls (487.83,190.66) and (475.66,202.83) .. (460.63,202.83) .. controls (445.61,202.83) and (433.43,190.66) .. (433.43,175.63) -- cycle ;
\draw  [color={rgb, 255:red, 208; green, 2; blue, 27 }  ,draw opacity=1 ] (461.43,165.63) .. controls (461.43,150.61) and (473.61,138.43) .. (488.63,138.43) .. controls (503.66,138.43) and (515.83,150.61) .. (515.83,165.63) .. controls (515.83,180.66) and (503.66,192.83) .. (488.63,192.83) .. controls (473.61,192.83) and (461.43,180.66) .. (461.43,165.63) -- cycle ;
\draw  [color={rgb, 255:red, 208; green, 2; blue, 27 }  ,draw opacity=1 ] (488.43,152.63) .. controls (488.43,137.61) and (500.61,125.43) .. (515.63,125.43) .. controls (530.66,125.43) and (542.83,137.61) .. (542.83,152.63) .. controls (542.83,167.66) and (530.66,179.83) .. (515.63,179.83) .. controls (500.61,179.83) and (488.43,167.66) .. (488.43,152.63) -- cycle ;
\draw [color={rgb, 255:red, 208; green, 2; blue, 27 }  ,draw opacity=1 ]   (70.93,101.68) -- (51.97,113.35) ;
\draw [shift={(50.27,114.4)}, rotate = 328.39] [color={rgb, 255:red, 208; green, 2; blue, 27 }  ,draw opacity=1 ][line width=0.75]    (6.56,-1.97) .. controls (4.17,-0.84) and (1.99,-0.18) .. (0,0) .. controls (1.99,0.18) and (4.17,0.84) .. (6.56,1.97)   ;
\draw [shift={(72.63,100.63)}, rotate = 148.39] [color={rgb, 255:red, 208; green, 2; blue, 27 }  ,draw opacity=1 ][line width=0.75]    (6.56,-1.97) .. controls (4.17,-0.84) and (1.99,-0.18) .. (0,0) .. controls (1.99,0.18) and (4.17,0.84) .. (6.56,1.97)   ;
\draw [color={rgb, 255:red, 208; green, 2; blue, 27 }  ,draw opacity=1 ]   (42.27,145.4) .. controls (58.5,134.89) and (60.14,128.95) .. (59.39,116.24) ;
\draw [shift={(59.27,114.4)}, rotate = 85.91] [color={rgb, 255:red, 208; green, 2; blue, 27 }  ,draw opacity=1 ][line width=0.75]    (10.93,-3.29) .. controls (6.95,-1.4) and (3.31,-0.3) .. (0,0) .. controls (3.31,0.3) and (6.95,1.4) .. (10.93,3.29)   ;
\draw [color={rgb, 255:red, 74; green, 144; blue, 226 }  ,draw opacity=1 ]   (60.28,91.24) -- (52.75,98.06) ;
\draw [shift={(51.27,99.4)}, rotate = 317.86] [color={rgb, 255:red, 74; green, 144; blue, 226 }  ,draw opacity=1 ][line width=0.75]    (4.37,-1.32) .. controls (2.78,-0.56) and (1.32,-0.12) .. (0,0) .. controls (1.32,0.12) and (2.78,0.56) .. (4.37,1.32)   ;
\draw [shift={(61.77,89.9)}, rotate = 137.86] [color={rgb, 255:red, 74; green, 144; blue, 226 }  ,draw opacity=1 ][line width=0.75]    (4.37,-1.32) .. controls (2.78,-0.56) and (1.32,-0.12) .. (0,0) .. controls (1.32,0.12) and (2.78,0.56) .. (4.37,1.32)   ;
\draw [color={rgb, 255:red, 74; green, 144; blue, 226 }  ,draw opacity=1 ]   (63.04,52.24) .. controls (57.28,56.08) and (57.05,69.14) .. (65,80.79) ;
\draw [shift={(66.04,82.24)}, rotate = 233.13] [color={rgb, 255:red, 74; green, 144; blue, 226 }  ,draw opacity=1 ][line width=0.75]    (10.93,-3.29) .. controls (6.95,-1.4) and (3.31,-0.3) .. (0,0) .. controls (3.31,0.3) and (6.95,1.4) .. (10.93,3.29)   ;
\draw    (492.2,84.6) .. controls (526.67,98.39) and (537.31,119) .. (516.6,151.16) ;
\draw [shift={(515.63,152.63)}, rotate = 303.69] [color={rgb, 255:red, 0; green, 0; blue, 0 }  ][line width=0.75]    (10.93,-3.29) .. controls (6.95,-1.4) and (3.31,-0.3) .. (0,0) .. controls (3.31,0.3) and (6.95,1.4) .. (10.93,3.29)   ;

\draw (66,37.4) node [anchor=north west][inner sep=0.75pt]  [color={rgb, 255:red, 74; green, 144; blue, 226 }  ,opacity=1 ]  {$\delta $};
\draw (117,67.4) node [anchor=north west][inner sep=0.75pt]    {$\delta $};
\draw (24,141.4) node [anchor=north west][inner sep=0.75pt]  [color={rgb, 255:red, 208; green, 2; blue, 27 }  ,opacity=1 ]  {$2\delta $};
\draw (475.63,68.4) node [anchor=north west][inner sep=0.75pt]    {$\widetilde{\mathcal{X}}$};

\end{tikzpicture}
    \caption{The black curve represents the bifurcation set $\widetilde{\mathcal{X}}$, and the blue curve outlines the boundary of its $\delta$-neighborhood $\widetilde{\mathcal{X}}_\delta$. Dashed black circles of radius $\delta$ are used to cover $\widetilde{\mathcal{X}}$, and red circles of radius $2\delta$ are centered at the same locations. These enlarged red circles together cover the entire neighborhood $\widetilde{\mathcal{X}}_\delta$.
}

    \label{figure:covering}
\end{figure}

\subsection{Upper-level Oracle Complexity Analysis}\label{sec:Outer-loop Iteration Complexity Analysis}

We now turn to the second term in the integral decomposition \eqref{eq:splitofintegral2}, which corresponds to the region $\R^n\setminus\widetilde{\mathcal{X}}_\delta.$ In this region, we aim to control the approximation error $\|y^{\text{cd}}(x)-\hat{y}(x)\|,$ where $y^{\text{cd}}(x)$ is the exact solution defined in \eqref{eq:ycd}, and $\hat{y}(x)$ is the output of the cubic-regularized Newton method (see Algorithm~\ref{alg:cubic_newton}) after finite  $K$ steps which are independent of $x$.

\begin{algorithm}[htbp]
\caption{Cubic-Regularized Newton Method for Solving Lower-level Problem}
\label{alg:cubic_newton}
\KwIn{fixed $x \in \mathcal{X}$; initial point $y_0$; number of steps $K$; regularization parameter $M=\overline{\overline{L}}_g$}
\KwOut{approximate solution $\hat{y}$}
\For{$k \gets 0$ \KwTo $K-1$}{
    Compute gradient $g_k \gets \nabla_y g(x, y_k)$\;
    Compute Hessian $H_k \gets \nabla^2_{yy} g(x, y_k)$\;
    Solve the subproblem:
    \[
    s_k \gets \arg\min_s \left\{ g_k^\top s + \frac{1}{2} s^\top H_k s + \frac{M}{6} \|s\|^3 \right\}
    \]
    Update $y_{k+1} \gets y_k + s_k$\;
}
Select the best iterate by the stationarity measure:
\[
k^\ast \gets \arg\min_{0 \le k \le K}\;
\max\!\left\{
\sqrt{\tfrac{1}{M}\,\big\|\nabla_y g(x,y_k)\big\|},
-\tfrac{2}{3M}\,\lambda_{\min}\!\big(\nabla^2_{yy} g(x,y_k)\big)
\right\}.
\]
\Return $\hat{y}=y_{k^\ast}$
\end{algorithm}

By Assumption~\ref{assumption:general1}(\ref{assumption:general1(7)}), the level set $\mathcal{Y}$ is compact, and by Assumption~\ref{assumption:descent} all iterates of the cubic-regularized Newton method remain in $\mathcal{Y}$; hence we only need to consider stationary points of $g(x,\cdot)$ within $\mathcal{Y}$. 

At a high level, our approach to analyzing the algorithm convergence is as follows: since we focus on points $x$ outside the bifurcation set $\widetilde{\mathcal X}$, $g(x,y)$ is Morse in $y$, and all of its stationary points are isolated and non-degenerate. At each such point, the norm of every Hessian eigenvalue is bounded below by $\mu>0$ (to be shown in Lemma \ref{lemma:eigenvalue}). Thus, by the Lipschitz continuity of Hessian from Assumption~\ref{assumption:general1}(\ref{assumption:general1(11)}), within an open ball of radius $0<r\leq\mu/(2\overline{\overline{L}}_g)$ centered at any stationary point in $\mathcal Y$, all eigenvalues of $\nabla^2_{yy} g(x,y)$ have absolute values lower bounded by $\mu/2$. As illustrated in Figure~\ref{fig:cubicnewton}, we denote these open neighborhoods by blue-shaded disks centered at each stationary point. Outside the union of all radius-$r$ open balls centered at stationary points, the remaining points in $\mathcal Y$ form a compact set that contains no stationary point. By continuity of $\nabla_y g(x,y)$, it follows that $\|\nabla_y g(x,y)\|$ is bounded below by a positive constant on this compact set. These structural conditions allow us to characterize the iterates produced by the cubic-regularized Newton method and to derive non-asymptotic bounds on the approximation error $\|\hat y(x)-y^{\mathrm{cd}}(x)\|$. Specifically, known results for cubic Newton~\citep{nesterov2006cubic} guarantee that after $K$ steps the algorithm produces a point $\hat y$ such that
$$
\|\nabla_y g(x,\hat y)\|=\mathcal O\!\left(K^{-2/3}\right),
$$
and the minimal eigenvalue of the Hessian is bounded from below,
$$
\lambda_{\min}\!\big(\nabla^2_{yy} g(x,\hat y)\big)=-\Omega\!\left(K^{-1/3}\right).
$$
These two estimates lead to a clear two-phase convergence behavior (to be shown in Lemma \ref{lemma:approximationerror}). First, by the gradient bound, once $K$ is large enough so that $\|\nabla_y g(x,\hat y)\|$ falls below the positive lower bound that holds outside these radius-$r$ open balls (the existence of such a positive lower bound will be shown in Lemma \ref{lemma:gradient_lower_bound}). Then the iterate $\hat y$ must lie inside the union of these open balls in $\mathcal Y$. Next, by the curvature condition together with the termination criterion for cubic Newton, which requires $\lambda_{\min}\!\big(\nabla^2_{yy} g(x,\hat y)\big) > -C/K^{1/3}$ for some constant $C$, the radius-$r$ open balls around saddle points and local maximizers are ruled out as possible locations of $\hat y$ whenever $K>(2C/\mu)^3$. Therefore, for sufficiently large $K$, $\hat y$ lies in an $r$-neighborhood of a local minimizer, where $g(x,y)$ is strongly convex in $y$. Within this neighborhood, where $g(x,\cdot)$ is strongly convex and also Lipschitz smooth in $y$ (as assumed in Assumption~\ref{assumption:general1}(\ref{assumption:general1(4)})), we can establish an upper bound on the distance between the approximate solution $\hat y$ and the exact solution $y^{\mathrm{cd}}$.

\begin{figure}
    \centering
    \includegraphics[width=0.5\linewidth]{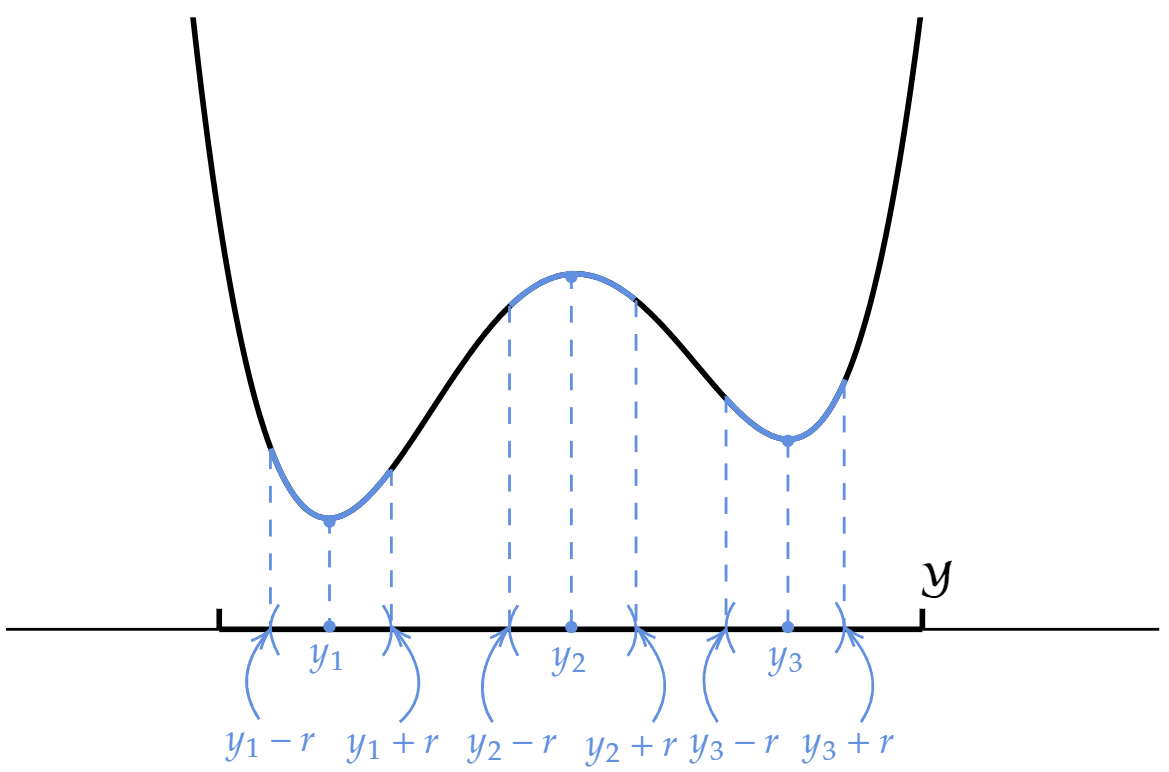}
    \caption{We construct open neighborhoods with radius $r$ around each of the three stationary points. Outside these neighborhoods, the gradient norm $\|\nabla_y g(x,y)\|$ is bounded below by a positive constant. Within each neighborhood, the eigenvalues of the Hessian $\nabla^2_{yy} g(x,y)$ have a positive lower bound.}
    \label{fig:cubicnewton}
\end{figure}

To formalize this high level idea, we now state two technical lemmas. Lemma \ref{lemma:eigenvalue} guarantees that outside $\widetilde{\mathcal{X}}_\delta$, which is $\delta$-neighborhood of the bifurcation set $\widetilde{\mathcal{X}}$, every stationary point $y \in \mathcal{Y}'$ of $g(x,\cdot)$ has a Hessian $\nabla^2_y g(x,y)$ whose eigenvalues are uniformly bounded away from zero. Here $\mathcal{Y}'$ is a slightly enlarged compact superset of $\mathcal{Y}$, introduced for technical convenience in the proof. Lemma~\ref{lemma:gradient_lower_bound} guarantees that for any $x \in \mathcal{X} \setminus \widetilde{\mathcal{X}}_\delta$, the gradient norm $\|\nabla_y g(x,y)\|$ is uniformly lower bounded by a positive constant $\alpha(\delta,r)$ for all $y$ away from the $r$-neighborhood of the stationary set. The constant $\alpha(\delta,r)$ depends only on $r$ and $\delta$, not on $x$. These two results play a key role in characterizing the behavior of the cubic-regularized Newton method.

\begin{lemma}\label{lemma:eigenvalue}
    Suppose Assumption~\ref{assumption:general1} holds. Let $\mathcal{Y}'\supset\mathcal{Y}$ be a compact set with $\mathcal{Y}\subset\mathrm{int}(\mathcal{Y}')$, where $\mathcal{Y}$ is defined in Assumption~\ref{assumption:general1}(\ref{assumption:general1(7)}). Then there exists a positive function $\mu(\delta)$ depending only on $\delta$ such that for any $x \in \mathcal{X} \setminus \widetilde{\mathcal{X}}_\delta$ and any stationary points $y\in\mathcal{Y}'$ of $g(x, \cdot)$, all eigenvalues $\lambda$ of $\nabla^2_{yy}g(x, y)$ satisfy $|\lambda| \geq \mu(\delta)$.
\end{lemma}

\begin{proof}
    See Appendix~\ref{proof:lemma:eigenvalue}.
\end{proof}

\begin{lemma}\label{lemma:gradient_lower_bound}
Suppose Assumption~\ref{assumption:general1} holds. Define $\mathrm{Crit}(x) := \{ y \in \mathcal{Y}' : \nabla_{y} g(x, y) = 0 \}$, where $\mathcal{Y}'$ is from Lemma~\ref{lemma:eigenvalue}. For any $r > 0$, there exists a positive constant $\alpha(\delta,r)$ only depends on $\delta$ and $r$ such that $\|\nabla_y g(x,y)\|\geq\alpha(\delta,r)$ for all $x \in \mathcal{X}\setminus\widetilde{\mathcal{X}}_\delta$ and all $y \in \mathcal{Y} \setminus \mathrm{Crit}_r(x)$, where $\mathrm{Crit}_r(x)$ denotes the $r$-neighborhood of $\mathrm{Crit}(x)$ in $\R^m$, i.e.,
$$\mathrm{Crit}_r(x):=\{y\in\R^m:\mathrm{dist}(y,\mathrm{Crit}(x))\leq r\}.$$
\end{lemma}

\begin{proof}
    See Appendix~\ref{proof:lemma:gradient_lower_bound}.
\end{proof}

We now formalize this two-phase convergence process in the following lemma, which provides an estimate on the gradient norm and the distance to the local minimizer after a finite number of iterations of the cubic-regularized Newton method, for any $x \in \mathcal{X} \setminus \widetilde{\mathcal{X}}_\delta$.

\begin{lemma}\label{lemma:approximationerror}
    Suppose Assumption~\ref{assumption:general1} holds and $x\in\mathcal{X}\setminus\widetilde{\mathcal{X}}_\delta$. We perform $K$ iterations of the cubic Newton method update for the lower-level problem at the point $x$ with $M=\overline{\overline{L}}_g$ from Assumption~\ref{assumption:general1}(\ref{assumption:general1(11)}). Then sequence $\{y_k\}_{k=1}^{K}$ generated by this method converges to a local minimizer $y^{\text{cd}}(x)$ when $K$ goes to infinity. If the iteration number $K$ satisfies
    \begin{align}\label{equation:requirement_for_K}
        K&>\max\left\{\frac{256(\overline{g}_0-\underline{g})\left(\overline{\overline{L}}_g\right)^{1/2}}{9\left(\min\left\{\alpha(\delta,\frac{\mu(\delta)}{2\overline{\overline{L}}_g}),\frac{(\mu(\delta))^2}{4\overline{\overline{L}}_g}\right\}\right)^{3/2}},\frac{768(\overline{g}_0-\underline{g})\overline{\overline{L}}^2_g}{\left(\mu(\delta)\right)^3}\right\}+\frac{256(\overline{g}_0-\underline{g})\left(\overline{\overline{L}}_g\right)^{1/2}}{9(\frac{\mu(\delta)\rho}{2})^{3/2}}\nonumber\\
        &=:K_1+K_2,
    \end{align} 
     where $\overline{g}_0$ and $\underline{g}$ are from Proposition~\ref{prop:general}(\ref{assumption:general1(6)}) and Assumption~\ref{assumption:general1}(\ref{assumption:general1(4)}), functions $\mu(\cdot)$ and $\alpha(\cdot,\cdot)$ are from Lemma~\ref{lemma:eigenvalue} and Lemma~\ref{lemma:gradient_lower_bound}, and $\rho>0$ is a constant which represents the approximation error in the lower-level solution, then we have
     $$\min_{k=K_1+1,...,K}\left\|\nabla_y g(x,y_k)\right\|\leq\mu(\delta)\rho.$$
     Define $\hat{y}(x):=y_{k_0}$, where $k_0=\arg\min_{k=K_1+1,...,K}\|\nabla_y g(x,y_k)\|$. We also have
    $$\|\hat{y}(x)-y^{\text{cd}}(x)\|\leq\rho.$$
\end{lemma}

\begin{proof}
    See Appendix~\ref{proof:lemma:approximationerror}.
\end{proof}

Lemma \ref{lemma:approximationerror} indicates that, once we are outside the $\delta$-neighbourhood of the bifurcation set, $K$ steps of the cubic-regularised Newton method already deliver an $\rho$-accurate lower-level solution. Section \ref{sec:measureofbifurcationpoints} in turn shows that this $\delta$-neighborhood occupies only a vanishingly small portion of the domain. Taken together, these two facts imply that the bias in our hypergradient estimator is dominated by an $\rho$-term from the regular region $\mathcal{X}\setminus\widetilde{\mathcal{X}}_\delta$ and a $\delta$-term arising from the $\delta$-neighborhood $\widetilde{\mathcal{X}}_\delta$. The next proposition formalizes this conclusion by giving an explicit upper bound on the gap between the exact gradient \eqref{eq:exactgradient} and the inexact gradient \eqref{eq:inexactgradient}.

\begin{proposition}\label{proposition:gradient_error}
Suppose Assumptions~\ref{assumption:method}-\ref{assumption:discontinuous_point},~\ref{assumption:Minkowski_dimension} hold, and the iteration number of the inner loop $K$ satisfies \eqref{equation:requirement_for_K} depends on $\delta$. Define $\nabla^K_x\phi^{\text{cd}}_\xi(x)$ as follows:
\begin{align}
\nabla^K_x\phi^{\text{cd}}_\xi(x):=\int_{\mathbb{R}^n}\nabla_xh_\xi(x-z)f(z,\hat{y}(z))dz,
\end{align}
where $\hat{y}(z)$ denotes the approximate lower-level solution obtained by applying $K$ steps of the cubic-regularized Newton method (see Algorithm~\ref{alg:cubic_newton}) at the point $z$. We have the following approximation error between exact gradient \eqref{eq:exactgradient} and inexact gradient \eqref{eq:inexactgradient} 
    \begin{align*}
        \left\|\nabla_x\pcd_{\xi }(x)-\nabla^K_x\phi^{\text{cd}}_\xi(x)\right\|\leq C_1(n,\xi)\rho+C_2(n,\xi)\delta^{(n-d)/2},
    \end{align*}
for some constant $C_1(n,\xi)$, $C_2(n,\xi)$ which depend on $n$ and $\xi$, where $d$ is the Minkowski dimension of $\widetilde{\mathcal{X}}$ assumed in Assumption~\ref{assumption:Minkowski_dimension}.

\end{proposition}

\begin{proof}
    See Appendix~\ref{proof:proposition:gradient_error}.
\end{proof}

Proposition~\ref{proposition:gradient_error} provides an explicit bound on the error incurred when using approximate lower-level solutions to evaluate the smoothed hyperfunction gradient. This bound consists of two parts: a term depending on the approximation accuracy $\rho$ of the cubic Newton method, and a term due to the $\delta$-neighborhood around the bifurcation set $\widetilde{\mathcal{X}}$.

In practice, we do not compute $\nabla^K_x\phi^{\text{cd}}_\xi(x)$ exactly, but estimate it using Monte Carlo sampling \eqref{eq:estimator}. This results in an additional variance term due to stochastic sampling. We now combine the bias bound from Proposition~\ref{proposition:gradient_error} with the variance bound of the Monte Carlo estimator to obtain a complete characterization of the stochastic gradient.  We consider the following biased stochastic gradient descent (\texttt{SGD}) update
\begin{align}\label{eq:update}
    x_{t+1}=\proj_{\mathcal{X}}(x_t-\beta_t\hat{\nabla}^{K_t}_x\phi^{\text{cd}}_{\xi }(x_t)),\quad\quad t=0,1,2,...,
\end{align}
where at every iteration, the search direction is computed using the following Monte-Carlo estimator: 
\begin{align}
    \hat{\nabla}^K_x\phi^{\text{cd}}_{\xi }(x)=\frac{1}{N\xi}\sum_{i=1}^Nu^{(i)}f(x+\xi u^{(i)},\hat{y}(x+\xi u^{(i)}))
\end{align}
Here $N$ denotes the number of Monte Carlo samples drawn independently from the standard normal distribution. Each $u^{(i)}\sim\mathcal{N}(0,I_n)$ is sampled independently from the standard multivariate Gaussian distribution in $\R^n$. This estimator is a biased estimator of the exact gradient $\nabla_x\pcd_{\xi }(x)$. The bias, defined as the difference between its expectation and the exact gradient
$$\left\|\E\left[\hat{\nabla}^K_x\phi^{\text{cd}}_{\xi }(x)\right]-\nabla_x\pcd_{\xi }(x)\right\|=\left\|\nabla^K_x\phi^{\text{cd}}_\xi(x)-\nabla_x\pcd_{\xi }(x)\right\|,$$
is bounded by Proposition~\ref{proposition:gradient_error}. The variance of the estimator is analyzed as follows: Let
\begin{align}
    \zeta:=\frac{1}{\xi }uf(x+\xi u,\hat{y}(x+\xi u)),\quad\quad\text{where }u\sim\mathcal{N}(0,I_n).
\end{align}
Then $\hat{\nabla}^K_x\phi^{\text{cd}}_{\xi }(x)$ is the empirical average of $N$ i.i.d. samples of $\zeta$. Note that we assume that the number of lower-level iterations satisfies $K\geq K_0$ in Assumption \ref{assumption:descent}, so $\hat{y}(x+\xi u))$ is always in the level set $\{y:g(x+\xi u,y)\leq g(x+\xi u,y_0)\}$. By Proposition~\ref{prop:general}(\ref{assumption:general1(8)}), we have
$$|f(x+\xi u,\hat{y}(x+\xi u))|\leq \overline{f}.$$
Thus, for each coordinate $j=1,..,n$, we obtain 
$$\E[\zeta^2_j]=\frac{1}{\xi^2}\E[u_j^2f(x+\xi u,\hat{y}(x+\xi u))^2]\leq\frac{1}{\xi^2}\E[u^2_j(\overline{f})^2]=\frac{(\overline{f})^2}{\xi ^2}.$$
This implies
$$\text{Var}(\zeta_j)\leq\frac{(\overline{f})^2}{\xi ^2},\quad\text{Var}\left(\left(\hat{\nabla}^K_x\phi^{\text{cd}}_{\xi }(x)\right)_j\right)\leq\frac{(\overline{f})^2}{N\xi ^2}.$$
Summing over all coordinates $j=1,\cdots, n$, the total variance satisfies
\begin{align}
    \E\left[\left\|\hat{\nabla}^K_x\phi^{\text{cd}}_{\xi }(x)-\nabla^K_x\phi^{\text{cd}}_\xi(x)\right\|^2\right]\leq\frac{n(\overline{f})^2}{N\xi^2}.
\end{align}

With the bias and variance bounds in place, we proceed to analyze the convergence of the proposed algorithm. The statement and proof of the following Theorem~\ref{thm:biased_sgd_conv} can be viewed as a variant of \citet[Theorem6.6]{lan2020first}: in \citet[Theorem6.6]{lan2020first}, the analysis is carried out for mirror descent methods, whereas here we specialize \citet[Equation~(6.2.1)]{lan2020first} to the case $h\equiv 0$ and take $V$ in \citet[Equation~(6.2.6)]{lan2020first} to be the squared Euclidean distance, which yields projected \texttt{SGD}. The key difference from \citet[Theorem~6.6]{lan2020first} is that our gradient estimator is biased, while theirs is unbiased.

\begin{theorem}\label{thm:biased_sgd_conv}
Suppose Assumptions~\ref{assumption:method}-\ref{assumption:discontinuous_point},~\ref{assumption:Minkowski_dimension} hold. Let $\{\rho_t\}$, $\{\delta_t\}$ and $\{N_t\}$ be sequences of positive numbers. At each iteration $t$, we solve the lower-level problem $g(x_t,y)$ using the cubic-regularized Newton method (Algorithm~\ref{alg:cubic_newton}) with $M=\overline{\overline{L}}_g$ from Assumption~\ref{assumption:general1}(\ref{assumption:general1(11)}) for $K_t$ steps, where $K_t$ satisfies 
\begin{align}\label{eq:requirementforkt}
        K_t>\max\left\{\frac{256(\overline{g}_0-\underline{g})\left(\overline{\overline{L}}_g\right)^{1/2}}{9\left(\min\left\{\alpha(\delta_t,\frac{\mu(\delta_t)}{2\overline{\overline{L}}_g}),\frac{(\mu(\delta_t))^2}{4\overline{\overline{L}}_g}\right\}\right)^{3/2}},\frac{768(\overline{g}_0-\underline{g})\overline{\overline{L}}^2_g}{\left(\mu(\delta_t)\right)^3}\right\}+\frac{256(\overline{g}_0-\underline{g})\left(\overline{\overline{L}}_g\right)^{1/2}}{9(\mu(\delta_t)\rho_t)^{3/2}}.
    \end{align} 
We bound the bias and variance as follows
\begin{align}\label{eq:bias}
\left\|\E\left[\hat{\nabla}^K_x\phi^{\text{cd}}_{\xi }(x)\right]-\nabla_x\pcd_{\xi }(x)\right\| & =\left\|\nabla^K_x\phi^{\text{cd}}_\xi(x)-\nabla_x\pcd_{\xi }(x)\right\|\nonumber\\
& \leq C_1(n,\xi)\cdot\rho +C_2(n,\xi)\cdot\delta^{(n-d)/2}=:\Delta(\rho,\delta)
\end{align}
\begin{align}\label{eq:variance}
\E\left[\left\|\hat{\nabla}^K_x\phi^{\text{cd}}_{\xi }(x)-\nabla^K_x\phi^{\text{cd}}_\xi(x)\right\|^2\right]\leq \frac{n(\overline{f})^2}{N\xi^2}=:\sigma^2(N).
\end{align}
Suppose the sequence $\{x_t\}$ is generated by Algorithm~\ref{alg:main} with step size $\beta_t<1/\overline{L}_{\phi^{\text{cd}}_{\xi}}$ for each iteration, where $\overline{L}_{\phi^{\text{cd}}_{\xi}}$ is from Proposition~\ref{proposition:lipschitz}, and the probability mass function $P_R(t)$ of the random index $R$ is chosen such that for any $t=1,...,T$
\begin{align*}
    P_R(t):=\mathrm{Prob}\{R=t\}=\frac{\beta_t-\overline{L}_{\phi^{\text{cd}}_{\xi}}\beta_t^2}{\sum_{t=1}^T(\beta_t-\overline{L}_{\phi^{\text{cd}}_{\xi}}\beta_t^2)}.
\end{align*}
Then for any $T \ge 1$, for the random index $R$ drawn according to the probability mass function $P_R(t)$ given above, we have
\begin{align*}
    \E[\|\widetilde{\Phi}_{\mathcal{X},R}\|^2]\leq\frac{2\overline{f}+\sum_{t=1}^T\sqrt{\frac{2}{\pi}}\frac{\overline{f}}{\xi}\Delta(\rho_t,\delta_t)\beta_t+\sum_{t=1}^T(\sigma^2(N_t)+(\Delta(\rho_t,\delta_t))^2)\beta_t}{\sum_{t=1}^T(\beta_t-\overline{L}_{\phi^{\text{cd}}_{\xi}}\beta_t^2)},
\end{align*}
where $\overline{f}$ is from Proposition~\ref{prop:general}(\ref{assumption:general1(8)}), and $\widetilde{\Phi}_{\mathcal{X},t}$ is defined as follow:
$$\widetilde{\Phi}_{\mathcal{X},t}:=\frac{x_t-\proj_{\mathcal{X}}(x_t-\beta_t \hat{\nabla}^{K_t}_x\phi^{\text{cd}}_{\xi }(x_t))}{\beta_t},$$
which coincides with the projected gradient mapping in the randomized setting considered in \cite[Theorem~6.6]{lan2020first}. In particular, choosing 
$$\rho_t=\frac{1}{t+1},\quad\delta_t=\frac{1}{(t+1)^{2/(n-d)}},\quad N_t=t+1$$
and constant step size $\beta_t=\beta<1/\overline{L}_{\phi^{\text{cd}}_{\xi}}$ yields the following convergence result
\begin{align}\label{eq:rate}
    \E[\|\widetilde{\Phi}_{\mathcal{X},R}\|^2]\leq\mathcal{O}\left(\frac{\log T}{T}\right).
\end{align}
\end{theorem}

\begin{proof}
    See Appendix~\ref{proof:thm:biased_sgd_conv}.
\end{proof}

\begin{corollary}
Under the setting of Theorem~\ref{thm:biased_sgd_conv}, we let $\rho_t=1/(t+1)$, $\delta_t=(t+1)^{-2/(n-d)}$, $N_t=t+1$ and use a constant
step size $\beta_t=\beta<\overline{L}_{\phi^{\text{cd}}_{\xi}}/2$.
To achieve $\mathbb{E}\big[\|\widetilde{\Phi}_{\mathcal X,R}\|^{2}\big]\le \epsilon^{2}$, the total number of function-value oracle calls $Q_f(\epsilon)=\sum_{t=1}^{T}N_t$ satisfies $Q_f(\epsilon)=\widetilde{O}(\epsilon^{-4})$.
\end{corollary}

However, in order to satisfy the requirement \eqref{eq:requirementforkt}, the number of inner iterations $K_t$ needed at each step remains unclear under the current assumptions. Specifically, the decay rates of $\mu(\delta)$ and $\alpha(\delta,\mu(\delta)/(2\overline{\overline{L}}_g))$ with respect to $\delta$ are unknown. To better understand the decay rate, we further investigate the structure of the bifurcation points of $g$ in the next section.

\subsection{Lower-level Oracle Complexity under Fold Bifurcation Assumption}\label{sec:foldbifurcation}

The upper-level convergence result obtained in the previous section relies on the requirement \eqref{eq:requirementforkt}, yet the exact lower-level oracle complexity remains unclear because it depends on the decay rates of $\mu(\delta)$ and $\alpha(\delta,\mu(\delta)/(2\overline{\overline{L}}g))$ as $\delta \to 0$. These rates are determined by intrinsic properties of $g(x,y)$, in particular by the nature of degenerate stationary points where $\nabla^2_{yy}g(x,y)$ becomes singular. To better characterize the local structure of such degenerate points, we draw upon bifurcation theory, which provides a framework for classifying equilibrium degeneracies in dynamical systems. 

Consider, for example, a parameterized system
\begin{align}\label{eq:dynamicsystem}
\dot{y} = F(x, y),
\end{align}
where $x \in \mathbb{R}^n$ is a parameter and $y \in \mathbb{R}^m$ is the state variable.
A central question in bifurcation theory is how the qualitative behavior of solutions to this system changes as $x$ varies. In particular, bifurcation occurs when small perturbations in $x$ induce sudden changes in the number or stability of equilibria, i.e., solutions $y$ satisfying $F(x,y)=0$.
Importantly, bifurcation theory provides a framework for classifying equilibrium degeneracies and shows that different types, such as fold and cusp, exhibit fundamentally distinct local behaviors.
These suggest that, in principle, different classes of degenerate stationary points can be analyzed separately within bilevel optimization.

In the bilevel setting, we regard $\nabla_y g(x,y)$ as analogous to $F(x,y)$, which connects the analysis of degenerate stationary points to classical results from bifurcation theory.
Since a comprehensive characterization of bifurcation structures in high-dimensional parameter and state spaces remains elusive, we focus on a prototypical and well-understood case: the fold bifurcation, one of the simplest yet most fundamental bifurcation types.
We introduce the fold bifurcation of stationary points as follows, with its classical definition provided in~\cite{kuznetsov1998elements,guckenheimer2013nonlinear}.

\begin{definition}[Fold Bifurcation of Stationary Points]\label{def:fold}
    Suppose $g(x,y)$ is three times continuously differentiable and $(\bar{x},\bar{y})$ is a stationary point of $g(x,y)$ with respect to $y$. We say $\bar{x}$ is a fold bifurcation point associated with the stationary point $\bar{y}$ if the following holds:
    \begin{enumerate}
        \item $\nabla^2_{yy}g(\bar{x},\bar{y})$ has exactly one zero eigenvalue and all other eigenvalues are nonzero;
        \item The unit eigenvector $v$ corresponding to the zero eigenvalue of $\nabla^2_{yy}g(\bar{x},\bar{y})$ satisfies
        $$(\nabla_x (\nabla_y g(x,y)^\top v))\big|_{(\bar{x},\bar{y})}\ne 0;$$
        \item Let $\nabla^3_{yyy} g$ denote the third-order derivative tensor field of $g$ with respect to $y$, and $v$ denote the unit eigenvector corresponding to the zero eigenvalue, then the following holds
        $$\nabla^3_{yyy} g(\bar{x},\bar{y})[v,v,v]\ne 0.$$
    \end{enumerate}
\end{definition}

The first condition in Definition~\ref{def:fold} ensures that the characteristic polynomial of the Hessian has a simple zero root, while the third condition in Definition~\ref{def:fold} rules out third-order degeneracy along the corresponding eigenvector $v$. 

The second condition in Definition~\ref{def:fold} requires that a tiny change of the parameter in some direction instantly changes the stationarity condition along the degenerate direction $v$. It is a prevalent condition, with prevalence understood in the parametric-family sense discussed in Remark~\ref{rmk:generic-vs-ae}. Specifically, it is not hard to see that the second condition is equivalent to requiring that the matrix
\begin{align}\label{eq:full_row_rank}
\begin{bmatrix}
\nabla^2_{yx} g(x,y) & \nabla^2_{yy} g(x,y)
\end{bmatrix}
\end{align}
has full row rank at any stationary point $(x,y)$ of $g(x,y)$ with respect to $y$. To justify this equivalent condition is prevalent, we observe that $\nabla_y g(x,y)$ can be viewed as a vector-valued map from $\mathcal{X} \times \mathbb{R}^m$ to $\mathbb{R}^m$. By applying Sard's theorem, one can conclude that after an arbitrarily small linear perturbation of $g$ with respect to $y$, its Jacobian \eqref{eq:full_row_rank} is non-degenerate with probability one. 

To illustrate the notion of a fold bifurcation point from Definition~\ref{def:fold}, we present the following example.

\begin{example}\label{example:foldbifurcation}
    Let $x=(x_1,x_2)\in\R^2$ and $y\in\R$. Consider $g(x,y)$ that is globally defined, but whose local behavior near $y=0$ (for $x_1\in[0,1]$, with $x_2$ arbitrary) is given by the expression
    \begin{align*}
        g(x,y)=(1-2x_1)y+(3x_1-2x_1^2)y^3.
    \end{align*}
    It is easy to compute that the stationary points of $g(x,y)$ with respect to $y$ near $y=0$ are
    \begin{align*}
        y^\ast(x)=\begin{cases}
            \emptyset\quad&\text{when}\quad 0\leq x_1<1/2\\
            \{0\}\quad&\text{when}\quad x_1=1/2\\
            \{\pm\sqrt{\frac{2x_1-1}{9x_1-6x_1^2}}\}\quad&\text{when}\quad 1/2<x_1\leq 1
        \end{cases}
    \end{align*}
    We check the three conditions in Definition~\ref{def:fold} at the stationary point $((1/2,x_2),0)$. 
    First, since there is only one $y$-variable, the Hessian $\nabla^2_{yy} g$ reduces to a scalar. At $(x_1,y)=(1/2,0)$ we obtain $\nabla^2_{yy} g=0$, which means the Hessian has exactly one zero eigenvalue, satisfying condition (1) in Definition~\ref{def:fold}. 
    Second, letting $v=1$ be the eigenvector corresponding to this zero eigenvalue, we compute
    $$
    \nabla_y g(x,y)=(1-2x_1)+3(3x_1-2x_1^2)y^2,
    $$
    and
    $$
    \frac{\partial}{\partial x_1}\nabla_y g(x,y)\big|_{(1/2,0)}=-2\neq 0,
    $$
    while the derivative with respect to $x_2$ vanishes. This shows that a perturbation of $x_1$ changes the stationarity condition along $v$, hence condition (2) in Definition~\ref{def:fold} is satisfied. 
    Finally, the third-order derivative at $(1/2,0)$ is
    \[
    \nabla^3_{yyy} g(1/2,0)=6\neq 0,
    \]
    verifying condition (3) in Definition~\ref{def:fold}. 
    Therefore, all three conditions in Definition~\ref{def:fold} hold, and $(1/2,x_2)$ is indeed a fold bifurcation point associated with the stationary point $y=0$.
    
    To gain intuition about the fold bifurcation of stationary points, we begin by examining the deformation of the function $g(x,y)$ as the parameter $x_1$ varies. Figure~\ref{fig:deformation} illustrates how the graph of $g(x,y)$ with respect to $y$ changes for different values of $x_1$. To further visualize the structure of the stationary point set, Figure~\ref{fig:3dfold} presents a 3D plot of the correspondence $x\mapsto y^\ast(x)$, where each point on the surface represents a stationary point of $g(x,y)$ with respect to $y$ near $y=0$.
    
    \begin{figure}[H]
    \centering
    \includegraphics[width=1\linewidth]{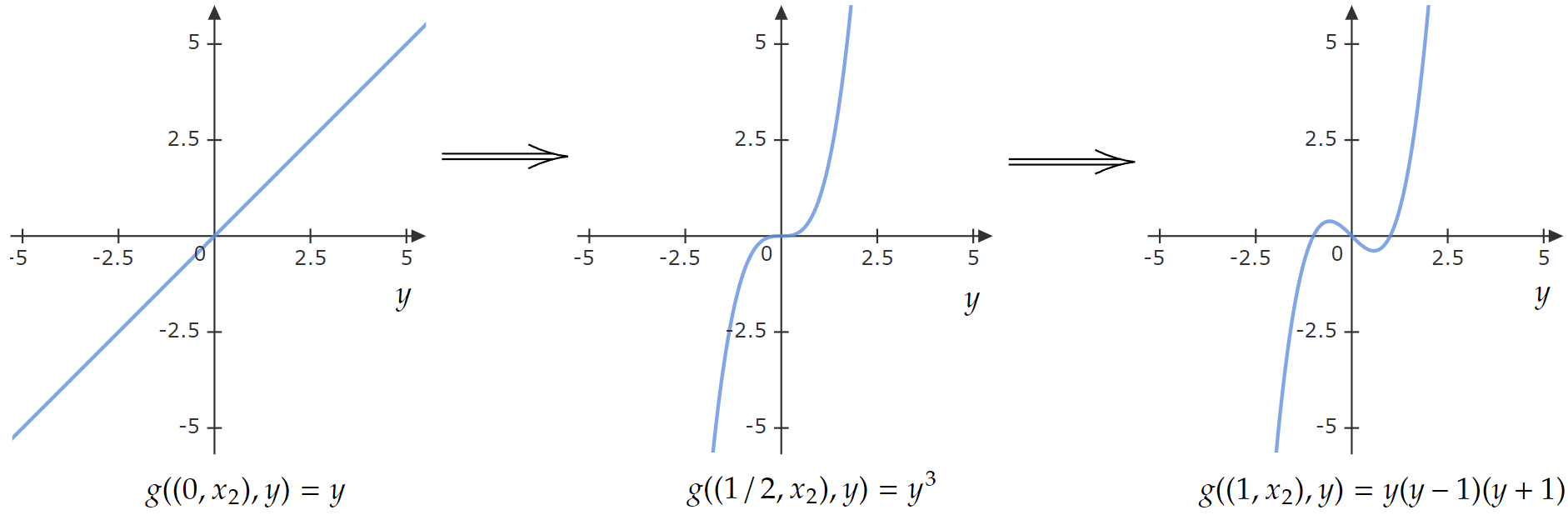}
    \caption{From left to right, the plots illustrate a typical fold bifurcation process in terms of the stationary point. When $x_1=0$, there is no stationary point near $y=0$. As the parameter increases, a degenerate stationary point emerges at $y=0$ when $x_1=1/2$. When $x_1=1$, this degenerate point has split into two non-degenerate stationary points near $y=0$, corresponding to one local minimum and one local maximum.}
    \label{fig:deformation}
\end{figure}

\begin{figure}[H]
    \centering
    \includegraphics[width=0.5\linewidth]{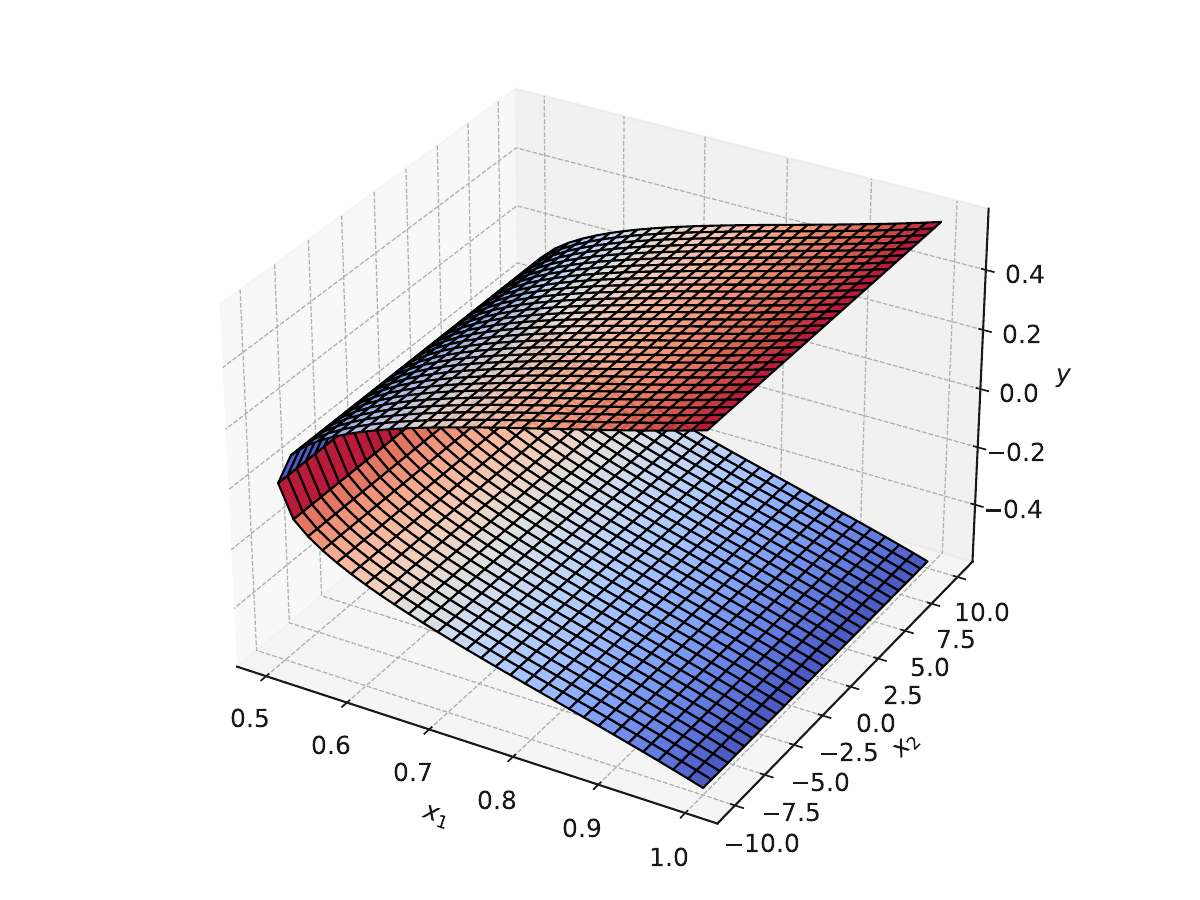}
    \caption{Visualization of the stationary point set $y^\ast(x)$ for $g(x,y)$.}
    \label{fig:3dfold}
\end{figure}

\end{example}

Under the assumption that all degenerate stationary points are fold bifurcation points, we obtain the following estimate for $\mu(\delta)$ and $\alpha(\delta,\mu(\delta)/(2\overline{\overline{L}}g))$:

\begin{theorem}\label{thm:globalversionmainthm}
    Suppose $g(x,y)$ is three times continuously differentiable, and all degenerate stationary points of $g(x,y)$ with respect to $y$ are fold bifurcation points defined in Definition~\ref{def:fold}, then there exist constants $D_1,D_2,D_3,D_4>0$ such that for any $\delta>0$
    \begin{align*}
    \mu(\delta)&\geq\min\{D_1\sqrt{\delta},D_2\}\\
        \alpha\left(\delta,\frac{\mu(\delta)}{2\overline{\overline{L}}_g}\right)&\geq\min\{D_3\delta,D_4\}
    \end{align*}
    where $\mu(\cdot)$ and $\mathcal{Y}'$ are from Lemma~\ref{lemma:eigenvalue} and $\alpha(\cdot,\cdot)$ is from Lemma~\ref{lemma:gradient_lower_bound}.
\end{theorem}

\begin{proof}
    See Appendix~\ref{proof:thm:globalversionmainthm}.
\end{proof}

\begin{remark}
    As shown in equation~\eqref{eq:jacobianofA} in the proof of Theorem~\ref{thm:globalversionmainthm}, under the assumption that all bifurcation points are fold bifurcation stationary points, the set of bifurcation points is in fact an $(n-1)$-dimensional manifold. Therefore, the Minkowski dimension in this case is $n-1$.
\end{remark}

With Theorem~\ref{thm:globalversionmainthm}, we can derive the following oracle complexity of Algorithm~\ref{alg:main}.

\begin{corollary}
Under the setting of Theorem~\ref{thm:biased_sgd_conv}, we assume that $g(x,y)$ is three times continuously differentiable, and all degenerate stationary points of $g(x,y)$ with respect to $y$ are fold bifurcation points (Definition~\ref{def:fold}). 
Let $\rho_t=1/(t+1)$, $\delta_t=(t+1)^{-2/(n-d)}$, $N_t=t+1$ and use a constant step size $\beta_t=\beta<\overline{L}_{\phi^{\text{cd}}_{\xi}}/2$. To achieve $\mathbb{E}\big[\|\widetilde{\Phi}_{\mathcal X,R}\|^{2}\big]\le \epsilon^{2}$, we have the following oracle complexity:
\begin{itemize}
    \item The total number of function-value oracle calls $Q_f(\epsilon)=\sum_{t=1}^{T}N_t$ satisfies $Q_f(\epsilon)=\widetilde{O}(\epsilon^{-4})$;
    \item The total number of gradient oracle calls $Q_{\nabla_y g}(\epsilon)=\sum_{t=1}^{T}K_tN_t$ satisfies $Q_{\nabla_y g}(\epsilon)=\widetilde{O}(\epsilon^{-10})$;
    \item The total number of Hessian oracle calls $Q_{\nabla^2_{yy} g}(\epsilon)=\sum_{t=1}^{T}K_tN_t$ satisfies $Q_{\nabla^2_{yy} g}(\epsilon)=\widetilde{O}(\epsilon^{-10})$.
\end{itemize}
Here, a \emph{function-value oracle call} refers to one evaluation of the upper-level objective $f(x,y)$, while \emph{gradient} and \emph{Hessian oracle calls} refer to evaluations of the lower-level derivatives $\nabla_y g(x,y)$ and $\nabla^2_{yy} g(x,y)$, respectively.
\end{corollary}

\section{Experiments}
\subsection{Numerical Example}\label{sec:numericalexample}

We consider the minimax problem
$$\min_{x\in\R^n}\max_{y\in\R^m}f(x,y),$$
which has a bilevel structure as follows:
$$\min_{x\in\R^n} f(x,y^\ast(x))\quad \st\quad y^\ast(x)\in\argmin_{y\in\R^m}-f(x,y).$$ 
In particular, nonconvex-nonconcave minimax problems are challenging, since standard algorithms such as gradient descent-ascent (\texttt{GDA}) are known to fail to converge, often getting trapped in cyclic behaviors~\citep{jin2020local}. 
Since nonconvex-nonconcave minimax problems are a special case of bilevel optimization with a nonconvex lower-level problem, we apply our algorithm \ouralgo{} to this setting and compare its behavior against \texttt{GDA}, particularly in situations where \texttt{GDA} fails to converge.

We construct the following nonconvex-nonconcave minimax problem:
\begin{align}
    \min_{x\in\R^n}\max_{y\in\R^m} f(x,y)=(x^2-y^2)\sin(x+y)+xy\sin(x-y),
\end{align}
The function $f(x,y)$ is neither convex in $x$ nor concave in $y$. We simulate the \texttt{GDA} flow using 15 initializations with random seeds from 0 to 14, and observe that 4 trajectories (with seeds 5, 10, 12, and 13) enter closed loops, indicating non-convergence. We apply our algorithm \ouralgo{} to this nonconvex-nonconcave problem, and set the algorithmic parameters as follows:
\begin{itemize}
    \item \textbf{Lower-level solver:} gradient descent method;
    \item \textbf{Outer loop step size:} $\beta = 0.005$;
    \item \textbf{Inner loop step size:} $\eta = 0.01$;
    \item \textbf{Number of outer iterations:} $T=10000$;
    \item \textbf{Number of inner iterations:} $K=200$;
    \item \textbf{Number of samples used to approximate the hyperfunction:} $N=3$.
\end{itemize}

Unlike \texttt{GDA}, which often exhibits non-convergent cyclic behaviors, our algorithm converges from all 15 random initializations. 
In particular, our algorithm converges either to a local minimizer of a continuous region of $\phi^{\text{cd}}(x)$, or, when $\phi^{\text{cd}}(x)$ is discontinuous, to the side of the discontinuity with the smaller function value.

Figure~\ref{fig:gda_closed_loops} illustrates the comparison between the \texttt{GDA} flow and the \ouralgo{} trajectory for seeds 5, 10, 12, and 13, where \texttt{GDA} fails to converge while our method successfully converges. The \ouralgo{} trajectory shown corresponds to the sequence
\begin{align*}
    (x_0,y_0),\ (x_0,\hat{y}(x_0)),\ (x_1,\hat{y}(x_1)),\ (x_2,\hat{y}(x_2)),\ \cdots,\ (x_T,\hat{y}(x_T)),
\end{align*}
where $\hat{y}(x)$ denotes an inexact solution of the lower-level problem $g(x,\cdot)$ obtained by running $K=200$ steps of gradient descent with step size~$\eta=0.01$. 
We mark in Figure~\ref{fig:gda_closed_loops} the point among the last 100 iterations that achieves the smallest function value, referred to as the “best of last 100.” 
We also mark the stationary point, i.e., a point at which the gradients of $f$ with respect to both $x$ and $y$ vanish. 

In addition, Figure~\ref{fig:pcd} plots $\phi^{\text{cd}}(x)$ for seeds 5, 10, 12, and 13, where the initial point $y_0$ used in its definition is set to the $y$-component of the initialization for each seed. See Appendix~\ref{sec:exp} for the plots corresponding to the remaining seeds.

In summary, the experiments confirm that our algorithm \ouralgo{} consistently converges in nonconvex-nonconcave minimax settings, highlighting its robustness compared with \texttt{GDA}.

\begin{figure}[ht]
  \centering
  \begin{minipage}[b]{0.45\textwidth}
    \centering
    \includegraphics[width=\textwidth]{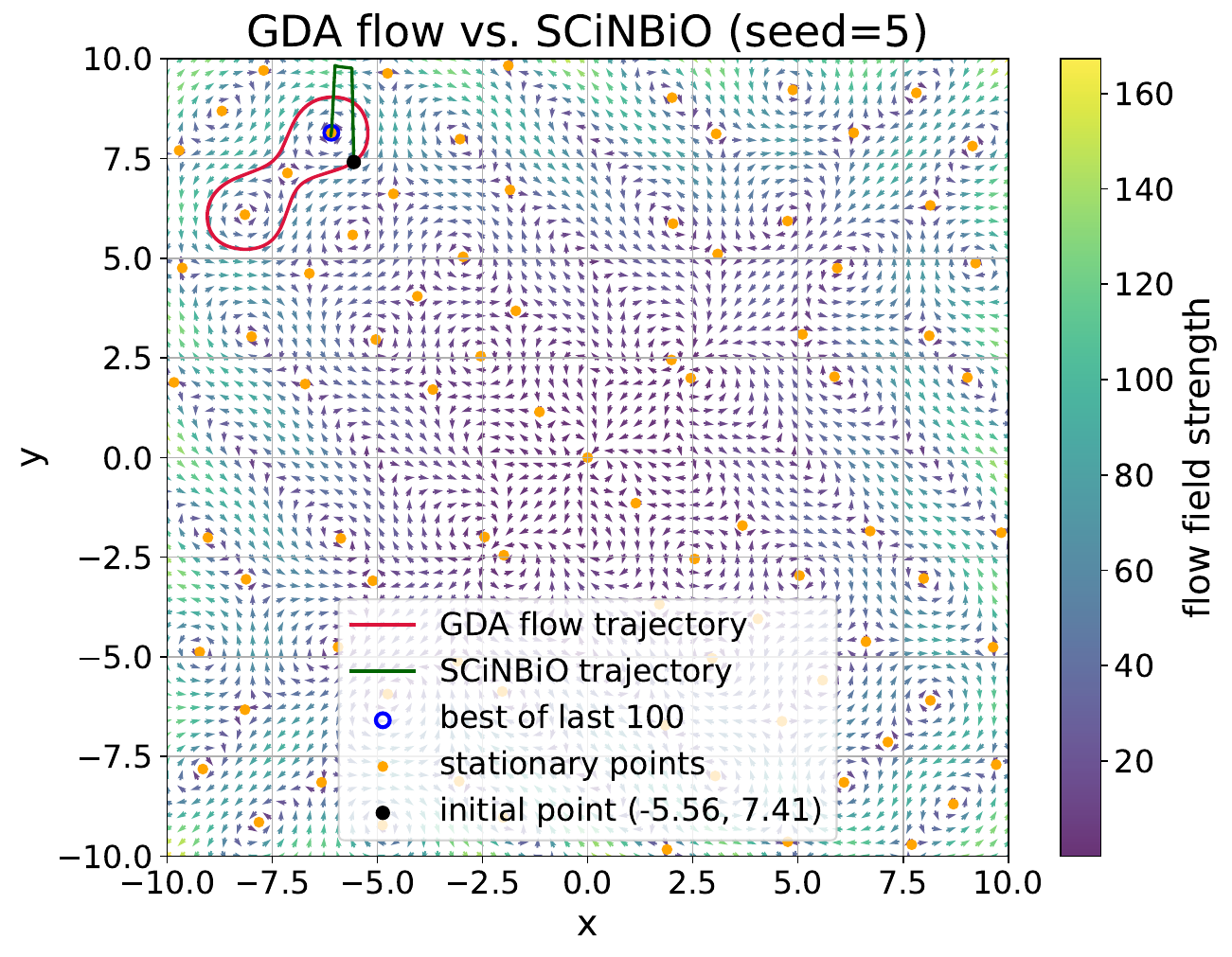}
    \par\small(a) Seed = 5
  \end{minipage}
  \hfill
  \begin{minipage}[b]{0.45\textwidth}
    \centering
    \includegraphics[width=\textwidth]{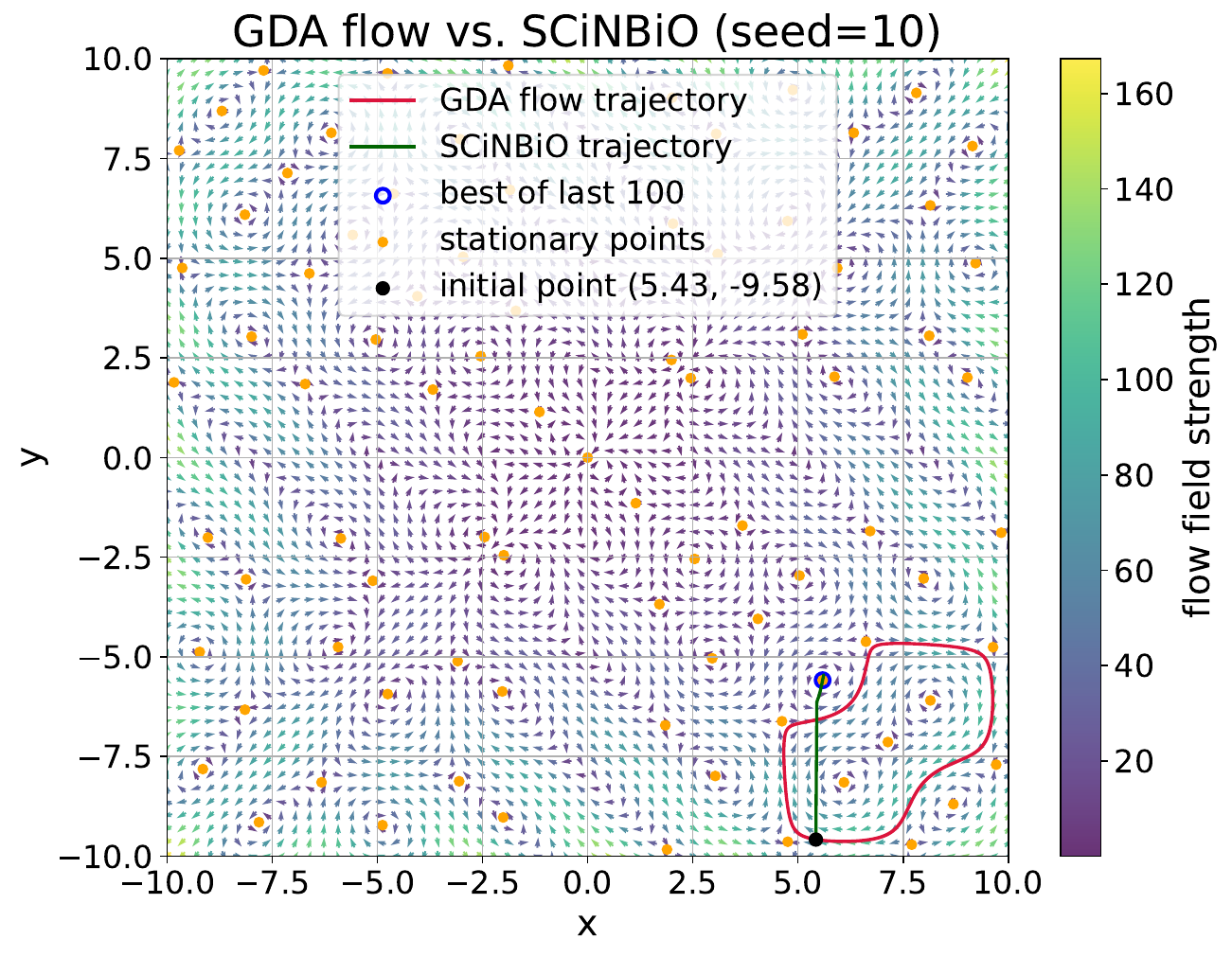}
    \par\small(b) Seed = 10
  \end{minipage}

  \vskip\baselineskip

  \begin{minipage}[b]{0.45\textwidth}
    \centering
    \includegraphics[width=\textwidth]{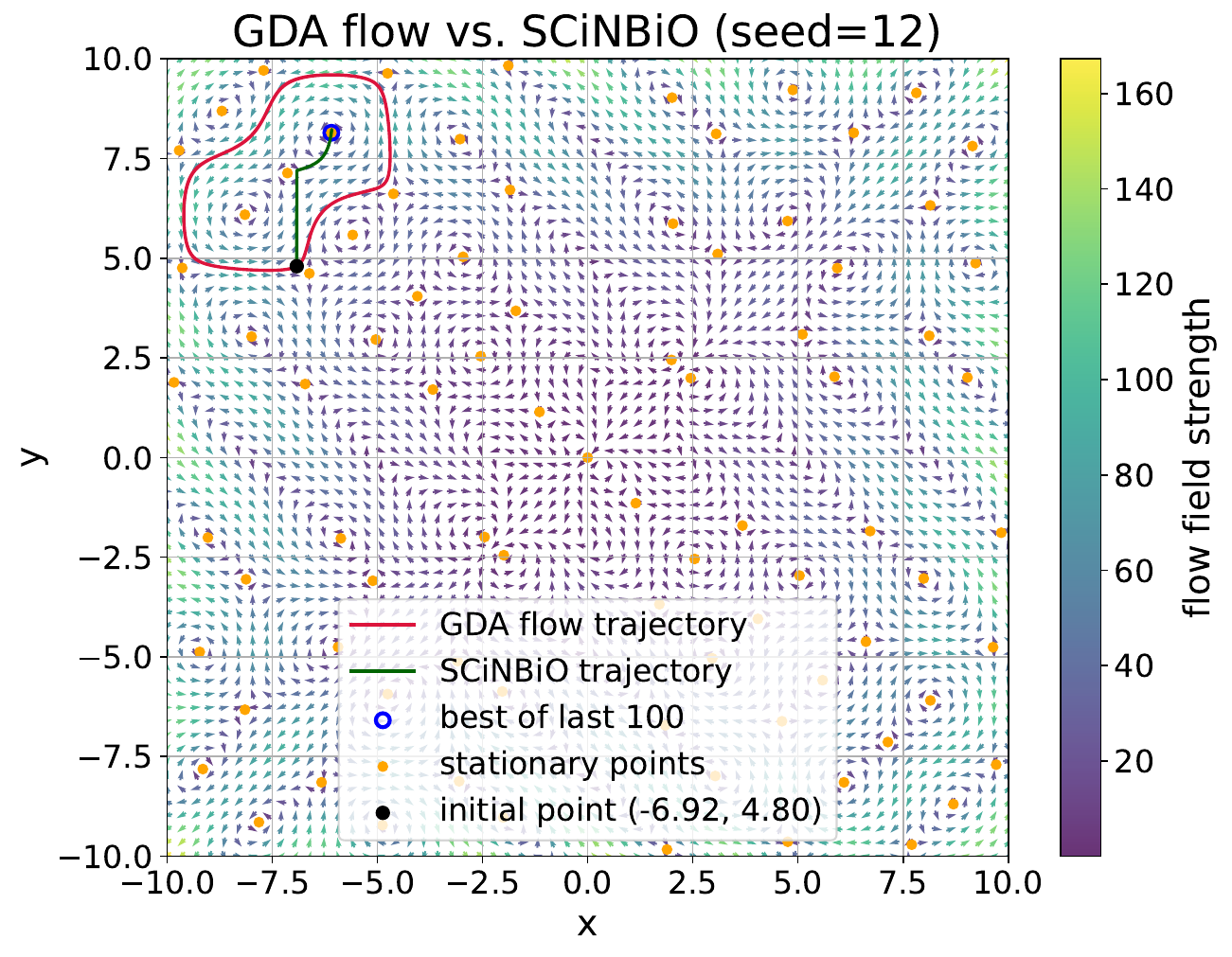}
    \par\small(c) Seed = 12
  \end{minipage}
  \hfill
  \begin{minipage}[b]{0.45\textwidth}
    \centering
    \includegraphics[width=\textwidth]{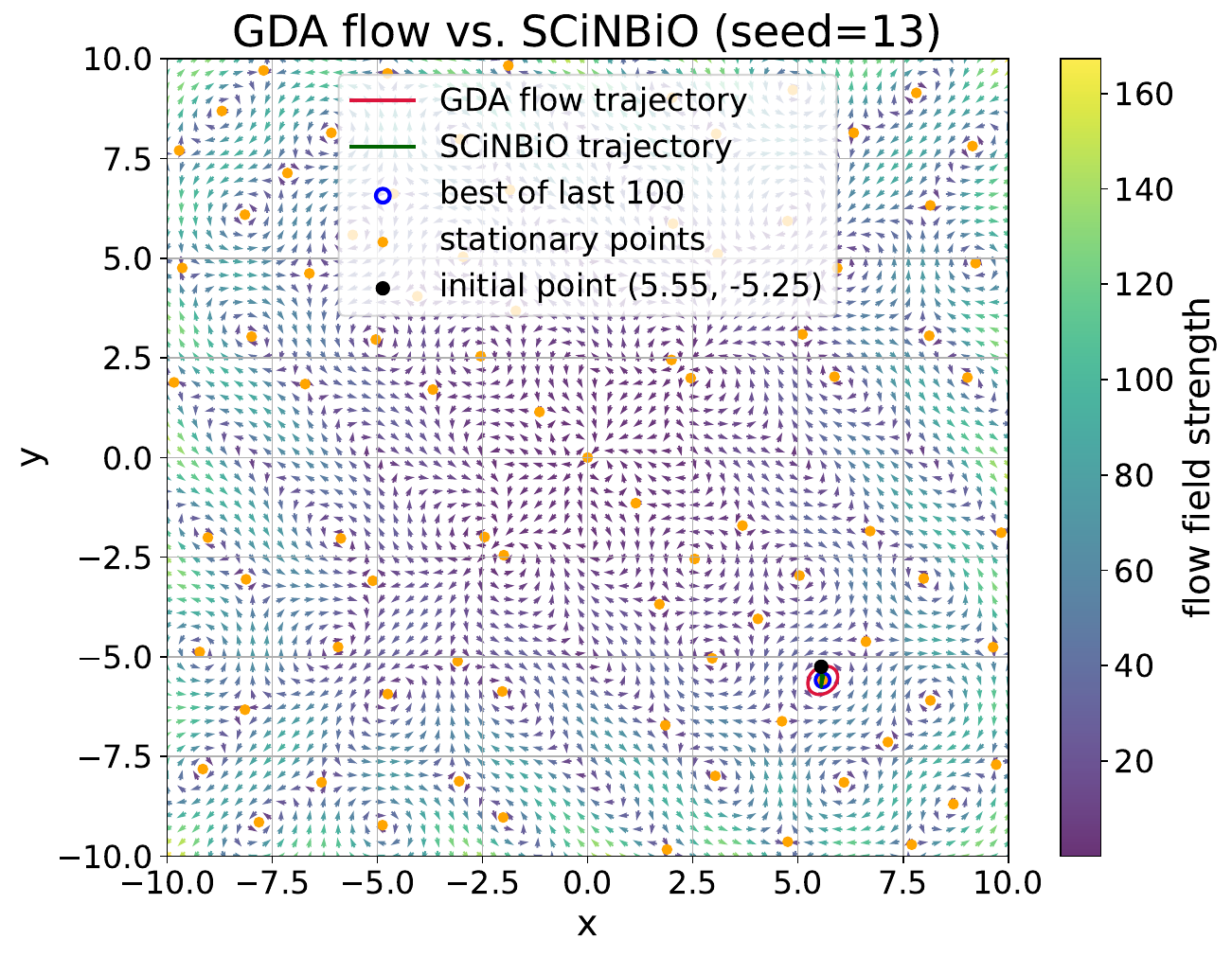}
    \par\small(d) Seed = 13
  \end{minipage}

  \caption{Closed-loop trajectories are observed under \texttt{GDA} for seeds 5, 10, 12, and 13, compared with the convergent behavior of \ouralgo{}. The arrows indicate the direction of the \texttt{GDA} vector field $(\nabla_x f(x,y), -\nabla_y f(x,y))$, and the color of each arrow represents its magnitude: lighter shades indicate larger vector norms, while darker shades indicate smaller ones.}
  \label{fig:gda_closed_loops}
\end{figure}

\begin{figure}[ht]
  \centering
  \begin{minipage}[b]{0.45\textwidth}
    \centering
    \includegraphics[width=\textwidth]{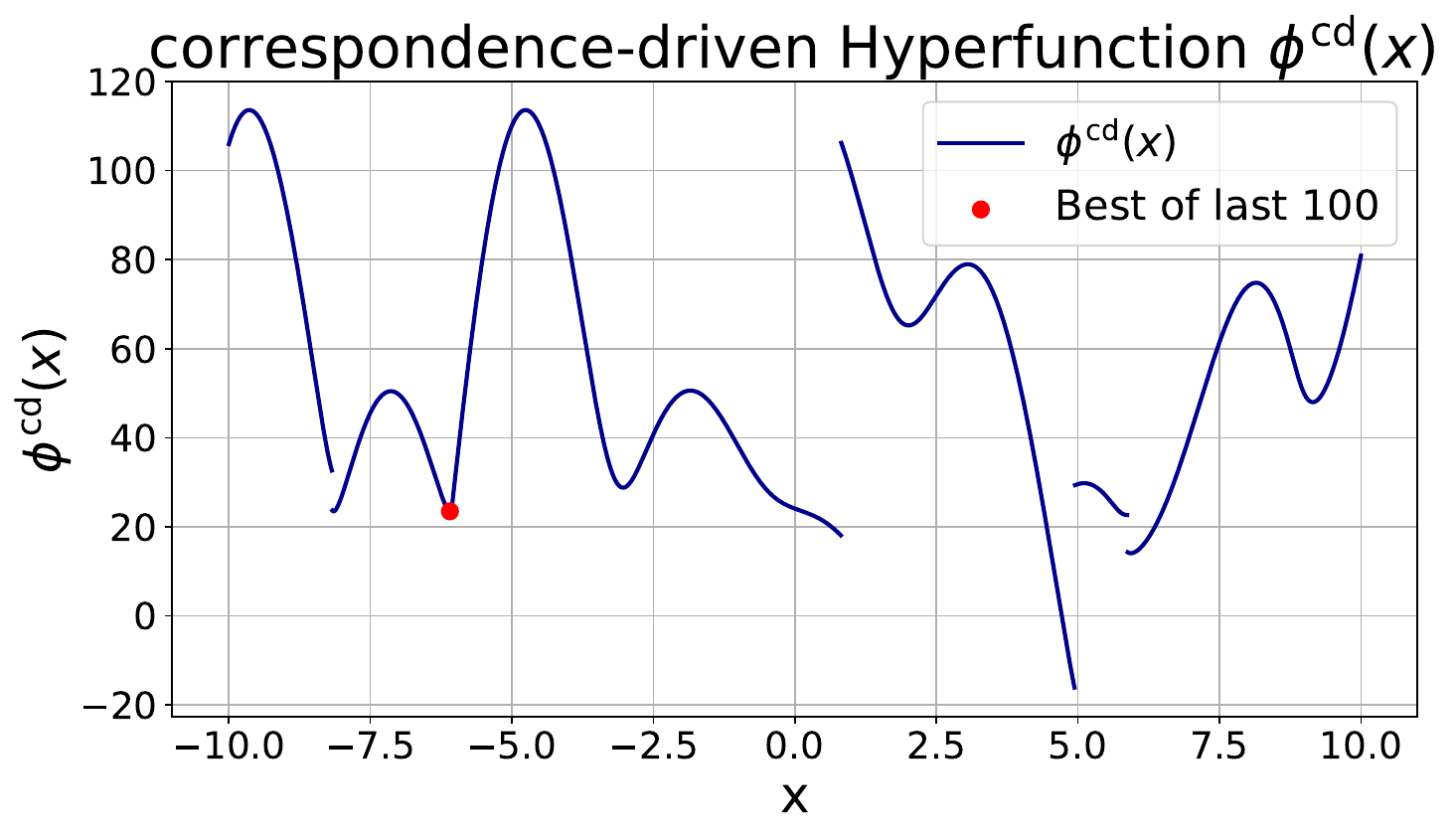}
    \par\small(a) Seed = 5
  \end{minipage}
  \hfill
  \begin{minipage}[b]{0.45\textwidth}
    \centering
    \includegraphics[width=\textwidth]{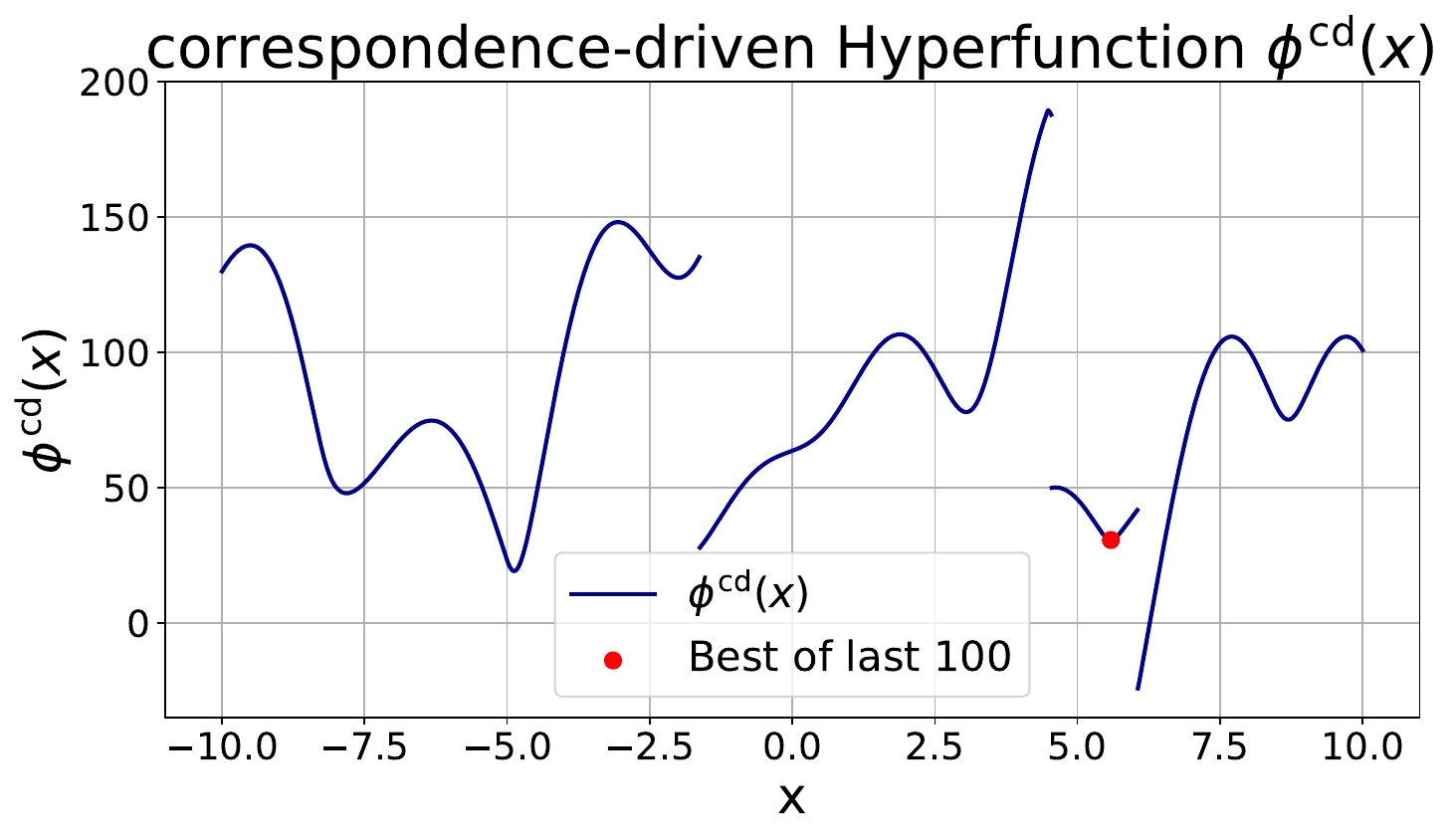}
    \par\small(b) Seed = 10
  \end{minipage}

  \vskip\baselineskip

  \begin{minipage}[b]{0.45\textwidth}
    \centering
    \includegraphics[width=\textwidth]{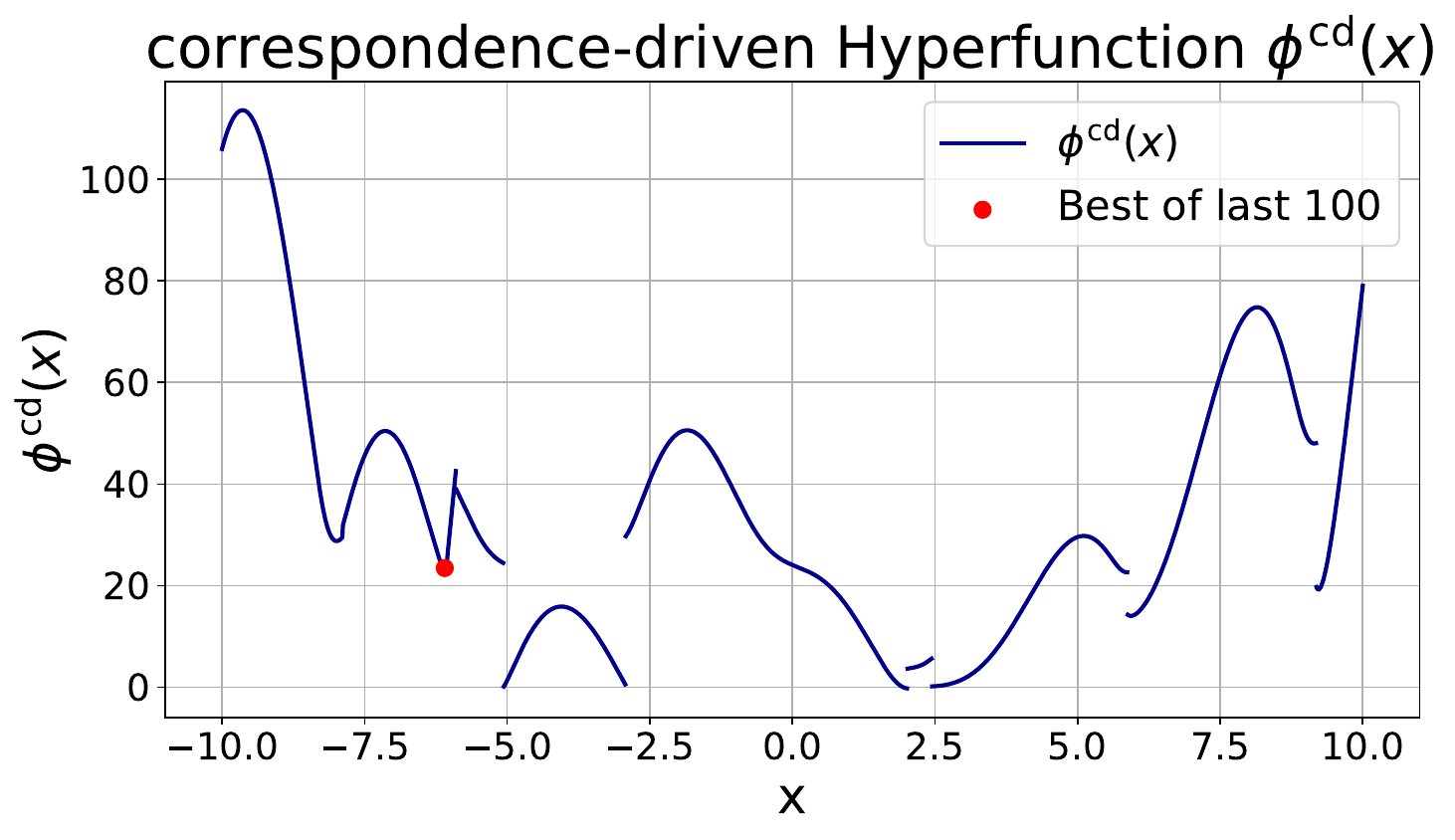}
    \par\small(c) Seed = 12
  \end{minipage}
  \hfill
  \begin{minipage}[b]{0.45\textwidth}
    \centering
    \includegraphics[width=\textwidth]{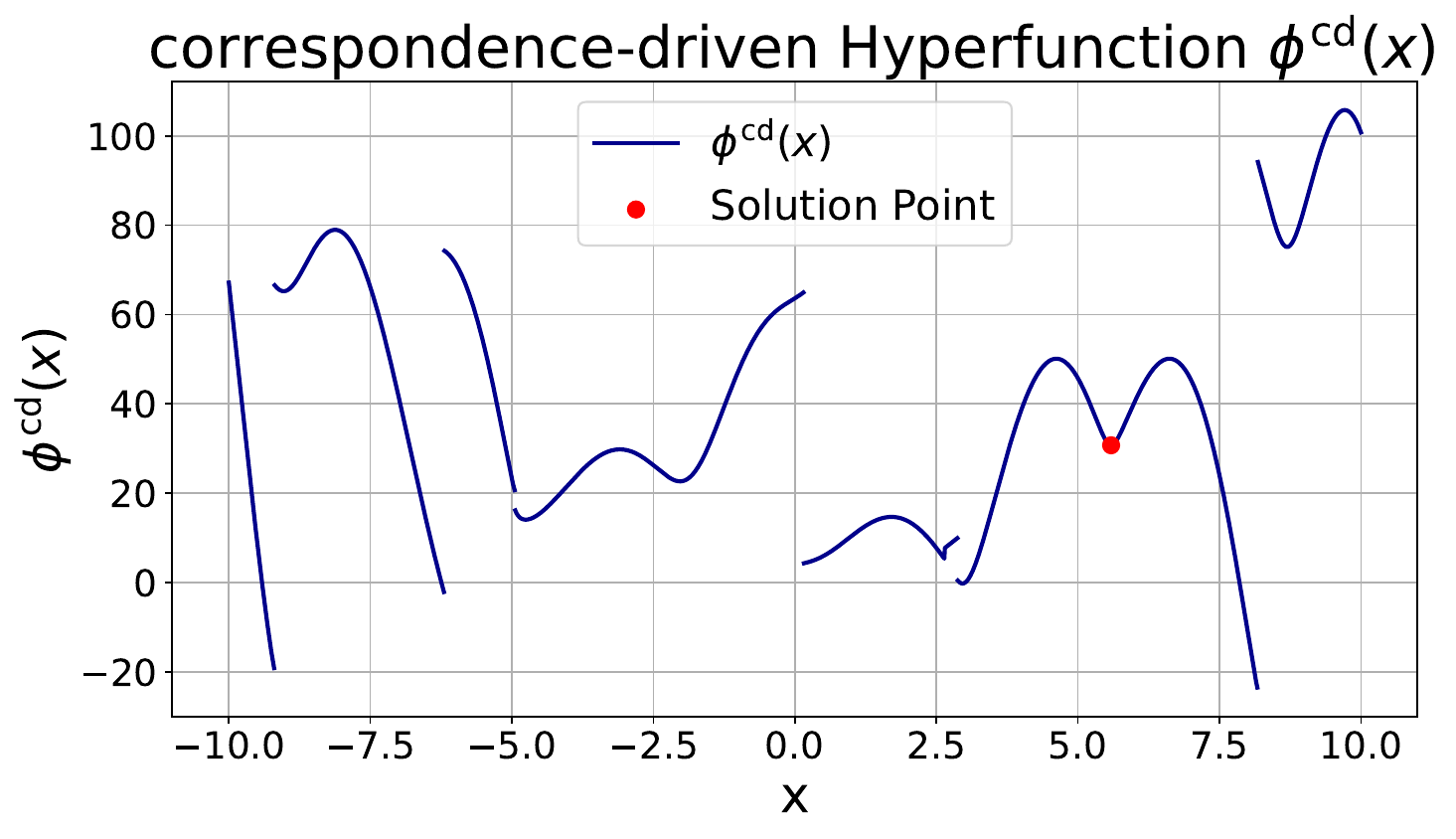}
    \par\small(d) Seed = 13
  \end{minipage}

  \caption{Plot of the \emph{correspondence-driven hyperfunction} $\phi^{\mathrm{cd}}(x)$ for seeds 5, 10, 12, and 13. In these four cases, \ouralgo{} successfully converges to a local minimum of $\phi^{\mathrm{cd}}(x)$.}
  \label{fig:pcd}
\end{figure}

\subsection{Hyperparameter Optimization}

In this section, we evaluate the performance of \ouralgo{} against \texttt{V-PBGD}~\citep{Onpenaltybasedbilevelgradientdescentmethod} and \texttt{BOME}~\citep{liu2022bome} on a bilevel hyperparameter optimization task. The goal is to tune a nonnegative $\ell_2$-regularization parameter $\lambda$ for a neural network.

We model the problem as choosing $\lambda$ to minimize the validation loss, subject to the network weights $\theta$ being optimized on the training set. To enforce $\lambda \ge 0$, we parameterize it as $\lambda=\mathrm{softplus}(x)=\log(1+e^x)$ with $x\in\mathbb{R}$. Formally, we solve:
\begin{subequations}\label{eq:mnist}
\begin{align}
\min_{\lambda\ge 0}\quad f(\theta^\ast(\lambda),\lambda)
&:=\mathrm{CE}_{\mathrm{val}}(\theta^\ast(\lambda))\label{eq:mnist_upper}\\
\text{s.t.}\quad 
\theta^\ast(\lambda)\in\argmin_{\theta}\quad g(\theta,\lambda)
&:=\mathrm{CE}_{\mathrm{train}}(\theta)+\tfrac{\lambda}{2B}\|\theta\|_2^2,\label{eq:mnist_lower}
\end{align}
\end{subequations}
where $\mathrm{CE}{\mathrm{train}}$ and $\mathrm{CE}{\mathrm{val}}$ denote the cross-entropy loss on training and validation sets, respectively, and $B$ is the training batch size. The cross-entropy loss for a dataset $\{(x_i,y_i)\}_{i=1}^n$ is defined as:
\begin{align}
\mathrm{CE}(\theta;{(x_i,y_i)}_{i=1}^n)
=-\frac{1}{n}\sum_{i=1}^n
\log\left(\frac{\exp(z_\theta(x_i){y_i})}{\sum_{j=1}^{10}\exp(z_\theta(x_i)j)}\right),
\end{align}
where $z_\theta(x_i)$ represents the classifier output.

We use the MNIST dataset with a split of $B=50000$ training examples for the lower-level problem~\eqref{eq:mnist_lower} and $10000$ validation examples for the upper-level objective~\eqref{eq:mnist_upper}. The classifier $z_\theta(\cdot)$ is a Convolutional Neural Network (CNN). To ensure the mapping $\theta\mapsto z_\theta(\cdot)$ remains differentiable, we employ SiLU activations and average pooling. The specific architecture is:
\begin{itemize}
\item Conv(1$\to$32,$3{\times}3$,pad=1)$\to$GroupNorm(8)$\to$SiLU$\to$AvgPool2d(2)
\item Conv(32$\to$64,$3{\times}3$,pad=1)$\to$GroupNorm(8)$\to$SiLU$\to$AvgPool2d(2)
\item Conv(64$\to$128,$3{\times}3$,pad=1)$\to$GroupNorm(8)$\to$SiLU$\to$AvgPool2d(2)
\item Flatten$\to$Linear(128$\cdot3\cdot3\to$256)$\to$GroupNorm(16)$\to$SiLU$\to$Linear(256$\to$10)
\end{itemize}
This deep architecture creates a highly nonconvex lower-level landscape.

We apply \ouralgo{} (our proposed Algorithm~\ref{alg:main}), \texttt{V-PBGD}~\citep{Onpenaltybasedbilevelgradientdescentmethod}, and \texttt{BOME}~\citep{liu2022bome} to the bilevel hyperparameter optimization~\eqref{eq:mnist_upper}. The specific experimental settings and algorithmic parameters are as follows: for all three algorithms, the lower-level optimization is performed using \texttt{SGD} with a step size $\eta=0.1$, a batch size of $B = 2048$, and $K=100$ iterations per inner solve. The additional parameters for each algorithm are specified below:

\begin{itemize}
\item \textbf{\ouralgo{}:} outer loop epoch number $T=100$, warm-started inner solves (i.e., initialize each inner-loop run at the final iterate from the previous inner-loop run, rather than restarting from the same fixed $y_0$ in Algorithm~\ref{alg:main}), outer loop step size $\beta_t = 0.001/\sqrt{t}$ using square-root decay, number of samples $N_{\mathrm{dirs}}=3$ for upper-level gradient estimator, smoothing radius $\xi=0.1$;
\item \texttt{V-PBGD:} outer loop epoch number $T = 300$, outer step size $\beta=0.01$, penalty parameter $\rho_t = \min\{1.05^{t-1},10\}$;
\item \texttt{BOME:} outer loop epoch number $T = 300$, outer loop step size $\eta = 0.1$.
\end{itemize}

To evaluate the performance, we measure both the upper-level validation loss and the lower-level objective, tracking the convergence over epochs. In Figure \ref{fig:hyper_exp}, we present the performance of \ouralgo{}, \texttt{V-PBGD}, and \texttt{BOME} on this hyperparameter optimization task, evaluated over 30 random seeds. The left plot and the right plot represent the upper-level validation loss and the lower-level objective, respectively, both plotted against the number of epochs. Since \ouralgo{} performs $N_{\mathrm{dirs}}=3$ inner-loop cycles for each outer-loop iteration, the plots are shown against the number of inner-loop epochs to provide a consistent comparison across the algorithms. The deep-colored curves represent the average performance across the 30 seeds for each algorithm, while the light-colored curves show the individual results from each of the 30 seeds. The results, shown in Figure \ref{fig:hyper_exp}, highlight the superior performance of \ouralgo{}. As seen in the figure, \ouralgo{} consistently achieves faster convergence and more stable optimization compared to \texttt{V-PBGD} and \texttt{BOME}, with strong performance in both the upper-level and lower-level tasks.

\begin{figure}
    \centering
    \includegraphics[width=0.9\linewidth]{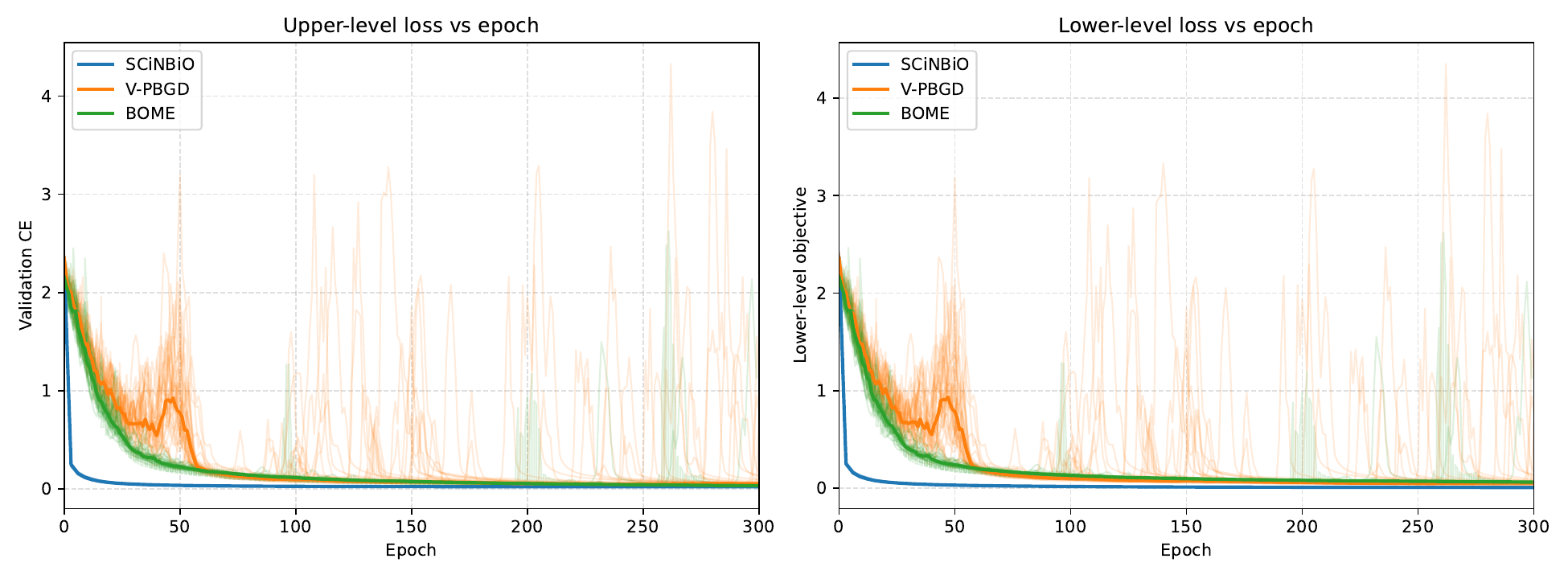}
    \caption{Plot of the upper-level \& lower-level function vs epoch for seeds 0 to 29}
  \label{fig:hyper_exp}
\end{figure}

\section{Conclusion}
In this paper, we proposed a \emph{correspondence-driven hyperfunction} formulation for bilevel optimization problems with nonconvex lower-level objectives. The proposed hyperfunction models the follower’s behavior by selecting among stationary points reachable by a fixed algorithm with a given initialization and step size, making it more meaningful and computationally tractable than the classical definition. To address its discontinuity, we applied Gaussian smoothing and established the convergence of the smoothed hyperfunction’s value, gradient, and proximal gradient to those of the original hyperfunction at appropriate points.  

We identified bifurcation phenomena as a key challenge for hyperfunction-based algorithms in the nonconvex setting, introduced the notion of prevalent assumptions, and proved that the property ``for almost every $x$, $g(x,\cdot)$ is Morse'' is prevalent under certain small perturbations. Under this assumption, we analyzed the geometric structure of the bifurcation set, and further connected bifurcation theory from dynamical systems to the bilevel setting by defining fold bifurcation points.  

Building on these results, we designed a biased projected \texttt{SGD}-based algorithm \ouralgo{} with a cubic-regularized Newton lower-level solver, and provided convergence guarantees together with oracle complexity bounds for the upper-level. Under the additional assumption that all degenerate stationary points are fold bifurcation points, we further obtained the lower-level oracle complexity of \ouralgo{}.  

Through numerical experiments, we demonstrated the effectiveness of \ouralgo{} in both nonconvex-nonconcave minimax problems and bilevel hyperparameter optimization tasks. In the nonconvex-nonconcave minimax setting, \ouralgo{} consistently converged from all random initializations, unlike \texttt{GDA}, which exhibited cyclic behaviors and failed to converge in several cases. For bilevel hyperparameter optimization, \ouralgo{} outperformed \texttt{V-PBGD} and \texttt{BOME} in terms of both convergence speed and stability, demonstrating its superior performance in solving bilevel optimization problems with nonconvex lower-level objectives.

Our work develops new modeling and algorithmic approaches for bilevel optimization with nonconvex lower-levels under minimal assumptions, offering practical solution methods together with theoretical insights and provable performance guarantees.
 Future work will aim to extend these results to settings with stochastic objectives or different lower-level solvers.

\clearpage
\bibliographystyle{abbrvnat} 
\bibliography{reference}

\clearpage
\appendix
\section*{Appendix}\

\section{Full Results of Numerical Example}\label{sec:exp}

In this section, we present the plots of the \emph{correspondence-driven hyperfunction} $\phi^{\mathrm{cd}}(x)$ for the remaining random seeds of the numerical example in Section~\ref{sec:numericalexample}, namely, seeds $0$--$14$ excluding $5$, $10$, $12$, and $13$.

\begin{figure*}[htbp]
  \centering

  \begin{minipage}[b]{0.3\textwidth}
    \centering
    \includegraphics[width=\textwidth]{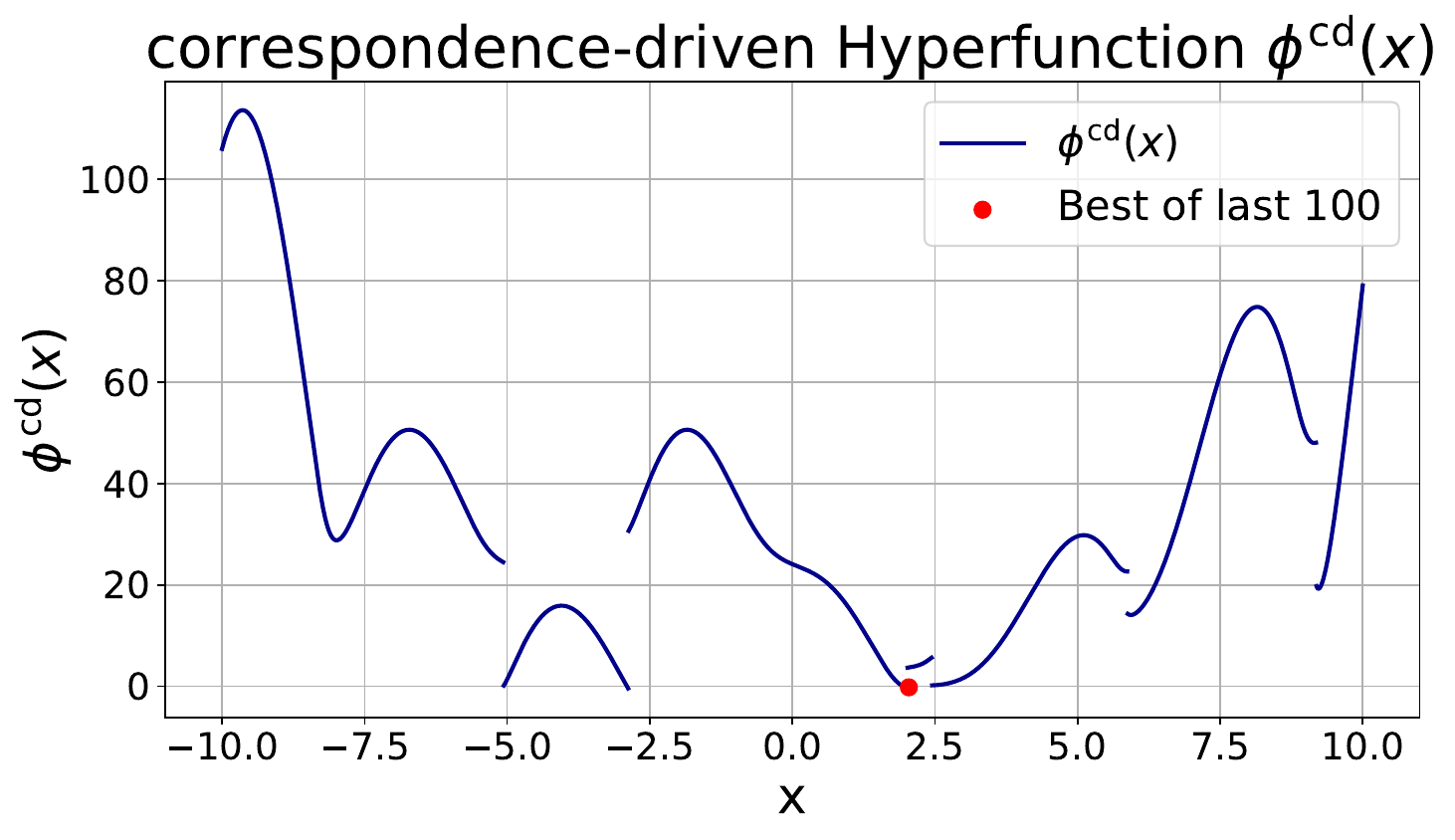}
    \par\small(a) Seed = 0
  \end{minipage}
  \hfill
  \begin{minipage}[b]{0.3\textwidth}
    \centering
    \includegraphics[width=\textwidth]{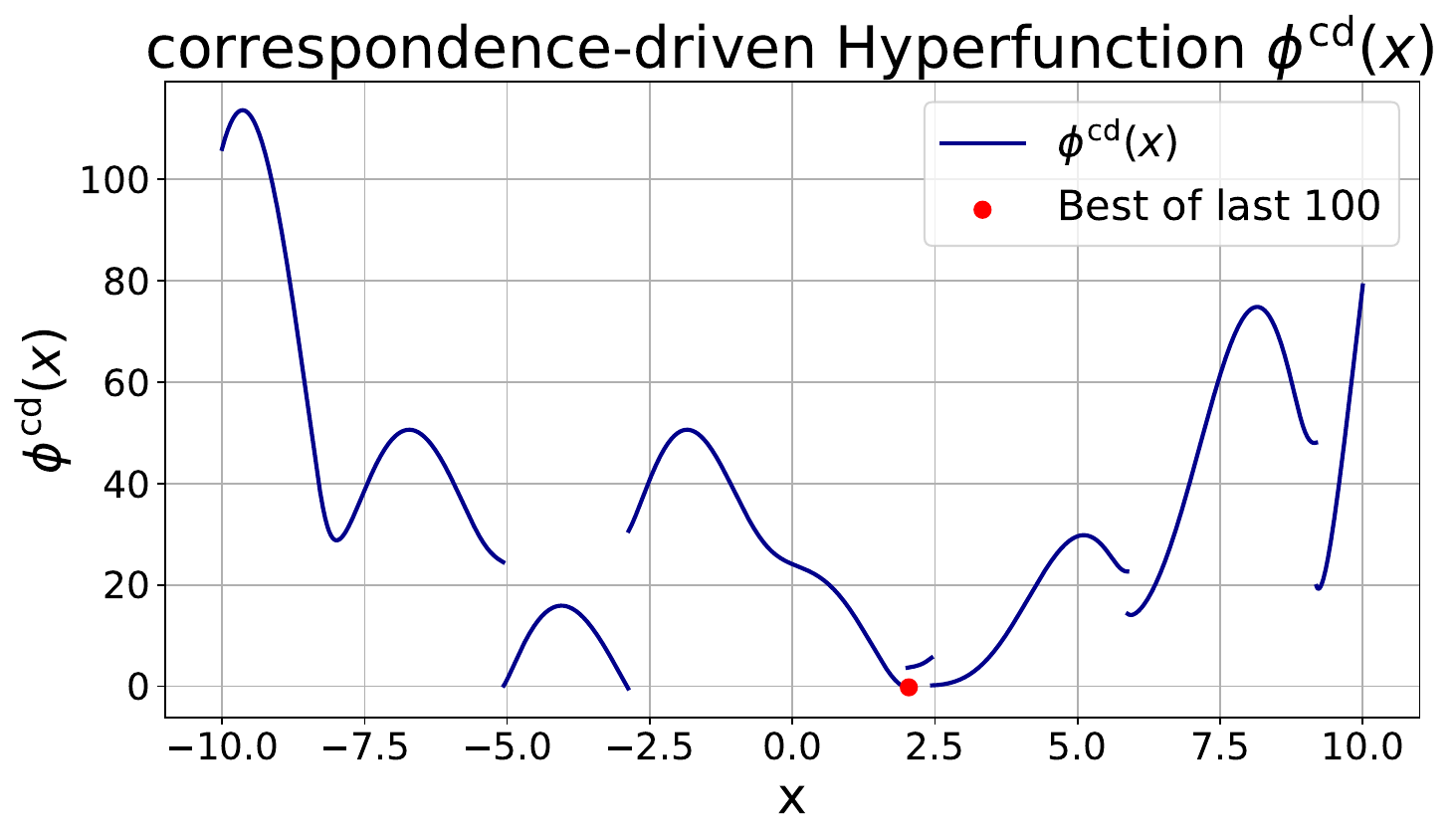}
    \par\small(b) Seed = 1
  \end{minipage}
  \hfill
  \begin{minipage}[b]{0.3\textwidth}
    \centering
    \includegraphics[width=\textwidth]{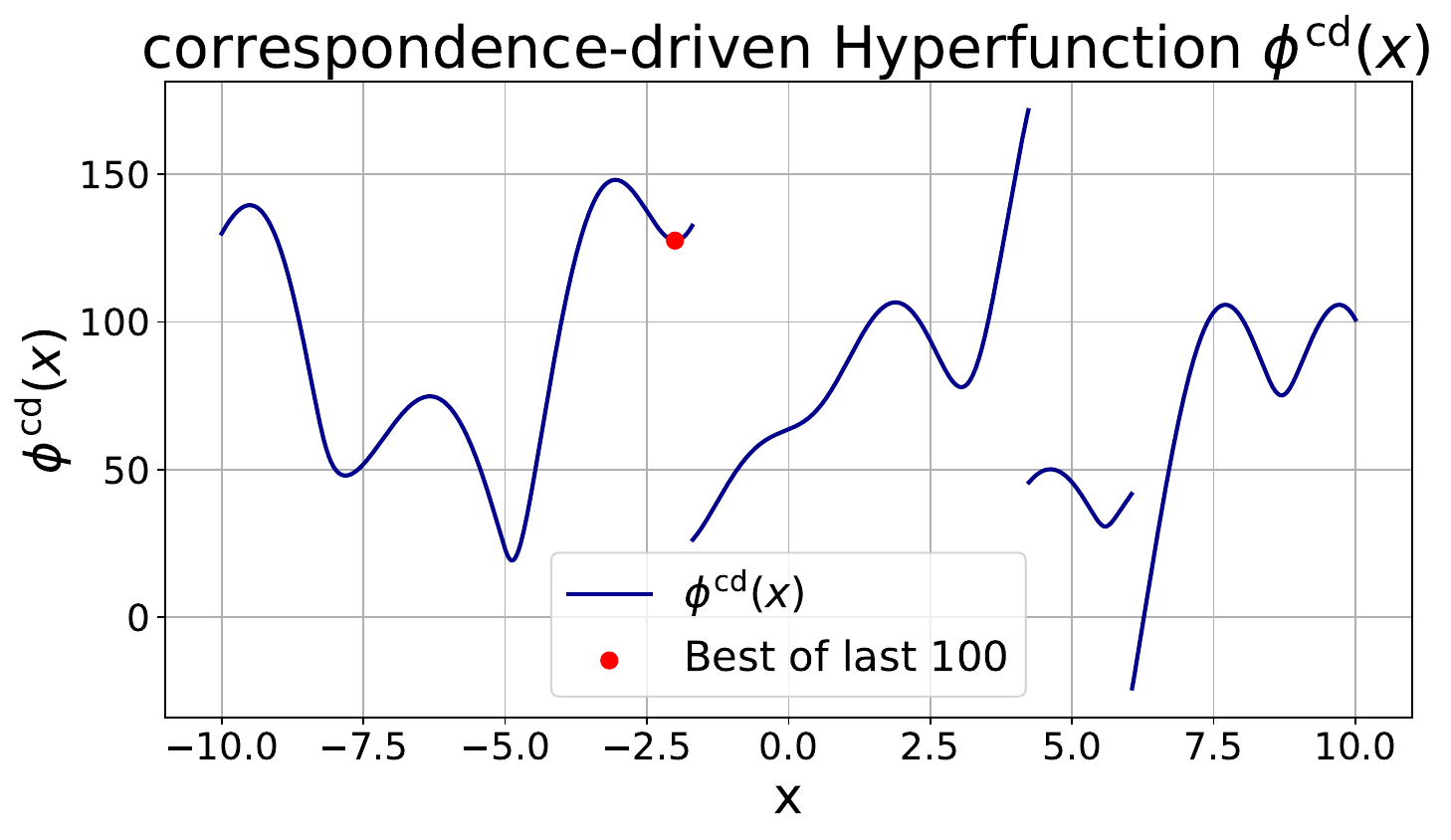}
    \par\small(c) Seed = 2
  \end{minipage}

  \vskip\baselineskip

  \begin{minipage}[b]{0.3\textwidth}
    \centering
    \includegraphics[width=\textwidth]{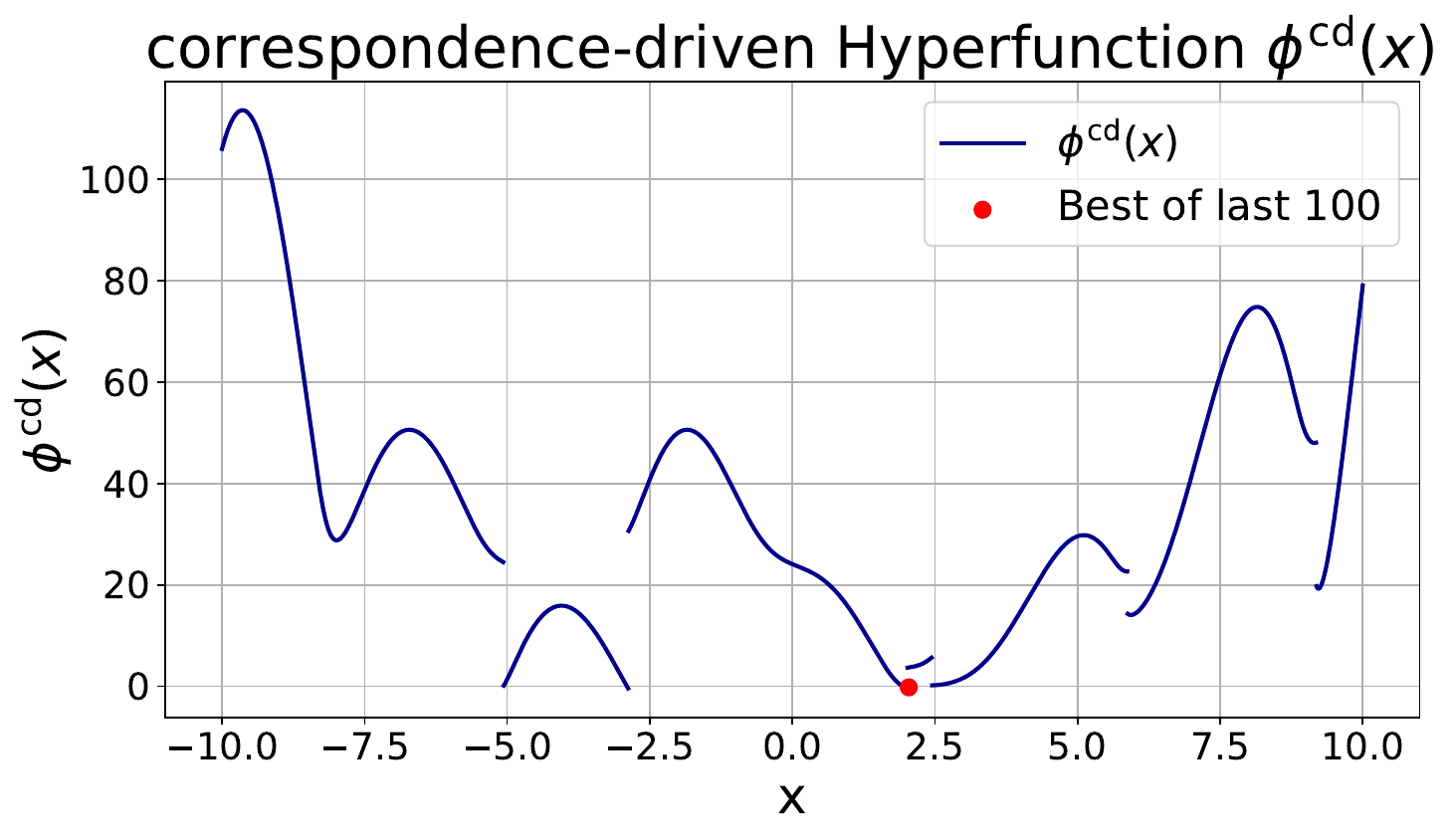}
    \par\small(d) Seed = 3
  \end{minipage}
  \hfill
  \begin{minipage}[b]{0.3\textwidth}
    \centering
    \includegraphics[width=\textwidth]{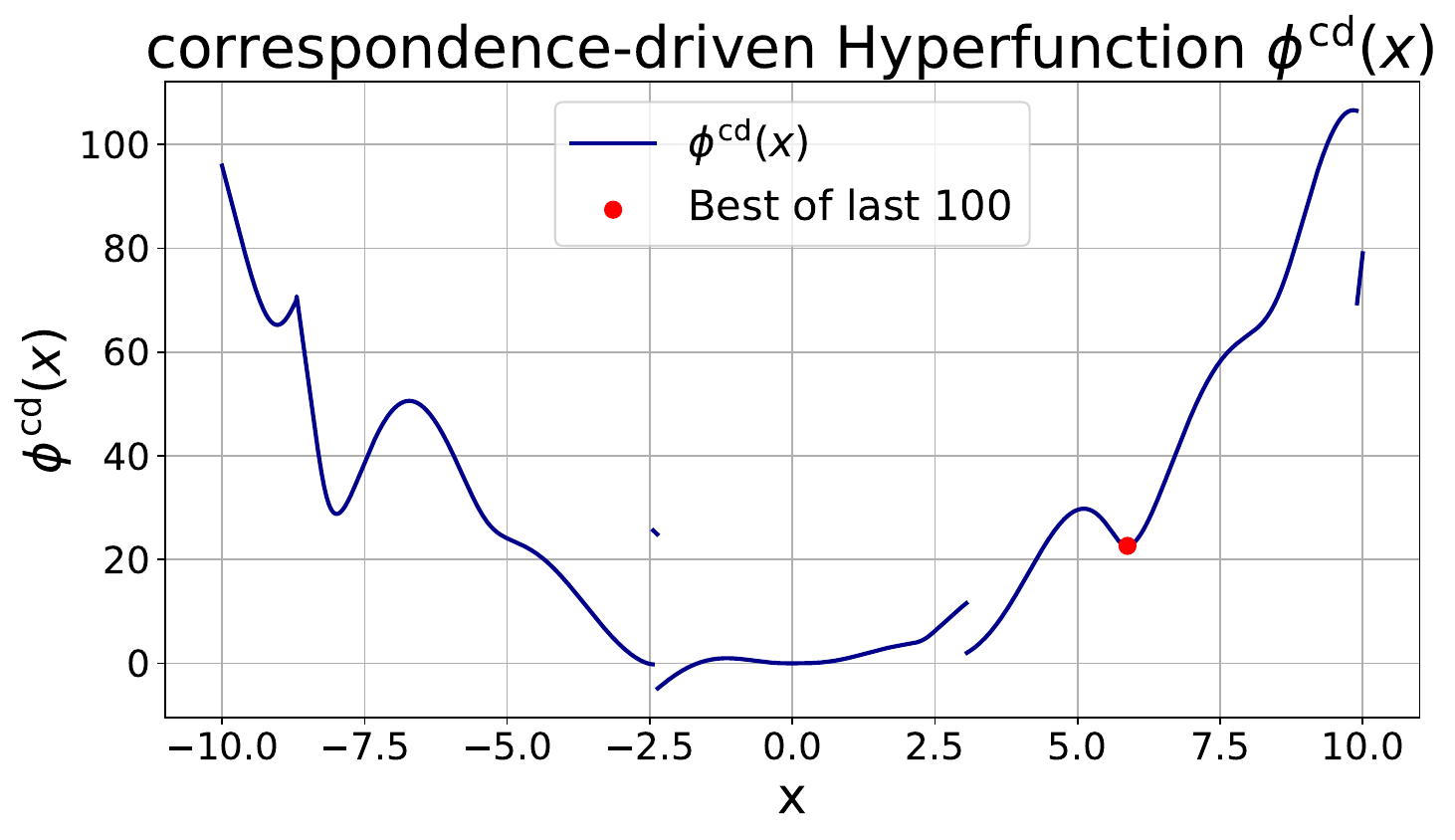}
    \par\small(e) Seed = 4
  \end{minipage}
  \hfill
  \begin{minipage}[b]{0.3\textwidth}
    \centering
    \includegraphics[width=\textwidth]{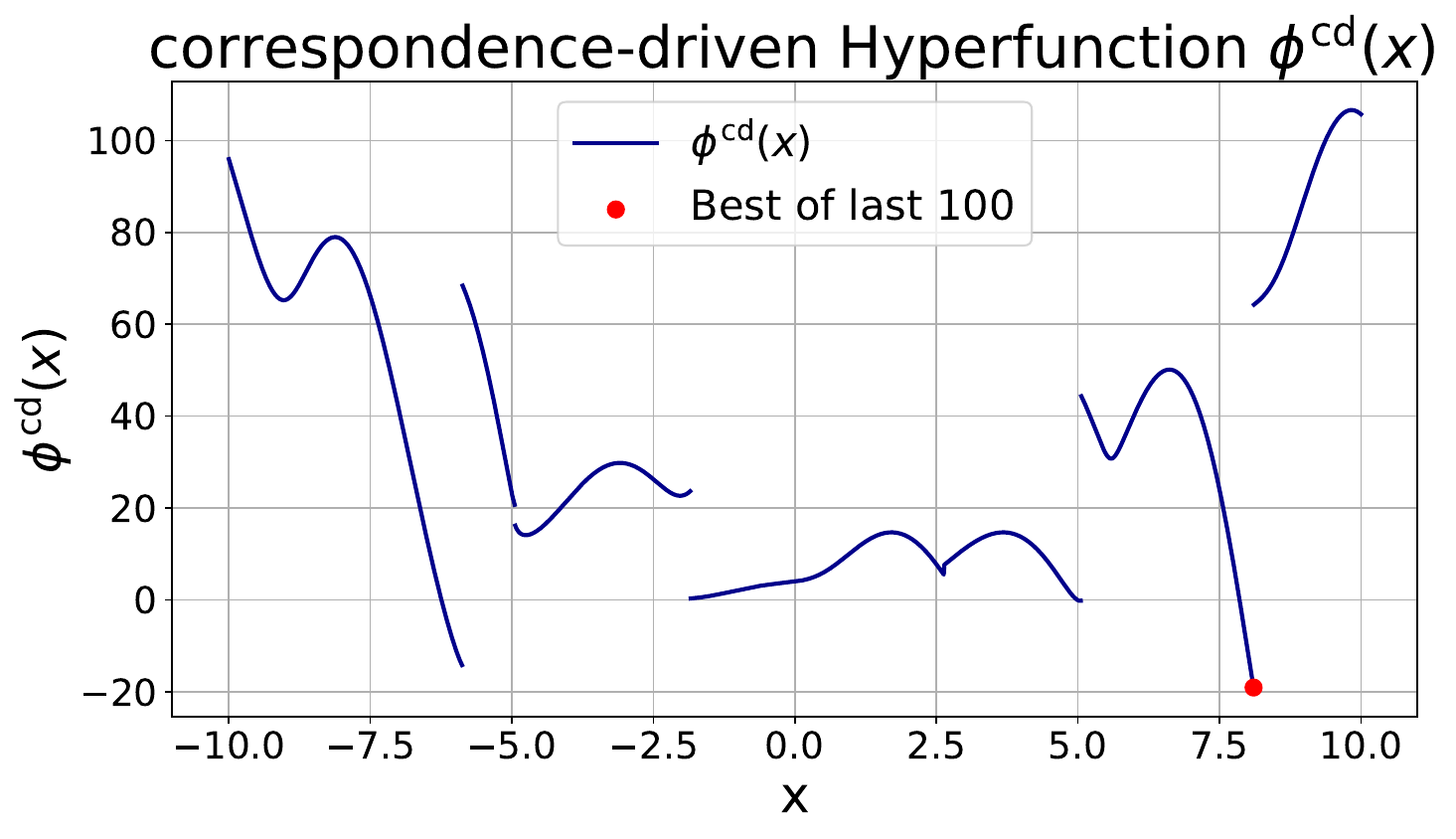}
    \par\small(f) Seed = 6
  \end{minipage}

  \vskip\baselineskip

  \begin{minipage}[b]{0.3\textwidth}
    \centering
    \includegraphics[width=\textwidth]{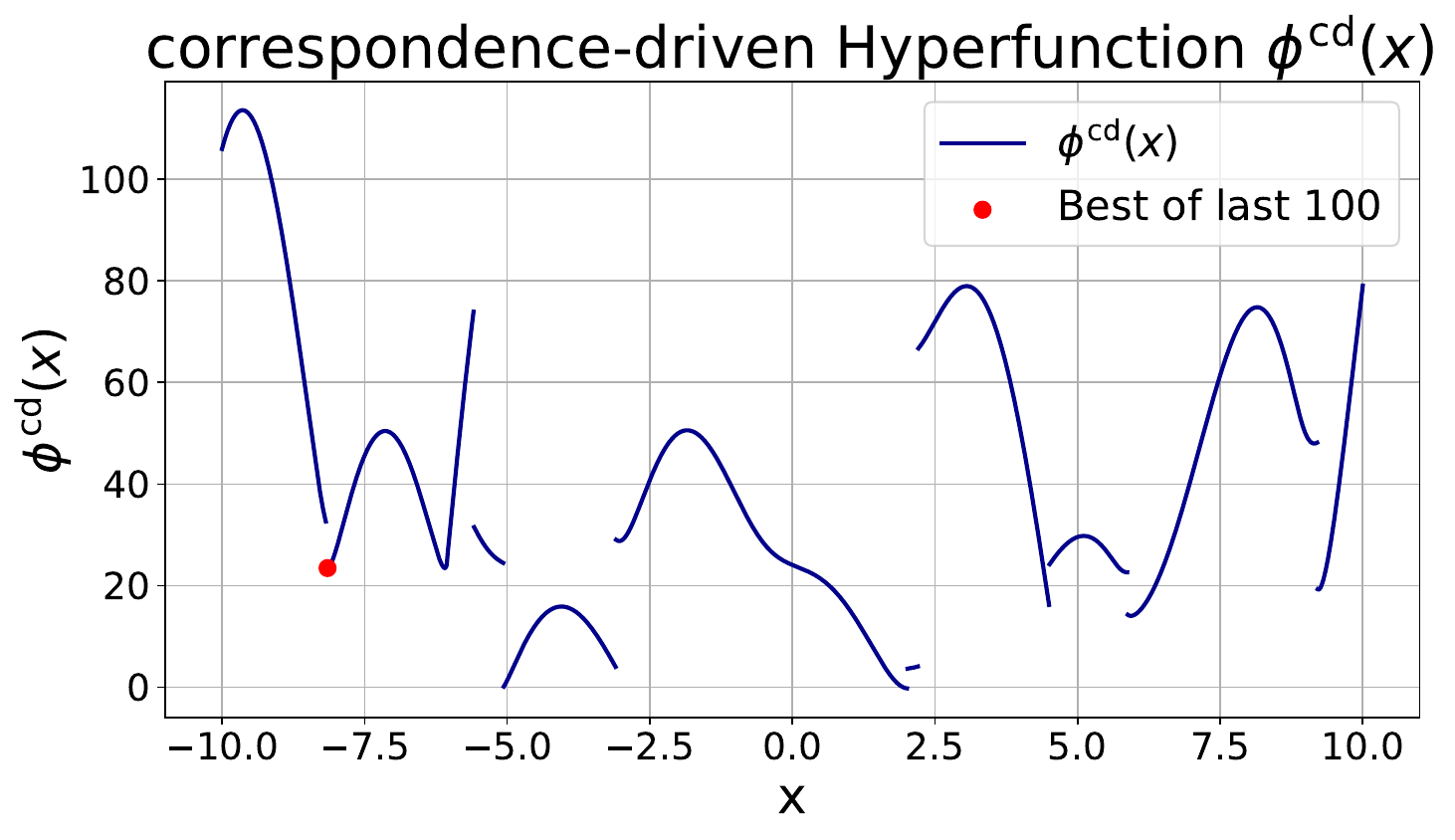}
    \par\small(g) Seed = 7
  \end{minipage}
  \hfill
  \begin{minipage}[b]{0.3\textwidth}
    \centering
    \includegraphics[width=\textwidth]{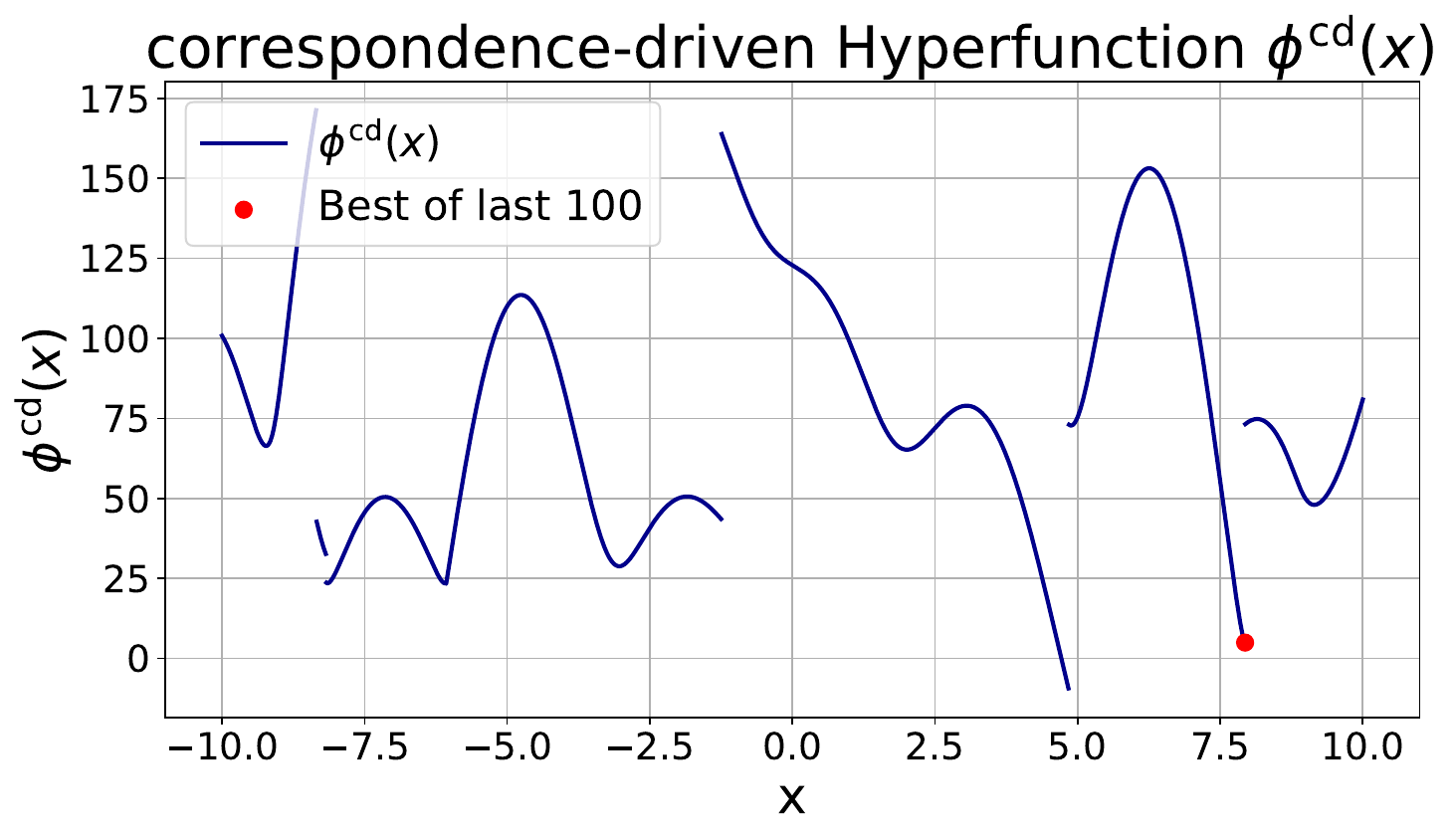}
    \par\small(h) Seed = 8
  \end{minipage}
  \hfill
  \begin{minipage}[b]{0.3\textwidth}
    \centering
    \includegraphics[width=\textwidth]{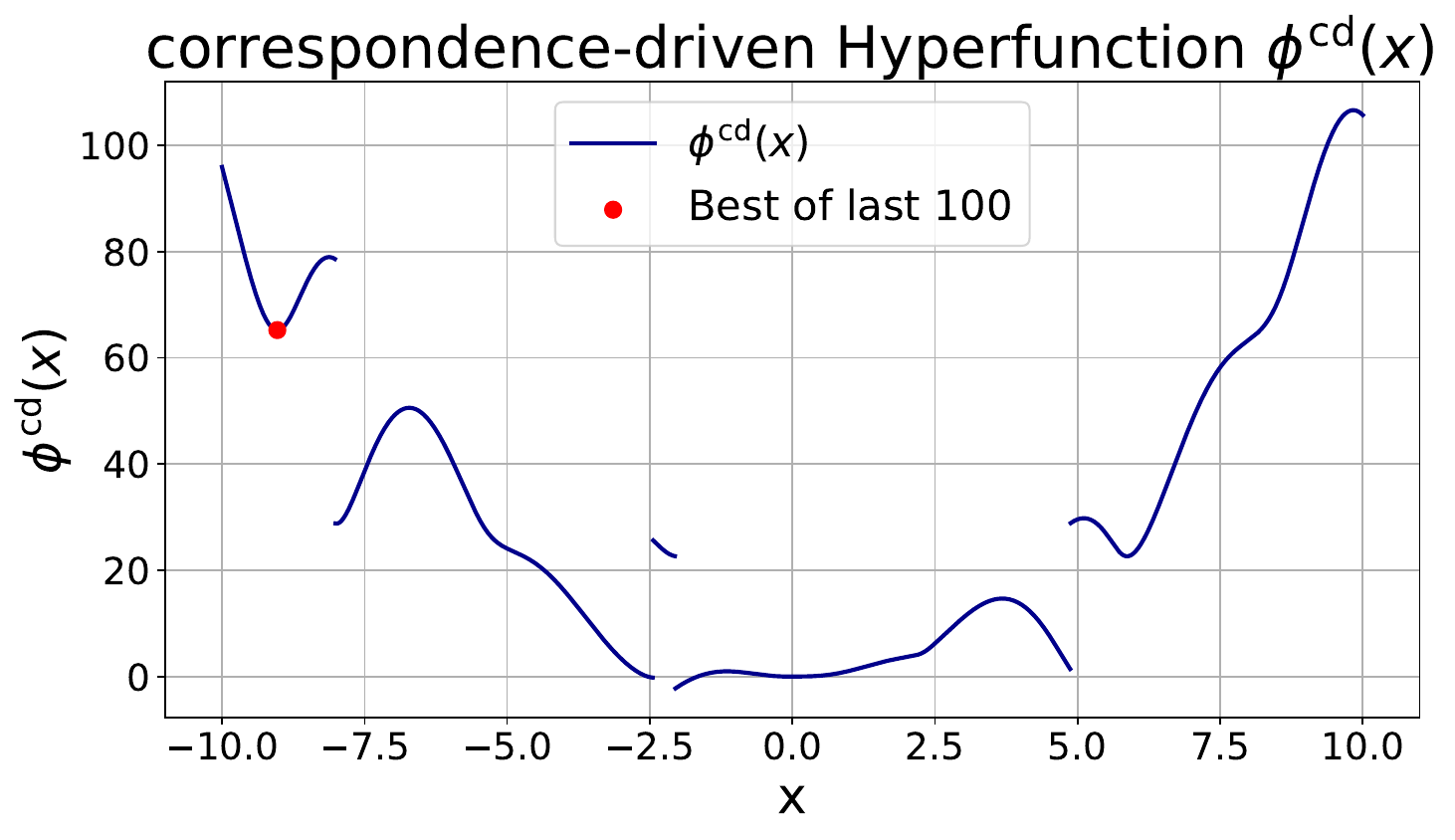}
    \par\small(i) Seed = 9
  \end{minipage}

  \vskip\baselineskip

  \begin{minipage}[b]{0.3\textwidth}
    \centering
    \includegraphics[width=\textwidth]{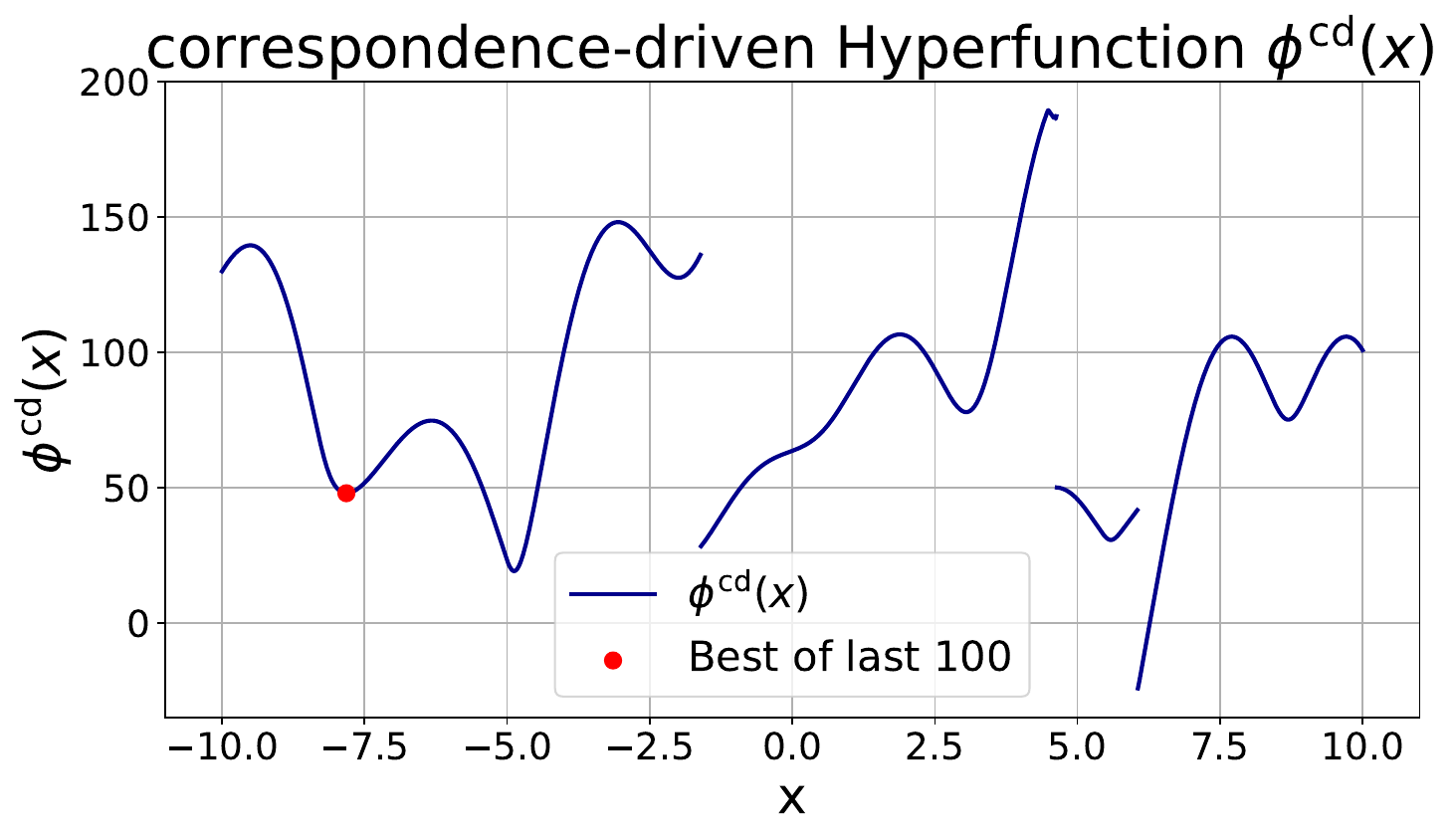}
    \par\small(j) Seed = 11
  \end{minipage}
  \hspace{0.1\textwidth}
  \begin{minipage}[b]{0.3\textwidth}
    \centering
    \includegraphics[width=\textwidth]{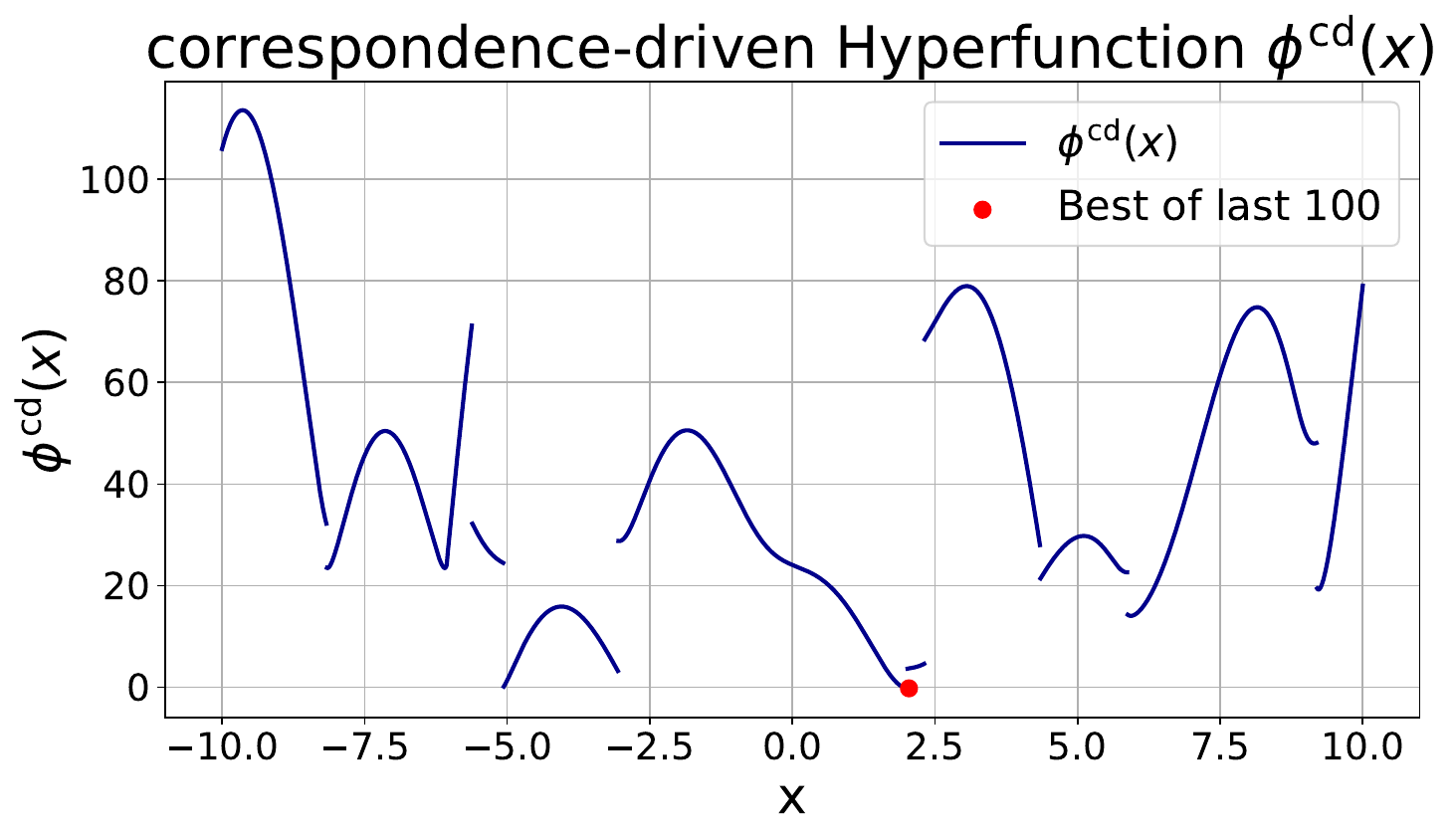}
    \par\small(k) Seed = 14
  \end{minipage}

  \caption{Plot of the \emph{correspondence-driven hyperfunction} $\phi^{\mathrm{cd}}(x)$ for different random seeds. In all cases, \ouralgo{} converges to either a local minimizer or the lower-value side of a discontinuity of $\phi^{\mathrm{cd}}(x)$.}
\end{figure*}

\section{Measure-Zero Discontinuities of Correspondence-Driven Solutions under Gradient Descent}\label{sec:MeasureZeroDiscontinuities}
\begin{theorem}\label{theorem:discontiunity_point}
    Suppose $g(x,y)$ is a $C^2$ function that is coercive in $y$ for any $x$. Let $\mathcal{M}$ be the gradient descent method with step size $0 < \eta < 1/\overline{L}_g$, where $\overline{L}_g$ is the Lipschitz smoothness coefficient of $g(x,\cdot)$. We uniformly choose $b$ from $[-\nu,\nu]^m$, where $\nu>0$ can be any constant, and perturb the initial point $y_0$ by setting it to $y_0+b$. Then, with probability one over the choice of $b$, the set of points in $\mathcal{X}\setminus\widetilde{\mathcal{X}}$ at which $y^{\text{cd}}(x)$ is discontinuous has Lebesgue measure zero, where $y^{\mathrm{cd}}(x)$ is the solution obtained by the algorithm $\mathcal{M}$ on the function $g(x,y)$ starting from the perturbed initialization $y_0+b$.
\end{theorem}

The core idea behind the proof of Theorem~\ref{theorem:discontiunity_point} is based on the stable manifold theorem~\citep[Page 122]{IntroductiontoDynamicalSystems} from discrete dynamical systems. Discontinuities in $y^{\text{cd}}(x)$ can only arise when the initial point $y_0$ lies on certain lower-dimensional stable manifolds associated with non-minimizing stationary points. However, such manifolds form a measure-zero subset of the domain. After perturbation, for almost every $x$, the initial point $y_0+b$ does not lie on any of them. As a result, the set of discontinuity points of $y^{\text{cd}}(x)$ has Lebesgue measure zero.

To proceed with the detailed proof, we introduce several relevant definitions and results. All concepts discussed here are within the framework of discrete dynamical systems.

\begin{definition}[Hyperbolic Set~(\cite{IntroductiontoDynamicalSystems}, Page 108)]
    \( M \) is a \( C^1 \) Riemannian manifold, \( U \subset M \) a non-empty open subset, and \( f : U \to f(U) \subset M \) a \( C^1 \) diffeomorphism. A compact, \( f \)-invariant subset \( \Lambda \subset U \) (i.e., $f(\Lambda)=\Lambda$) is called \emph{hyperbolic} if there exist \( \lambda \in (0,1) \), \( C > 0 \), and families of subspaces \( E^s(x) \subset T_x M \) and \( E^u(x) \subset T_x M \), \( x \in \Lambda \), such that for every \( x \in \Lambda \),
\begin{enumerate}
    \item \( T_x M = E^s(x) \oplus E^u(x) \),
    \item \( \| df_x^n v^s \| \leq C \lambda^n \| v^s \| \) for every \( v^s \in E^s(x) \) and \( n \geq 0 \),
    \item \( \| df_x^{-n} v^u \| \leq C \lambda^n \| v^u \| \) for every \( v^u \in E^u(x) \) and \( n \geq 0 \),
    \item \( df_x E^s(x) = E^s(f(x)) \) and \( df_x E^u(x) = E^u(f(x)) \).
\end{enumerate}
\end{definition}

Here \(E^s(x)\) and \(E^u(x)\) denote the stable and unstable subspaces at \(x\in\Lambda\). The space \(T_xM\) is the tangent space of the manifold \(M\) at the point \(x\), consisting of the velocities of smooth curves through \(x\). The map \(df_x: T_xM \to T_{f(x)}M\) is the derivative (differential, pushforward) of \(f\) at \(x\). For \(n\ge 1\), the notation
\[
df_x^{\,n} := d(f^{n})_x,
\qquad
df_x^{\,-n} := d(f^{-n})_x
\]
denotes the derivatives of the \(n\)-th forward iterates $f^n=f\circ f\circ\cdots f$ and backward iterates $f^{-n}=f^{-1}\circ f^{-1}\circ\cdots f^{-1}$. 

Informally, a hyperbolic set $\Lambda$ is one where, at every point, the tangent space splits into two complementary kinds of directions: along the “stable” directions $E^s$, repeated application of $f$ drags nearby points toward $\Lambda$; along the “unstable” directions $E^u$, the same iterations push points away. For instance, the map $f(x,y)=((1-2\alpha)x,(1+2\alpha)y))$ with a small constant $\alpha$ (which you can view as gradient step for the function $x^2-y^2$) maps original point to itself: points on the $x$-axis shrink under repeated application of the map and converge to $(0,0)$, whereas points on the $y$-axis expand and escape.

\begin{example}[Hyperbolic Fixed Point~(\cite{Hyperbolicity}, Page 11)]
    We say $x$ is a hyperbolic fixed point of the diffeomorphism $f:\R^n\to\R^n$ if $f(x)=x$ and all eigenvalues of $\nabla_{x} f(x)$ have norm different from one. A hyperbolic fixed point is a special case of a hyperbolic set.
\end{example}

\begin{definition}[Stable Manifold~(\cite{IntroductiontoDynamicalSystems}, Page 122)]
    \( M \) is a \( C^1 \) Riemannian manifold, \( U \subset M \) a non-empty open subset, and \( f : U \to f(U) \subset M \) a \( C^1 \) diffeomorphism. Let \( \Lambda \) be a hyperbolic set of \( f : U \to M \) and \( x \in \Lambda \). The (global) stable manifold of \( x \) is defined by
\[
W^s(x) = \{ y \in M : \mathrm{dist}(f^n(x), f^n(y)) \to 0 \text{ as } n \to \infty \}.
\]
\end{definition}

\begin{theorem}[Global Version of Stable Manifold Theorem~(\cite{IntroductiontoDynamicalSystems}, Corollary 5.6.6)]
    The global stable manifolds are embedded \(C^1\) submanifolds of \( M \); and it is homeomorphic to the open unit balls in corresponding dimensions.
\end{theorem}

If $\Lambda={x^*}$ is a hyperbolic fixed point of a diffeomorphism $f:\R^n\to\R^n$, then the dimension of the stable manifold at $x^*$ is equal to the number of eigenvalues of $\nabla_x f(x^*)$ with norm less than one.

\begin{theorem}[Fubini-Tonelli theorem]Let $ A $ and $ B $ be both $\sigma$-finite measure spaces, and $ A \times B $ be their product measurable space. If $ f: A \times B \to [0, \infty] $ is a non-negative measurable function, then$$\begin{aligned}
\int_A \left( \int_B f(x, y) \, dy \right) dx 
= \int_B \left( \int_A f(x, y) \, dx \right) dy 
= \int_{A \times B} f(x, y) \, d(x, y).
\end{aligned}$$\end{theorem}

\begin{lemma}
    Suppose $\mathcal{M}$ is gradient descent method, and $0<\eta<1/\overline{L}_g$. The following two hold:
    \begin{enumerate}
        \item Consider the map $F_x:y\mapsto y-\eta\nabla_y g(x,y)$, then $F_x(y)$ is a diffeomorphism on $\mathbb{R}^m$;
        \item If $g(x,y)$ is Morse in $y$ at $x$, and $\widetilde{y}$ is a stationary of $g(x,y)$, then $\widetilde{y}$ is a hyperbolic fixed point of $F$.
    \end{enumerate}
\end{lemma}

\begin{proof}
    We first prove that $F_x$ is injective. If there exists $y_1\ne y_2$ such that $F_x(y_1)=F_x(y_2)$, i.e., $y_1-\eta\nabla_y g(x,y_1)=y_2-\eta\nabla_y g(x,y_2)$, then $y_1-y_2=\eta(\nabla_y g(x,y_1)-\nabla_y g(x,y_2))$. However, we also have
    $$\eta\|\nabla_y g(x,y_1)-\nabla_y g(x,y_2)\|\leq\eta\overline{L}_g\|y_1-y_2\|<\|y_1-y_2\|.$$
    This leads to a contradiction, so $F$ is injective.

    Next, we prove that $F_x$ is surjective. For any $y_1$, consider the following problem
    $$\psi(y)=\frac{1}{2}\|y-y_1\|^2-\eta g(x,y).$$
    It is easy to see that $\psi(y)$ is coercive, so it has a stationary point, i.e., there exists $y$ such that 
    $$y-y_1-\eta\nabla_y g(x,y)=0.$$
    This means $y_1=y-\eta\nabla_y g(x,y)$. Thus $F$ is surjective.

    Finally, we prove that $F_x$ and $F_x^{-1}$ is differentiable. The differentiability of $F_x$ is clear. To prove the differentiability of $F_x^{-1}$, note that 
    $$\nabla_y F_x(y)=I-\eta \nabla^2_{yy} g(x,y)$$
    is non-degenerate since $\eta<1/\overline{L}_g$. Therefore, by implicit function theorem, $F^{-1}$ is differentiable at $F_x(y)$. We have proved that $F_x$ is surjective, so $F^{-1}$ is differentiable at any point.

    Since $g(x,y)$ is Morse in $y$, all eigenvalues of $\nabla^2_{yy} g(x,\widetilde{y})$ are nonzero. Note that the eigenvalues of $\nabla_y F$ are $1-\eta\lambda_i$, where $\lambda_i$ are the eigenvalues of $\nabla^2_{yy} g(x,\widetilde{y})$ that are nonzero. This combines the fact that $\eta<1/\overline{L}_g\leq1/\max_i\{|\lambda_i|\}$ ensures that the norm of $1-\eta\lambda_i$ is different from $1$ for any $i$. Therefore, $\widetilde{y}$ is a hyperbolic fixed point.
\end{proof}

\noindent\textbf{Detailed proof of Theorem~\ref{theorem:discontiunity_point}:} Fix an arbitrary $x \in \mathcal{X}\setminus\widetilde{\mathcal{X}}$ and consider the gradient descent sequence $\{y_k\}$ initiated at $y_0$ with a step size $0 < \eta < 1/\overline{L}_g$. We first prove that the sequence generated by the gradient descent method will converge to a stationary point. Since the gradient $\nabla_y g(x, \cdot)$ is $\overline{L}_g$-Lipschitz continuous, standard optimization theory provides the following descent inequality for consecutive iterations:$$g(x, y_{k+1}) \leq g(x, y_k) - \left( \frac{1}{\eta} - \frac{\overline{L}_g}{2} \right) \|y_{k+1} - y_k\|^2.$$Because $\eta < 1/\overline{L}_g$, the coefficient $C = \left(\frac{1}{\eta} - \frac{\overline{L}_g}{2}\right)$ is strictly positive. This strict descent property ensures that the entire sequence $\{y_k\}$ is strictly confined within the initial sublevel set $\mathcal{L}(x, y_0) := \{y \in \mathbb{R}^m \mid g(x, y) \leq g(x, y_0)\}$. By the coercivity assumption on $g(x, \cdot)$, this sublevel set is bounded, and being closed, it is compact.  Summing the descent inequality from $k=0$ to infinity yields a telescoping sum:$$C \sum_{k=0}^{\infty} \|y_{k+1} - y_k\|^2 \leq g(x, y_0) - \inf_y g(x, y).$$Since $g(x,\cdot)$ is coercive, it is bounded from below, making the right side of the inequality finite. Thus, the general term of the series must approach zero, strictly proving that $\|y_{k+1} - y_k\| \to 0$, and consequently $\|\nabla_y g(x, y_k)\| \to 0$ as $k \to \infty$. By assumption, $g(x,\cdot)$ is a Morse function, meaning all its stationary points are strictly isolated. Because $\mathcal{L}(x, y_0)$ is a compact set, the intersection of these isolated stationary points with $\mathcal{L}(x, y_0)$ must form a finite discrete set, which we denote by $\mathcal{P}(x, y_0)$. Let $\Omega$ denote the set of all accumulation points of $\{y_k\}$. Since $\{y_k\}$ stays within the compact set $\mathcal{L}(x, y_0)$, the Bolzano-Weierstrass theorem ensures $\Omega$ is non-empty. Furthermore, since the gradient vanishes along the sequence, any accumulation point must be a stationary point, meaning $\Omega \subseteq \mathcal{P}(x, y_0)$. Crucially, a standard result in analysis states that if a bounded sequence satisfies $\|y_{k+1} - y_k\| \to 0$, its set of accumulation points $\Omega$ must be a connected set. Since $\Omega$ is a non-empty connected subset of the finite discrete set $\mathcal{P}(x, y_0)$, it must be a singleton. Therefore, the entire sequence converges strictly to a unique stationary point.

Next, we characterize the discontinuity of $y^{\text{cd}}(x)$. The basin of attraction for each stationary point corresponds to its stable manifold. According to the Global Stable Manifold Theorem, the stable manifold of a local minimizer is an open $m$-dimensional submanifold, whereas the stable manifolds of non-minimizing stationary points (saddle points and local maximizers) have dimension strictly less than $m$. If the initialization $y_0$ lies in the interior of the basin of attraction of a local minimizer for parameter $x$, it is easy to see that for sufficiently small perturbations of $x$, the gradient descent sequence will still converge to the continuous continuation of that local minimizer. Thus, discontinuities in $y^{\text{cd}}(x)$ arise only if the initialization $y_0$ lies exactly on the boundary separating different basins of attraction, which corresponds precisely to the stable manifolds of non-minimizing stationary points. Let $\Phi_k(x, y)$ denote the $k$-th iterate of gradient descent initiated at $y$ under the parameter $x$. As established, for any $(x, y) \in (\mathcal{X}\setminus\widetilde{\mathcal{X}}) \times \mathbb{R}^m$, the sequence converges strictly to a stationary point. Thus, the pointwise limit function $\Phi_\infty(x, y) := \lim_{k \to \infty} \Phi_k(x, y)$ is well-defined. Since $\Phi_k(x, y)$ is continuous jointly in $(x, y)$ for each finite $k$, its pointwise limit $\Phi_\infty(x, y)$ is a Borel measurable function. We define the critical set of initializations in the product space, $\mathcal{G}$, as the set of all pairs $(x, y)$ such that the gradient descent sequence converges to a non-minimizing stationary point. A stationary point is a non-minimizer if and only if the minimal eigenvalue of the Hessian matrix of $g(x, \cdot)$ at that point is strictly negative. Therefore, $\mathcal{G}$ can be rigorously defined as:$$\mathcal{G} = \left\{ (x, y) \in (\mathcal{X}\setminus\widetilde{\mathcal{X}}) \times \mathbb{R}^m : \lambda_{\min}\Big(\nabla^2_{yy} g\big(x, \Phi_\infty(x,y)\big)\Big) < 0 \right\}.$$Since $g$ is a $C^2$ function, the minimal eigenvalue function $\lambda_{\min}(\nabla^2_{yy} g(x, y))$ is continuous. The composition of this continuous function with the Borel measurable function $\Phi_\infty(x, y)$ ensures that the map $(x, y) \mapsto \lambda_{\min}\big(\nabla^2_{yy} g(x, \Phi_\infty(x,y))\big)$ is Borel measurable. Consequently, $\mathcal{G}$ is the preimage of the open interval $(-\infty, 0)$ under a measurable function, making $\mathcal{G}$ a Borel set. For a fixed $x \in \mathcal{X}\setminus\widetilde{\mathcal{X}}$, let $\mathcal{A}(x) = \{ y \in \mathbb{R}^m : (x, y) \in \mathcal{G} \}$ be the $x$-section of $\mathcal{G}$. Since each non-minimizing stationary point has a stable manifold of dimension strictly less than $m$, $\mathcal{A}(x)$ is an at most countable union of lower-dimensional submanifolds. Because the countable union of measure-zero sets still has measure zero, $\mathcal{A}(x)$ has Lebesgue measure zero in $\mathbb{R}^m$. We introduce the indicator function for the perturbed initialization $y_0 + b$ falling into this critical set:$$I_{\mathcal{A}}(x, b) = \mathbb{I}\{ y_0 + b \in \mathcal{A}(x) \} = \mathbb{I}\{ (x, y_0 + b) \in \mathcal{G} \}.$$Because $\mathcal{G}$ is a Borel set, the function $(x, b) \mapsto I_{\mathcal{A}}(x, b)$ is jointly measurable. Notice that the condition $y_0 + b \in \mathcal{A}(x)$ is strictly equivalent to $b \in \mathcal{A}(x) - y_0$. Since the Lebesgue measure is translation-invariant, the shifted set $\mathcal{A}(x) - y_0$ also has Lebesgue measure zero in $\mathbb{R}^m$. Therefore, integrating the indicator function with respect to the perturbation $b$ over any domain, including $[-\nu, \nu]^m$, identically yields zero:$$\int_{[-\nu,\nu]^m} I_{\mathcal{A}}(x, b) \, db = 0, \quad \forall x \in \mathcal{X}\setminus\widetilde{\mathcal{X}}.$$By Fubini's theorem, since the indicator function is non-negative and measurable, we can swap the order of integration:$$\int_{[-\nu,\nu]^m} \left( \int_{\mathcal{X}\setminus\widetilde{\mathcal{X}}} I_{\mathcal{A}}(x, b) \, dx \right) db = \int_{\mathcal{X}\setminus\widetilde{\mathcal{X}}} \left( \int_{[-\nu,\nu]^m} I_{\mathcal{A}}(x, b) \, db \right) dx = 0.$$This implies that for almost every perturbation $b \in [-\nu, \nu]^m$, the integral $\int_{\mathcal{X}\setminus\widetilde{\mathcal{X}}} I_{\mathcal{A}}(x, b) \, dx$ is exactly zero. Since the integrand is non-negative, this strictly means the set of $x$ where $y^{\text{cd}}(x)$ is discontinuous has Lebesgue measure zero.

\section{Detailed proof}

\subsection{Proof of Theorem~\ref{theorem:nowhere_dense}}\label{proof:theorem:nowhere_dense}

Before presenting the detailed proof, we provide some definitions and theorems that will be used.

\begin{definition}[Borel Set]
    A Borel set is any subset of a topological space that can be formed from its open sets (or, equivalently, from closed sets) through the operations of countable union, countable intersection, and relative complement.
\end{definition}

\begin{definition}[Analytic Set]
    The set $A$ is called an analytic set if it is the continuous image of a Borel set in a separable completely metrizable topological space.
\end{definition}

According to \citet[Theorem 29.7]{ClassicalDescriptiveSetTheory}, any analytic set in Euclidean space is measurable.

\begin{theorem}[Sard's Theorem]
    Let $f : \mathbb{R}^n \to \mathbb{R}^m$ be $C^k$, where $k \geq \max\{n - m + 1, 1\}$. Let $X \subset \mathbb{R}^n$ denote the critical set of $f$, which is the set of points $x \in \mathbb{R}^n$ at which the Jacobian matrix of $f$ has rank strictly less than $m$. Then the critical value set, that is, the image $f(X)$, has Lebesgue measure $0$ in $\mathbb{R}^m$.
\end{theorem}

\noindent\textbf{Detailed proof of Theorem~\ref{theorem:nowhere_dense}:} Define the indicator function $I(a,x) = \mathbb{I}\{ \widetilde{g}_a(x,\cdot) \text{ is not Morse} \}$, where $\widetilde{g}_a(x,y)=g(x,y)+a^\top y$. First, we establish the measurability of $I(a,x)$. The set of points $(x,y,a)$ satisfying $\nabla_y \widetilde{g}_a = 0$ and $\det(\nabla_{yy}^2\widetilde{g}_a) = 0$ is the zero set of continuous functions, and thus a Borel set in $\mathbb{R}^{n+2m}$. Its projection onto the $(x,a)$-space is an analytic set, which is Lebesgue measurable. Consequently, $I(a,x)$ is a measurable function. Next, fix $x \in \mathcal{X}$. The function $\widetilde{g}_a(x,\cdot)$ fails to be Morse if and only if $-a$ is a critical value of the map $y \mapsto \nabla_y g(x,y)$. By Sard's theorem, the set of such critical values has Lebesgue measure zero. Therefore, for any $x$, we have$$\int_{[-\nu,\nu]^m} I(a,x) \, da = 0.$$ Integrating over $\mathcal{X}$, we obtain $\int_{\mathcal{X}} \int_{[-\nu,\nu]^m} I(a,x) \, da \, dx = 0$. By Fubini's theorem, we can exchange the order of integration:$$\int_{[-\nu,\nu]^m} \left( \int_{\mathcal{X}} I(a,x) \, dx \right) da = 0.$$ Since the inner integral $\int_{\mathcal{X}} I(a,x) \, dx$ is non-negative, it must be zero for almost every $a \in [-\nu,\nu]^m$. This proves that with probability one over the random draw of $a$, the set $\widetilde{\mathcal{X}}$ (where $\widetilde{g}_a$ is not Morse) has measure zero.

\subsection{Proof of Proposition~\ref{proposition:limit}}\label{proof:proposition:limit}
    \textbf{Smoothness for any $\xi>0$:} We first show that $\pcd_\xi\in C^\infty(\mathbb{R}^n)$. Note that we assume that the number of lower-level iterations satisfies $K\geq K_0$ in Assumption \ref{assumption:descent}. $y^{\text{cd}}(x)$ is always in the level set $\{y:g(x,y)\leq g(x,y_0)\}$. By Proposition~\ref{prop:general}(\ref{assumption:general1(8)}), we have
$$|\pcd_\xi(x)|\leq \overline{f}.$$
    Since the Gaussian kernel $h_\xi(x-z)$ is smooth in $x$, and
    $$|\partial^\alpha_x h_\xi(x-z)|\leq\frac{C_\alpha}{\xi^{|\alpha|}}h_{\xi/2}(x-z)$$
    for some constant $C_\alpha$, the product $\partial^\alpha_z h_\xi(z)\pcd(x-z)$ is dominated by an integrable function. Here \(\partial^\alpha_x\) denotes the partial derivative with respect to \(x\) of multi-index \(\alpha\)=$(\alpha_1,\cdots\alpha_n)$, i.e., \(\partial^\alpha_x = \partial^{\alpha_1+\cdots\alpha_n}/\partial x_1^{\alpha_1} \cdots \partial x_n^{\alpha_n}\). Thus by Leibniz rule for parameter integrals, we have
    $$\partial^\alpha_x\pcd_\xi(x)=\int_{\mathbb{R}^n}\partial^\alpha_x h_\xi(x-z)\pcd(z)dz$$
    is well-defined and smooth for any multi-index $\alpha$. Hence $\pcd_\xi\in C^\infty(\mathbb{R}^n)$.
    
    \noindent\textbf{Pointwise convergence:} Suppose $\pcd$ is continuous at $\bar{x}$. Then for any $\delta>0$, there exists $r>0$ such that $|\pcd(z)-\pcd(\bar{x})|<\delta$ whenever $\|z-\bar{x}\|<r$. Then
    $$\pcd_\xi(\bar{x})-\pcd(\bar{x})=\int_{\R^n} h_\xi(\bar{x}-z)[\pcd(z)-\pcd(\bar{x})]dz.$$
    Splitting the integral, we obtain
    \begin{align*}
        \pcd_\xi(\bar{x})-\pcd(\bar{x})&=\int_{\|z-\bar{x}\|<r} h_\xi(\bar{x}-z)[\pcd(z)-\pcd(\bar{x})]dz+\int_{\|z-\bar{x}\|\geq r} h_\xi(\bar{x}-z)[\pcd(z)-\pcd(\bar{x})]dz\\
        &=:I_1+I_2.
    \end{align*}
    For $\|z-\bar{x}\|<r$, we have $|\pcd(z)-\pcd(\bar{x})|<\delta$, so
    $$|I_1|\leq\delta\int_{\R^n}h_\xi(\bar{x}-z)dz=\delta.$$
    For $\|z-\bar{x}\|\geq r$, since $|\pcd(z)-\pcd(\bar{x})|\leq|\pcd(z)|+|\pcd(\bar{x})|\leq 2\overline{f}$ by \eqref{eq:boundoff}, we get
    $$|I_2|\leq 2\overline{f}\int_{\|z-\bar{x}\|\geq r}h_\xi(\bar{x}-z)dz=2\overline{f}\mathbb{P}_{Z\sim\mathcal{N}(0,I_n)}(\|Z\|\geq \frac{r}{\xi})\to 0\text{ as }\xi\to 0.$$
    Therefore, the following holds for any $\delta$
    $$\lim_{\xi\to 0}|\pcd_\xi(\bar{x})-\pcd(\bar{x})|\leq \lim_{\xi\to 0}|I_1|+|I_2|\leq\delta.$$
    Let $\delta$ approaches $0$, we conclude
    $$\lim_{\xi\to 0}\pcd_\xi(\bar{x})=\pcd(\bar{x}).$$

    \noindent\textbf{Gradient convergence:} Assume $\pcd$ is differentiable at $\bar{x}$. Then
    $$\nabla_x\pcd_\xi=\int_{\R^n}\pcd(z)\nabla_x h_\xi(\bar{x}-z)dz=-\int_{\R^n}\pcd(z)\frac{\bar{x}-z}{\xi^2}h_{\xi}(\bar{x}-z)dz.$$
    Note that
    $$\int_{\R^n}\pcd(\bar{x})\frac{\bar{x}-z}{\xi^2}h_{\xi}(\bar{x}-z)dz=\int_{\R^n}\pcd(\bar{x})\frac{z}{\xi^2}h_{\xi}(z)dz=0$$
    because $zh_{\xi}(z)$ is antisymmetric in $z$. Set $w=\bar{x}-z$, and $\pcd(z)-\pcd(\bar{x})=-\nabla_x\pcd(\bar{x})^\top w+r(z)$, where $r(z)=o(\|w\|)$, we have
    \begin{align*}
        \nabla_x\pcd_\xi(\bar{x})&=-\int_{\R^n}[\pcd(z)-\pcd(\bar{x})]\frac{\bar{x}-z}{\xi^2}h_\xi(\bar{x}-z)dz\\
        &=-\int_{\R^n}[-\nabla_x\pcd(\bar{x})^\top w]\frac{w}{\xi^2}h_\xi(w)dw-\int_{\R^n}r(z)\frac{w}{\xi^2}h_\xi(w)dw\\
        &=:A_\xi+B_\xi.
    \end{align*}
    By computing the component of $A_\xi$, we obtain
    \begin{align*}
        (A_\xi)_i=\sum_j\partial_{x_j}\pcd(\bar{x})\int_{\R^n}w_jw_i\frac{1}{\xi^2}h_\xi(w)dw=\sum_j\partial_{x_j}\pcd(\bar{x})\delta_{ij}=\partial_{x_i}\pcd(\bar{x}).
    \end{align*}
    Therefore, $A_\xi=\nabla_x\pcd(\bar{x})$. Here we used the result that $\int_{\R^n}w_jw_i\frac{1}{\xi^2}h_\xi(w)dw=\delta_{ij}$. This follows by defining the random vector $W\equiv(W_1,...,W_n)^\top\sim\mathcal{N}(0,\xi^2 I_n)$, whose density is precisely $h_\xi(w)$. Then
    $$\int_{\R^n}w_jw_i\frac{1}{\xi^2}h_\xi(w)dw=\frac{1}{\xi^2}\E[W_iW_j]=\frac{1}{\xi^2}(\E[W_iW_j]-\E[W_i]\E[W_j])=\frac{1}{\xi^2}\text{Cov}(W_i,W_j)=\delta_{ij}.$$
    
    The vanishing of the remainder term $B_\xi$ follows by exactly the same near-far splitting and Gaussian tail-bound argument used above in the pointwise convergence. We conclude
    $$\lim_{\xi\to 0}\nabla_x\pcd_\xi(\bar{x})=\nabla_x\pcd(\bar{x}).$$

\subsection{Proof of Proposition~\ref{proposition:boundedgradient}}\label{proof:proposition:boundedgradient}

For any vector $v$ with norm $1$, we have
\begin{align*}
    v^\top\nabla_x\pcd_\xi(x)&=v^\top\int_{\mathbb{R}^n}\nabla_xh_\xi(x-z)\phi^{\text{cd}}(z)dz\\
    &\overset{\text{(i)}}{=}v^\top\int_{\mathbb{R}^n}\nabla_u h_\xi(u)\phi^{\text{cd}}(x-u)du\\
    &=-\frac{1}{\xi^2}\int_{\R^n}v^\top u h_\xi(u)\pcd(x-u)du.
\end{align*}
In (i) we substitute $u=x-z$. Note that $|\pcd|\leq \overline{f}$, we have the bound
\begin{align*}
    |v^\top\nabla_x\pcd_\xi(x)|\leq\frac{\overline{f}}{\xi^2}\int_{\R^n}|v^\top u|h_\xi(u)du=\frac{\overline{f}}{\xi^2}\E[|v^\top Z|]=\frac{\overline{f}}{\xi^2}\xi\sqrt{\frac{2}{\pi}}=\sqrt{\frac{2}{\pi}}\frac{\overline{f}}{\xi}
\end{align*}
since for any unit vector $v$, the random variable $v^\top Z$ with $Z\sim\mathcal{N}(0,\xi^2I_n)$ is still a one-dimensional Gaussian, i.e., $v^\top Z\sim\mathcal{N}(0,\xi^2)$. 

Thus, we conclude
\begin{align*}
    \|\nabla_x\pcd_\xi(x)\|\leq\sqrt{\frac{2}{\pi}}\frac{\overline{f}}{\xi}.
\end{align*}
\subsection{Proof of Proposition~\ref{proposition:lipschitz}}\label{proof:proposition:lipschitz}
    Note that the Hessian of $\phi_\xi(x)$ can be computed as follows
    $$H(x)=\int_{\mathbb{R}^n}\nabla^2_{xx} h_\xi(x-z)\phi^{\text{cd}}(z)dz,$$
    we need to estimate the operator norm of $H(x)$. By definition
    $$\left\|H(x)\right\|_{\mathrm{op}}=\sup_{\|\eta\|=1}\left|\eta^\intercal H(x) \eta \right|.$$
    For any $\eta$, we have
    \begin{align*}
        \left|\eta^\intercal H(x)\eta\right|&=\left|\int_{\mathbb{R}^n}\eta^\intercal \nabla^2_{xx} h_\xi(x-z)\eta\phi^{\text{cd}}(z)dz\right|\overset{(i)}{\leq} \overline{f}\int_{\mathbb{R}^n}\left|\eta^\intercal \nabla^2_{xx} h_\xi(x-z)\eta\right| dz=\overline{f}\int_{\mathbb{R}^n}\left|\eta^\intercal \nabla^2_{zz} h_\xi(z)\eta\right|dz,
    \end{align*}
    where (i) is from \eqref{eq:boundoff}. By computation, we have
    \begin{align*}
        \nabla_{zz}^2 h_\xi(z)&=\nabla_z \left(-h_\xi(z)\frac{1}{\xi^2}h_\xi(z)z\right)=h_\xi(z)\left(\frac{zz^\top}{\xi^4}-\frac{I}{\xi^2}\right).
    \end{align*}
    Note that $h_\xi(z)$ depends only on $\|z\|$, it is invariant under every orthogonal transformation, i.e. $h_\xi(z)=h_\xi(Qz)$ for any orthogonal matrix $Q$. Suppose $Q$ satisfies $Qe_1=\eta$, where $e_1=(1,0,...,0)^\top$. We obtain that
    \begin{align*}
        \int_{\R^n}\left|\eta^\top\nabla^2_{zz}h_\xi(z)\eta \right|dz&=\int_{\R^n}\left|e_1^\top Q^\top\nabla_{zz}h_\xi(z)Q e_1 \right| dz\\
        &=\int_{\R^n}\left|e_1^\top Q^\top h_\xi(z)\left(\frac{zz^\top}{\xi^4}-\frac{I}{\xi^2}\right) Q e_1 \right|dz\\
        &=\int_{\R^n}\left|e_1^\top h_\xi(z)\left(\frac{(Q^\top z)(Q^\top z)^\top}{\xi^4}-\frac{I}{\xi^2}\right) e_1 \right|dz\\
        &=\int_{\R^n}\left|e_1^\top h_\xi(Qz)\left(\frac{zz^\top}{\xi^4}-\frac{I}{\xi^2}\right) e_1 \right|dz\\
        &=\int_{\R^n}\left|e_1^\top h_\xi(z)\left(\frac{zz^\top}{\xi^4}-\frac{I}{\xi^2}\right) e_1 \right|dz\\
        &=\int_{\R^n}\left|e_1^\top\nabla^2_{zz}h_\xi(z)e_1 \right|dz.
    \end{align*}

    Set $w=z/\xi$, then $z=\xi w$ and $dz=\xi^n dw$. Therefore
    \begin{align*}
        \int_{\mathbb{R}^n}\left|\eta^\intercal \nabla^2_{zz} h_\xi(z)\eta \right|dz&=\int_{\mathbb{R}^n}\left|e_1^\intercal \nabla^2_{zz} h_\xi(z)e_1 \right|dz&\\
        &=\frac{1}{\xi^2}\frac{1}{(2\pi\xi^2)^{n/2}}\int_{\mathbb{R}^n}\left|\frac{(\eta^\intercal(\xi w))^2}{\xi^2}-1\right|\exp\left(-\frac{\xi^2\|w\|^2}{2\xi^2}\right)\xi^n dw\\
        &=\frac{1}{\xi^2}\frac{1}{(2\pi\xi^2)^{n/2}}\int_{\mathbb{R}^n}\left|w_1^2-1\right|\exp\left(-\frac{\|w\|^2}{2}\right)\xi^n dw\\
        &=\frac{1}{\xi^2}\frac{1}{(2\pi)^{n/2}}\int_{\mathbb{R}}|w_1^2-1|\exp(-\frac{w_1^2}{2})dw_1\cdot\prod ^n_{j=2}\int_{\mathbb{R}}\exp(-\frac{w_j^2}{2})dw_j\\
        &\overset{\text{(i)}}{=}\frac{1}{\xi^2}\frac{1}{\sqrt{2\pi}}\int_{\mathbb{R}}|w_1^2-1|\exp(-\frac{w_1^2}{2})dw_1\\
        &\overset{\text{(ii)}}{\leq }\frac{1}{\xi^2}
    \end{align*}
where in $\text{(i)}$ we use 
$$\int_{\mathbb{R}}\exp(-\frac{w_j^2}{2})dw_j=\sqrt{2\pi},\quad\forall j$$
and in $\text{(ii)}$ we use the following numerical result
$$\frac{1}{\sqrt{2\pi}}\int_{\mathbb{R}}|w_1^2-1|\exp(-\frac{w_1^2}{2})dw_1\approx0.968<1.$$

\subsection{Proof of Theorem~\ref{theorem:poly_Minkowski}}\label{proof:theorem:poly_Minkowski}

We first present the definitions of semialgebraic sets and semialgebraic functions.

\begin{definition}[Semialgebraic Set]
    A subset $S$ of $\R^n$ is a semialgebraic set if it is a finite union of sets defined by polynomial equalities of the form $\{x\in\R^n:p(x)=0\}$ and of sets defined by polynomial inequalities of the form $\{x\in\R^n:p(x)<0\}$.
\end{definition}

\begin{definition}[Semialgebraic Function]
    A semialgebraic function is a function whose graph is a semialgebraic graph set.
\end{definition}

\begin{lemma}[Yomdin-Gromov Algebraic Lemma~\citep{yomdin1987volume,gromov1987entropy}]\label{lemma:algebraiclemma}
Let $X \subset [0,1]^n$ be a semialgebraic set of dimension $\mu$ defined by conditions
\[
p_j(x)=0 \quad \text{or} \quad p_j(x)<0,
\]
where $p_j$ are polynomials and $\sum \deg p_j=\beta$. Let $r\in\mathbb{N}$. 
There exists a constant $C=C(n,\mu,r,\beta)$ and semialgebraic functions
\[
\phi_1,\ldots,\phi_C : (0,1)^\mu \to X
\]
such that their images cover $X$ and $\|\phi_j\|_r \le 1$ for $j=1,\ldots,C$, where $\|\phi_j\|_r:=\max_{\alpha\leq r}\|D^\alpha\phi_j\|/\alpha!$.
\end{lemma}

\noindent\textbf{Detailed proof of Theorem~\ref{theorem:poly_Minkowski}:} Let $\mathcal{X}^*$ denote the set of $x \in \mathcal{D}$ such that the lower-level function $g(x,\cdot)$ is not Morse, where $\mathcal{D}$ is the projection of the natural domain of $g$ onto the $x$-coordinates. By Definition~\ref{def:bifurcationpointset}, $\widetilde{\mathcal{X}} = \mathcal{X}^* \cap \mathcal{X}$, where $\mathcal{X}$ is the compact constraint set of upper-level problem from Assumption~\ref{assumption:general1}(\ref{assumption:general1(2)}). It is not hard to see that $\mathcal{X}^*=\text{Proj}_{x}\{(x,y):\nabla_y g(x,y)=0,\ \det(\nabla^2_{yy} g(x,y))=0\}$, where $\text{Proj}_x:\mathbb{R}^{n+m}\to\mathbb{R}^n$ denotes the projection onto the $x$ coordinates. Note that the partial derivatives of a semi-algebraic function are also semi-algebraic functions~\citep[Exercise 2.10]{coste2000introduction}; the product of two semi-algebraic functions is again a semi-algebraic function; and the zero set of a semi-algebraic function is a semi-algebraic set. It follows that the set $\{(x,y):\nabla_y g(x,y)=0,\ \det(\nabla^2_{yy} g(x,y))=0\}$ is a semi-algebraic set in $\mathbb{R}^{n+m}$. By the Tarski-Seidenberg theorem~\citep[Theorem 2.2.1]{Bochnak1992RealAG}, the projection of a semi-algebraic set is semi-algebraic. Thus, the projection $X^*$ is a semi-algebraic set in $\mathbb{R}^n$.

To bound the Minkowski dimension of $\widetilde{\mathcal{X}}$, we construct a bounded semi-algebraic superset. Since the constraint set $\mathcal{X}$ is compact, there exists a sufficiently large radius $R > 0$ such that $\mathcal{X}$ is contained in the closed Euclidean ball $B_R = \{x \in \mathbb{R}^n : \|x\| \le R\}$. We define the $Z = \mathcal{X}^* \cap B_R$. Since $B_R$ is semi-algebraic and the intersection of semi-algebraic sets remains semi-algebraic, $Z$ is a bounded semi-algebraic set. Furthermore, by construction, we have the inclusion $\widetilde{\mathcal{X}} \subseteq Z$.

Let $\mu$ be the dimension of $Z$. By Assumption~\ref{assumption:nowhere_dense}, $\mathcal{X}^*$ has Lebesgue measure zero. Since $Z\subset \mathcal{X}^*$, it follows that the Lebesgue measure of $Z$ is also zero. For semi-algebraic sets, zero measure implies that the dimension is strictly less than the ambient dimension, i.e., $\mu \le n-1$.

We apply the Yomdin-Gromov Algebraic (Lemma~\ref{lemma:algebraiclemma}) to $Z$. This lemma guarantees the existence of finitely many semi-algebraic maps $\phi_j: [0,1]^\mu \to \mathbb{R}^n$ such that $Z \subseteq \bigcup_j \text{Im}(\phi_j)$ and the first-order derivatives of each $\phi_j$ are uniformly bounded. The boundedness of derivatives ensures that each $\phi_j$ is Lipschitz continuous. It is a standard property that Lipschitz mappings do not increase the upper Minkowski dimension. Since the upper Minkowski dimension of the unit hypercube $[0,1]^\mu$ is exactly $\mu$, we have $\overline{\dim}_{\mathrm{box}}(\text{Im}(\phi_j)) \le \overline{\dim}_{\mathrm{box}}([0,1]^\mu) = \mu$. 

By the finite stability of the upper Minkowski dimension (i.e., $\overline{\dim}_{\mathrm{box}}(\cup_j A_j) = \max_j \overline{\dim}_{\mathrm{box}}(A_j)$) and its monotonicity with respect to inclusion, we obtain:
$$
\overline{\dim}_{\mathrm{box}}(\widetilde{\mathcal{X}}) \le \overline{\dim}_{\mathrm{box}}(Z) \le \overline{\dim}_{\mathrm{box}}\left( \bigcup_{j} \text{Im}(\phi_j) \right) = \max_j \overline{\dim}_{\mathrm{box}}(\text{Im}(\phi_j)) \le \mu.
$$
Given that $\mu \le n-1$, we conclude that $\overline{\dim}_{\mathrm{box}}(\widetilde{\mathcal{X}}) \le n-1$.


    \subsection{Proof of Lemma~\ref{lem:measureofneighborhood}}\label{proof:lem:measureofneighborhood}
    By definition of upper Minkowski dimension, we have
    $$\limsup_{r\to 0}\frac{\log(N(\widetilde{\mathcal{X}},r))}{-\log r}=d.$$
    There exists $r_0$ such that
    $\frac{\log(N(\widetilde{\mathcal{X}},r))}{-\log r}\leq d+(n-d)/2=(d+n)/2$ holds for any $0<r<r_0$. This implies
    $N(\widetilde{\mathcal{X}},r)\leq M_1r^{-(d+n)/2}$ for some constant $M_1$, and any $0<r<r_0$. On the other hand, we have 
    $$N(\widetilde{\mathcal{X}},r)\leq N(\widetilde{\mathcal{X}},r_0)\quad\text{for all}\quad r\geq r_0.$$
    Combining these two, we obtain
    $$N(\widetilde{\mathcal{X}},r)\leq Mr^{-(d+n)/2},$$
    where $M=\max\{M_1,N(\widetilde{\mathcal{X}},r_0)r_0^{(d+n)/2}\}.$
    We select $N(\widetilde{\mathcal{X}},\delta)$ balls with radius $\delta$ to cover $\widetilde{\mathcal{X}}$, denoted as $\{B(x_i,\delta)\}_{i=1}^{N(\widetilde{\mathcal{X}},\delta)}$. Suppose $y\in\widetilde{\mathcal{X}}_{\delta}$, we can find $x\in\widetilde{\mathcal{X}}$ such that $d(x,y)\leq \delta$. Since $\{B(x_i,\delta)\}_{i=1}^{N(\widetilde{\mathcal{X}},\delta)}$ covers $\widetilde{\mathcal{X}}$, there exists $x_i$ such that $d(x,x_i)\leq \delta$. Thus we obtain that $d(y,x_i)\leq 2\delta$. This implies $\{B(x_i,2\delta)\}_{i=1}^{N(\widetilde{\mathcal{X}},\delta)}$ covers $\widetilde{\mathcal{X}}_{\delta}$. Note that the total volume of these $N(\widetilde{\mathcal{X}},\delta)$ balls with radius 
    $2\delta$ is $N(\widetilde{\mathcal{X}},\delta)(2\delta)^n\omega_n$, where $\omega_n=(\pi)^{n/2}/\Gamma(n/2+1)$, we obtain (\ref{equation:volume_of_balls}).

\subsection{Proof of Lemma~\ref{lemma:eigenvalue}}\label{proof:lemma:eigenvalue}
    We prove it by contradiction. Suppose there exists a sequence $\{(x_n,y_n)\}\subset\mathcal{X}\setminus(\widetilde{\mathcal{X}}_\delta)\times\mathcal{Y}'$ such that $\nabla_y g(x_n,y_n)=0$ for any $n$ and the smallest eigenvalue in norm of $\nabla^2_{yy}g(x_n,y_n)$ tends to zero as $n\to\infty$. By Assumption~\ref{assumption:general1}(\ref{assumption:general1(4)}), the norm of eigenvalues of $\nabla^2_{yy}g(x_n, y_n)$ is bounded above by $\overline{L}_g$. Since the smallest eigenvalue tends to zero, it follows that $\det(\nabla^2_{yy}g(x_n, y_n)) \to 0$. Note that $\mathcal{Y}'$ is compact, ${(x_n, y_n)}$ has a subsequence, which we denote again by ${(x_n, y_n)}$ for simplicity, that converges to a point $(\bar{x}, \bar{y})$ in $\overline{\mathcal{X} \setminus \widetilde{\mathcal{X}}_\delta} \times \mathcal{Y}'$, where the overline denotes the topological closure. By the continuity of the determinant function, we have $\nabla_y g(\bar{x},\bar{y})=0$ and $\det(\nabla^2_{yy} g(\bar{x},\bar{y}))=0$, which implies $\bar{x}\in\widetilde{\mathcal{X}}$. This contradicts that $\bar{x}\in\overline{\mathcal{X} \setminus \widetilde{\mathcal{X}}_\delta}$. Hence, we reach a contradiction, which completes the proof, i.e., over the set $\mathcal{X} \setminus \widetilde{\mathcal{X}}_\delta \times \mathcal{Y}'$, the norm of the smallest eigenvalue of $\nabla^2_{yy} g(x,y)$ at all points satisfying $\nabla_y g(x,y)=0$ admits a positive lower bound, which clearly depends only on $\delta$.

\subsection{Proof of Lemma~\ref{lemma:gradient_lower_bound}}\label{proof:lemma:gradient_lower_bound}
We first prove by contradiction that for any fixed $r$, the infimum of $\|\nabla_y g(x,y)\|$ over 
$$K:=\{(x,y)\in\overline{\mathcal{X}\setminus\widetilde{\mathcal{X}}_\delta}\times\mathcal{Y}:y\in\overline{\mathcal{Y}\setminus \mathrm{Crit}_r(x)}\}$$
is positive, where the overline denotes the topological closure. If this doesn't hold, there exists a sequence $\{(x_n,y_n)\}\subset K$ such that $\|\nabla_y g(x_n,y_n)\|$ goes to $0$ as $n$ tends to infinity. Note that $\overline{\mathcal{X}\setminus\widetilde{\mathcal{X}}_\delta}$ and $\mathcal{Y}$ are compact, we can suppose that $\{(x_n,y_n)\}$ converges to a point $(\bar{x},\bar{y})\in(\overline{\mathcal{X}\setminus\widetilde{\mathcal{X}}_\delta})\times \mathcal{Y}$, and $\nabla_y g(\bar{x},\bar{y})=0$ (see Figure~\ref{fig:lemma4.2p2}, $(\bar{x},\bar{y})$ is on the blue dashed curve inside the black rectangle). 
Note that the function $g(\bar{x},\cdot)$ is Morse. As $x$ varies near $\bar{x}$, each stationary point moves smoothly (by the implicit-function theorem). Therefore, when $x$ varies slightly around $\bar{x}$ such as $\|x-\bar{x}\|\leq\Delta_x$ with sufficient small constant $\Delta_x$, the Cartesian product $B_{\bar{x}}(\Delta_x)\times B_{\bar{y}}(r/2)$ is entirely outside $K$. That is, the point $(\bar{x},\bar{y})$ is contained in an open neighborhood that lies entirely outside $K$. This contradicts the fact that a sequence $\{(x_n,y_n)\}\subset K$ converging to $(\bar{x},\bar{y})$. As illustrated in Figure~\ref{fig:lemma4.2p2}, it is also intuitively clear that a sequence lying inside the black solid box but outside the blue tubular region cannot converge to a point on the blue dashed curve.

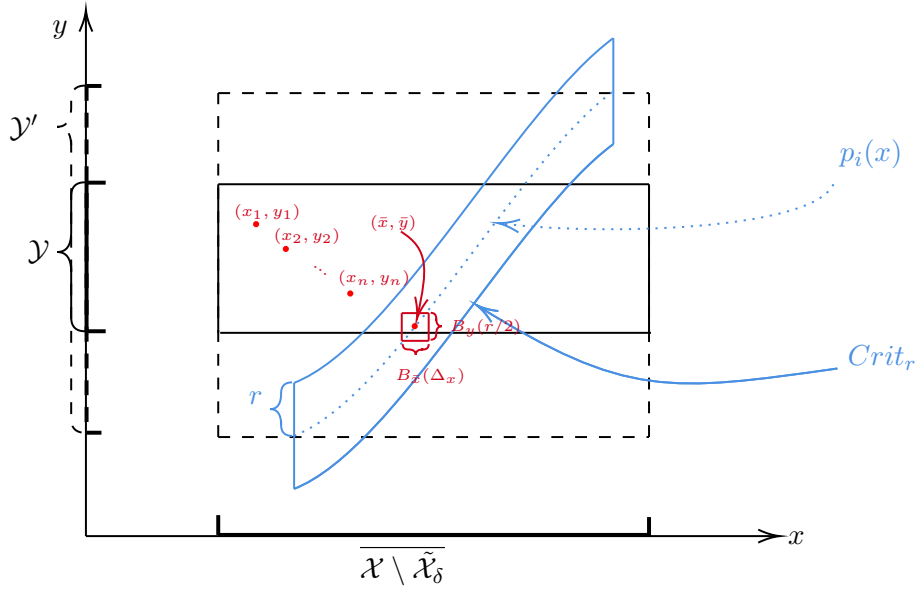
\begin{figure}
    \centering

\tikzset{every picture/.style={line width=0.75pt}} 

\begin{tikzpicture}[x=0.75pt,y=0.75pt,yscale=-1,xscale=1]

\draw    (144.27,274) -- (492.27,274) ;
\draw [shift={(494.27,274)}, rotate = 180] [color={rgb, 255:red, 0; green, 0; blue, 0 }  ][line width=0.75]    (10.93,-3.29) .. controls (6.95,-1.4) and (3.31,-0.3) .. (0,0) .. controls (3.31,0.3) and (6.95,1.4) .. (10.93,3.29)   ;
\draw    (144.27,274) -- (144.27,10.4) ;
\draw [shift={(144.27,8.4)}, rotate = 90] [color={rgb, 255:red, 0; green, 0; blue, 0 }  ][line width=0.75]    (10.93,-3.29) .. controls (6.95,-1.4) and (3.31,-0.3) .. (0,0) .. controls (3.31,0.3) and (6.95,1.4) .. (10.93,3.29)   ;
\draw  [dash pattern={on 4.5pt off 4.5pt}]  (211,224) -- (428.27,224) ;
\draw  [dash pattern={on 4.5pt off 4.5pt}]  (211,224) -- (211,50.4) ;
\draw  [dash pattern={on 4.5pt off 4.5pt}]  (428.27,224) -- (428.27,50.4) ;
\draw  [dash pattern={on 4.5pt off 4.5pt}]  (211,50.4) -- (428.27,50.4) ;
\draw    (211,171.4) -- (211,96.4) ;
\draw    (211,96.4) -- (428.27,96.4) ;
\draw    (212,171.4) -- (429.27,171.4) ;
\draw    (428.27,171.4) -- (428.27,96.4) ;
\draw [color={rgb, 255:red, 74; green, 144; blue, 226 }  ,draw opacity=1 ] [dash pattern={on 0.84pt off 2.51pt}]  (249.27,223.4) .. controls (300.27,203.4) and (361.27,85.4) .. (410.27,49.4) ;
\draw [color={rgb, 255:red, 74; green, 144; blue, 226 }  ,draw opacity=1 ]   (249.27,196.65) -- (249.27,250.15) ;
\draw [color={rgb, 255:red, 74; green, 144; blue, 226 }  ,draw opacity=1 ]   (410.27,22.65) -- (410.27,76.15) ;
\draw [color={rgb, 255:red, 74; green, 144; blue, 226 }  ,draw opacity=1 ]   (249.27,250.15) .. controls (300.27,230.15) and (361.27,112.15) .. (410.27,76.15) ;
\draw [color={rgb, 255:red, 74; green, 144; blue, 226 }  ,draw opacity=1 ] [dash pattern={on 0.84pt off 2.51pt}]  (521.27,96.4) .. controls (510.38,120.16) and (401.48,122.36) .. (353.44,116.19) ;
\draw [shift={(352,116)}, rotate = 7.71] [color={rgb, 255:red, 74; green, 144; blue, 226 }  ,draw opacity=1 ][line width=0.75]    (10.93,-3.29) .. controls (6.95,-1.4) and (3.31,-0.3) .. (0,0) .. controls (3.31,0.3) and (6.95,1.4) .. (10.93,3.29)   ;
\draw [line width=1.5]    (211,263.4) -- (211,273.4) ;
\draw [line width=1.5]    (428.27,264.4) -- (428.27,274) ;
\draw [line width=1.5]    (211,273.4) -- (428.27,273.4) ;
\draw  [color={rgb, 255:red, 74; green, 144; blue, 226 }  ,draw opacity=1 ] (248.27,196.4) .. controls (244.56,196.4) and (242.71,198.25) .. (242.71,201.96) -- (242.71,201.96) .. controls (242.71,207.25) and (240.86,209.9) .. (237.15,209.9) .. controls (240.86,209.9) and (242.71,212.55) .. (242.71,217.84)(242.71,215.46) -- (242.71,217.84) .. controls (242.71,221.55) and (244.56,223.4) .. (248.27,223.4) ;
\draw [color={rgb, 255:red, 74; green, 144; blue, 226 }  ,draw opacity=1 ]   (523.27,189.4) .. controls (438.69,201.34) and (416.49,204.37) .. (342.39,158.1) ;
\draw [shift={(341.27,157.4)}, rotate = 32.07] [color={rgb, 255:red, 74; green, 144; blue, 226 }  ,draw opacity=1 ][line width=0.75]    (10.93,-3.29) .. controls (6.95,-1.4) and (3.31,-0.3) .. (0,0) .. controls (3.31,0.3) and (6.95,1.4) .. (10.93,3.29)   ;
\draw [color={rgb, 255:red, 74; green, 144; blue, 226 }  ,draw opacity=1 ]   (523.27,189.4) .. controls (438.69,201.34) and (416.49,204.37) .. (342.39,158.1) ;
\draw [shift={(341.27,157.4)}, rotate = 32.07] [color={rgb, 255:red, 74; green, 144; blue, 226 }  ,draw opacity=1 ][line width=0.75]    (10.93,-3.29) .. controls (6.95,-1.4) and (3.31,-0.3) .. (0,0) .. controls (3.31,0.3) and (6.95,1.4) .. (10.93,3.29)   ;
\draw [color={rgb, 255:red, 74; green, 144; blue, 226 }  ,draw opacity=1 ]   (249.27,250.15) .. controls (300.27,230.15) and (361.27,112.15) .. (410.27,76.15) ;
\draw [color={rgb, 255:red, 74; green, 144; blue, 226 }  ,draw opacity=1 ]   (249.27,196.65) .. controls (300.27,176.65) and (361.27,58.65) .. (410.27,22.65) ;
\draw [line width=1.5]    (144.6,95.6) -- (144.6,170.6) ;
\draw [line width=1.5]    (144.6,170.6) -- (154.24,170.6) ;
\draw [line width=1.5]    (144.6,95.6) -- (154.24,95.6) ;
\draw [line width=1.5]  [dash pattern={on 5.63pt off 4.5pt}]  (144.6,46.79) -- (144.6,129.01) -- (144.6,221.81) ;
\draw [line width=1.5]  [dash pattern={on 5.63pt off 4.5pt}]  (144.6,221.81) -- (154.24,221.81) ;
\draw [line width=1.5]  [dash pattern={on 5.63pt off 4.5pt}]  (144.6,46.79) -- (154.24,46.79) ;
\draw   (143.84,95.41) .. controls (139.17,95.41) and (136.84,97.74) .. (136.84,102.41) -- (136.84,120.61) .. controls (136.84,127.28) and (134.51,130.61) .. (129.84,130.61) .. controls (134.51,130.61) and (136.84,133.94) .. (136.84,140.61)(136.84,137.61) -- (136.84,163.61) .. controls (136.84,168.28) and (139.17,170.61) .. (143.84,170.61) ;
\draw  [dash pattern={on 4.5pt off 4.5pt}][line width=0.75]  (143.84,46.61) .. controls (139.17,46.61) and (136.84,48.94) .. (136.84,53.61) -- (136.84,57.41) .. controls (136.84,64.08) and (134.51,67.41) .. (129.84,67.41) .. controls (134.51,67.41) and (136.84,70.74) .. (136.84,77.41)(136.84,74.41) -- (136.84,214.81) .. controls (136.84,219.48) and (139.17,221.81) .. (143.84,221.81) ;
\draw  [color={rgb, 255:red, 255; green, 0; blue, 0 }  ,draw opacity=1 ][fill={rgb, 255:red, 255; green, 0; blue, 0 }  ,fill opacity=1 ] (309.06,168.04) .. controls (309.06,167.47) and (309.52,167) .. (310.1,167) .. controls (310.68,167) and (311.14,167.47) .. (311.14,168.04) .. controls (311.14,168.62) and (310.68,169.09) .. (310.1,169.09) .. controls (309.52,169.09) and (309.06,168.62) .. (309.06,168.04) -- cycle ;
\draw [color={rgb, 255:red, 208; green, 2; blue, 27 }  ,draw opacity=1 ]   (303.56,161.44) -- (303.56,175.44) ;
\draw [color={rgb, 255:red, 208; green, 2; blue, 27 }  ,draw opacity=1 ]   (303.56,161.44) -- (317.14,161.44) ;
\draw [color={rgb, 255:red, 208; green, 2; blue, 27 }  ,draw opacity=1 ]   (317.14,161.44) -- (317.14,175.44) ;
\draw [color={rgb, 255:red, 208; green, 2; blue, 27 }  ,draw opacity=1 ]   (303.56,175.44) -- (317.14,175.44) ;
\draw  [color={rgb, 255:red, 208; green, 2; blue, 27 }  ,draw opacity=1 ] (304.06,177.09) .. controls (304.06,178.88) and (304.96,179.78) .. (306.75,179.78) -- (306.75,179.78) .. controls (309.32,179.78) and (310.6,180.68) .. (310.6,182.47) .. controls (310.6,180.68) and (311.88,179.78) .. (314.45,179.78)(313.29,179.78) -- (314.45,179.78) .. controls (316.24,179.78) and (317.14,178.88) .. (317.14,177.09) ;
\draw  [color={rgb, 255:red, 208; green, 2; blue, 27 }  ,draw opacity=1 ] (318.14,174.86) .. controls (320,174.86) and (320.93,173.93) .. (320.93,172.07) -- (320.93,172.07) .. controls (320.93,169.41) and (321.86,168.08) .. (323.73,168.08) .. controls (321.86,168.08) and (320.93,166.75) .. (320.93,164.09)(320.93,165.29) -- (320.93,164.09) .. controls (320.93,162.23) and (320,161.3) .. (318.14,161.3) ;
\draw [color={rgb, 255:red, 208; green, 2; blue, 27 }  ,draw opacity=1 ]   (303.14,120.3) .. controls (317.28,129.07) and (316.68,139.28) .. (311.55,162.71) ;
\draw [shift={(311.14,164.54)}, rotate = 282.53] [color={rgb, 255:red, 208; green, 2; blue, 27 }  ,draw opacity=1 ][line width=0.75]    (10.93,-3.29) .. controls (6.95,-1.4) and (3.31,-0.3) .. (0,0) .. controls (3.31,0.3) and (6.95,1.4) .. (10.93,3.29)   ;
\draw  [color={rgb, 255:red, 255; green, 0; blue, 0 }  ,draw opacity=1 ][fill={rgb, 255:red, 255; green, 0; blue, 0 }  ,fill opacity=1 ] (229.06,116.54) .. controls (229.06,115.97) and (229.52,115.5) .. (230.1,115.5) .. controls (230.68,115.5) and (231.14,115.97) .. (231.14,116.54) .. controls (231.14,117.12) and (230.68,117.59) .. (230.1,117.59) .. controls (229.52,117.59) and (229.06,117.12) .. (229.06,116.54) -- cycle ;
\draw  [color={rgb, 255:red, 255; green, 0; blue, 0 }  ,draw opacity=1 ][fill={rgb, 255:red, 255; green, 0; blue, 0 }  ,fill opacity=1 ] (244.06,129.04) .. controls (244.06,128.47) and (244.52,128) .. (245.1,128) .. controls (245.68,128) and (246.14,128.47) .. (246.14,129.04) .. controls (246.14,129.62) and (245.68,130.09) .. (245.1,130.09) .. controls (244.52,130.09) and (244.06,129.62) .. (244.06,129.04) -- cycle ;
\draw  [color={rgb, 255:red, 255; green, 0; blue, 0 }  ,draw opacity=1 ][fill={rgb, 255:red, 255; green, 0; blue, 0 }  ,fill opacity=1 ] (276.56,151.54) .. controls (276.56,150.97) and (277.02,150.5) .. (277.6,150.5) .. controls (278.18,150.5) and (278.64,150.97) .. (278.64,151.54) .. controls (278.64,152.12) and (278.18,152.59) .. (277.6,152.59) .. controls (277.02,152.59) and (276.56,152.12) .. (276.56,151.54) -- cycle ;

\draw (497,270.4) node [anchor=north west][inner sep=0.75pt]    {$x$};
\draw (126,12.4) node [anchor=north west][inner sep=0.75pt]    {$y$};
\draw (523,72.4) node [anchor=north west][inner sep=0.75pt]  [color={rgb, 255:red, 74; green, 144; blue, 226 }  ,opacity=1 ]  {$p_{i}( x)$};
\draw (113.4,123.4) node [anchor=north west][inner sep=0.75pt]    {$\mathcal{Y}$};
\draw (104.2,59.8) node [anchor=north west][inner sep=0.75pt]    {$\mathcal{Y} '$};
\draw (281,277.4) node [anchor=north west][inner sep=0.75pt]    {$\overline{\mathcal{X} \setminus \widetilde{\mathcal{X}}_{\delta }}$};
\draw (225,199.4) node [anchor=north west][inner sep=0.75pt]  [color={rgb, 255:red, 74; green, 144; blue, 226 }  ,opacity=1 ]  {$r$};
\draw (527,176.4) node [anchor=north west][inner sep=0.75pt]  [color={rgb, 255:red, 74; green, 144; blue, 226 }  ,opacity=1 ]  {$Crit_{r}$};
\draw (299.06,187.67) node [anchor=north west][inner sep=0.75pt]  [font=\tiny,color={rgb, 255:red, 208; green, 2; blue, 27 }  ,opacity=1 ]  {$B_{\bar{x}}( \Delta _{x})$};
\draw (326.56,163.17) node [anchor=north west][inner sep=0.75pt]  [font=\tiny,color={rgb, 255:red, 208; green, 2; blue, 27 }  ,opacity=1 ]  {$B_{y}( r/2)$};
\draw (286.06,110.4) node [anchor=north west][inner sep=0.75pt]  [font=\tiny,color={rgb, 255:red, 208; green, 2; blue, 27 }  ,opacity=1 ]  {$( \bar{x} ,\bar{y})$};
\draw (217.56,104.9) node [anchor=north west][inner sep=0.75pt]  [font=\tiny,color={rgb, 255:red, 208; green, 2; blue, 27 }  ,opacity=1 ]  {$( x_{1} ,y_{1})$};
\draw (237.56,116.9) node [anchor=north west][inner sep=0.75pt]  [font=\tiny,color={rgb, 255:red, 208; green, 2; blue, 27 }  ,opacity=1 ]  {$( x_{2} ,y_{2})$};
\draw (258.77,134.78) node [anchor=north west][inner sep=0.75pt]  [font=\tiny,color={rgb, 255:red, 208; green, 2; blue, 27 }  ,opacity=1 ,rotate=-35.56]  {$...$};
\draw (270.06,138.4) node [anchor=north west][inner sep=0.75pt]  [font=\tiny,color={rgb, 255:red, 208; green, 2; blue, 27 }  ,opacity=1 ]  {$( x_{n} ,y_{n})$};

\end{tikzpicture}
    \caption{The figure illustrates the case where the stationary point curve is a single curve $p_i(x)$. The solid-line rectangle represents the set $(\overline{\mathcal{X}\setminus\widetilde{\mathcal{X}}_\delta})\times \mathcal{Y}$, while the dashed-line rectangle represents $(\overline{\mathcal{X}\setminus\widetilde{\mathcal{X}}_\delta})\times \mathcal{Y}'$. The blue dashed curve denotes the restriction of the stationary point curve $p_i(x)$ to the domain $(\overline{\mathcal{X}\setminus\widetilde{\mathcal{X}}_\delta})\times \mathcal{Y}'$. The red tubular region consists of the slices $\mathrm{Crit}_r(x)$ for each $x$, while the blue tubular region consists of the slices $\mathrm{Crit}_{r_0}(x)$. Removing the blue tubular region from the solid rectangle yields $K$.}
    \label{fig:lemma4.2p2}
\end{figure}

\subsection{Proof of Lemma~\ref{lemma:approximationerror}}\label{proof:lemma:approximationerror}
    From \citet[Theorem 1]{nesterov2006cubic}, for the first $K_1$ iterations, the cubic-regularized Newton method guarantees:
    \begin{align}\label{eq:nesterov_bound}
    \min_{k=1,\dots,K_1} \nu_{\overline{\overline{L}}g}(y_k) \leq \frac{8}{3}\left(\frac{3(g(x,y_0)-\underline{g})}{2K_1\overline{\overline{L}}g}\right)^{1/3},
    \end{align}
    where the stationarity measure is defined as 
    $$\nu_{M}(y) := \max\left\{\sqrt{\frac{1}{M}|\nabla_y g(x,y)|}, -\frac{2}{3M}\lambda{\min}(\nabla^2_{yy} g(x,y))\right\}.$$
    Let $k^*$ be the index minimizing this measure within the first $K_1$ steps. By the specific choice of $K_1$ in \eqref{equation:requirement_for_K}, substituting $K_1$ into the RHS of \eqref{eq:nesterov_bound} ensures that $y_{k^*}$ satisfies two key conditions:\\
    
    \noindent\textbf{1. Gradient Condition:} The gradient norm satisfies:
    \begin{align}\label{eq:grad_condition}
    |\nabla_y g(x,y_{k^*})| < \min\left\{\alpha\left(\delta, \frac{\mu(\delta)}{2\overline{\overline{L}}_g}\right), \frac{(\mu(\delta))^2}{4\overline{\overline{L}}_g}\right\}.
    \end{align}
    
    \noindent\textbf{2. Curvature Condition:} The smallest eigenvalue of the Hessian satisfies:
    \begin{align}\label{eq:curv_condition}
    \lambda_{\min}(\nabla^2_{yy} g(x,y_{k^*})) > -\frac{\mu(\delta)}{2}.
    \end{align}
    We now analyze the geometric implications of these bounds.\\
    
    \noindent\textbf{Phase I: Localization to a Strongly Convex Neighborhood.} Since $x \in \mathcal{X} \setminus \widetilde{\mathcal{X}}_\delta$, $g(x, \cdot)$ is Morse. Let $\{p_j\}$ denote its discrete stationary points in $\mathcal{Y}'$. From Lemma \ref{lemma:gradient_lower_bound}, the first term in the minimum of \eqref{eq:grad_condition} implies that $y_{k^*}$ must lie within the $r$-neighborhood of some stationary point $p_j$, where $r := \mu(\delta)/(2\overline{\overline{L}}_g)$. Specifically, $y_{k^*} \in \bigcup_j B(p_j, r)$. Within any such ball $B(p_j, r)$, the Lipschitz continuity of the Hessian (Assumption~\ref{assumption:general1}(\ref{assumption:general1(11)})) and Lemma \ref{lemma:eigenvalue} imply that all eigenvalues of $\nabla^2_{yy} g(x,y)$ are lower bounded by $\mu(\delta) - \overline{\overline{L}}_g \cdot r = \mu(\delta)/2$ in absolute value. However, for saddle points or local maximizers, the smallest eigenvalue would be less then or equal to $-\mu(\delta)/2$. Condition \eqref{eq:curv_condition} ($\lambda_{\min} > -\mu(\delta)/2$) contradicts this, thereby excluding balls centered at non-minimizers. Furthermore, the monotonicity of the algorithm (Assumption \ref{assumption:descent}) excludes any local minimizers outside $\mathcal{Y}$ that have function values higher than $g(x,y_0)$. Thus, $y_{k^*}$ must reside in $B(p_{j^*}, r)$ for some local minimizer $p_{j^*} \in \mathcal{Y}$, denoted hereafter as $y^{\text{cd}}(x)$. In this neighborhood $B(y^{\text{cd}}(x), r)$, $g(x, \cdot)$ is strongly convex with modulus $\mu(\delta)/2$. The second term in \eqref{eq:grad_condition}, $\|\nabla_y g(x,y_{k^*})\| < (\mu(\delta))^2 / (4\overline{\overline{L}}_g)$, ensures that $y_{k^*}$ lies the ball $B(p_{j^*}, r)$. Specifically, it belongs to a local sublevel set $\mathcal{F}=\{y\in B(y^{\text{cd}}(x),r):g(x,y)-g(x,y^{\text{cd}}(x))\leq(\mu(\delta))^3/(16\overline{\overline{L}}_g^2)\}$. By \citet[Lemma 2]{nesterov2006cubic}, the algorithm will never leave this region $\mathcal{F}$ in subsequent iterations.\\
    
    \noindent\textbf{Phase II: Final Convergence.} The sequence now evolves inside a region where $g(x,\cdot)$ is strongly convex. Starting from $y_{k^*}$, we perform an additional $K_2$ iterations. Applying \citet[Theorem 1]{nesterov2006cubic} again for these $K_2$ steps, the choice of $K_2$ in \eqref{equation:requirement_for_K} guarantees finding a point $\hat{y}(x)$ such that:\begin{align*}|\nabla_y g(x,\hat{y}(x))| \leq \frac{\mu(\delta)\rho}{2}.\end{align*}Finally, using the $\frac{\mu(\delta)}{2}$-strong convexity of $g(x,\cdot)$ in the ball $B(y^{\text{cd}}(x),r)$, we have:\begin{align*}|\hat{y}(x) - y^{\text{cd}}(x)| \leq \frac{1}{\mu(\delta)/2} |\nabla_y g(x,\hat{y}(x))| \leq \frac{2}{\mu(\delta)} \cdot \frac{\mu(\delta)\rho}{2} = \rho.\end{align*}This completes the proof. 
    
\subsection{Proof of Proposition~\ref{proposition:gradient_error}}\label{proof:proposition:gradient_error}
Since we redefine $f(x,\cdot)$ to be constant $\overline{f}$ for $x \notin \mathcal{X}$ around \eqref{eq:extension}, it follows that
\begin{align*}
        & \left\|\nabla_x\pcd_{\xi }(x)-\nabla^K_x\phi^{\text{cd}}_\xi(x)\right\|\\
        =&\left\|\int_{\R^n}\nabla_xh_{\xi }(x-z)f(z,y^{\text{cd}}(z))dz-\int_{\R^n}\nabla_xh_{\xi }(x-z)f(z,\hat{y}(z))dz\right\|\\
        =&\left\|\int_{\mathcal{X}}\nabla_xh_{\xi }(x-z)f(z,y^{\text{cd}}(z))dz-\int_{\mathcal{X}}\nabla_xh_{\xi }(x-z)f(z,\hat{y}(z))dz\right\|\\
        \leq&\left\|\int_{\mathcal{X}\setminus\widetilde{\mathcal{X}}_{\delta}}\nabla_xh_{\xi }(x-z)\cdot\left|f(z,y^{\text{cd}}(z))-f(z,\hat{y}(z))\right|dz\right\|+\left\|\int_{\widetilde{\mathcal{X}}_{\delta}}\nabla_xh_{\xi }(x-z)\cdot\left|f(z,y^{\text{cd}}(z))-f(z,\hat{y}(z))\right|dz\right\|\\
\end{align*}
Note that $K$ satisfies \eqref{equation:requirement_for_K}, according to Lemma \ref{lemma:approximationerror}, we have $\|\hat{y}(z)-y^{\text{cd}}(z)\|\leq\rho$ for any $z\in\mathcal{X}\setminus\widetilde{\mathcal{X}}_{\delta}$. By Lipschitz continuity of $f$ (from Assumption~\ref{assumption:general1}(\ref{assumption:general1(10)})), it follows $|f(z,y^{\text{cd}}(z))-f(z,\hat{y}(z))|\leq L_f\rho$. From \eqref{eq:boundoff}, we know that $|f(z,y^{\text{cd}}(z))|\leq \overline{f}$ for any $z\in\mathcal{X}$. By an argument similar to that of \eqref{eq:boundoff}, we have $|f(z,\hat{y}(z))|\leq \overline{f}$ for any $z\in\mathcal{X}$ as well. Thus, the following holds:
\begin{align*}
        \left\|\nabla_x\pcd_{\xi }(x)-\nabla^K_x\phi^{\text{cd}}_\xi(x)\right\|\leq&L_f\rho \int_{\mathbb{R}^n}\left\|\nabla_xh_{\xi }(x-z)\right\|dz+2\overline{f}\int_{\widetilde{\mathcal{X}}_{\delta}}\left\|\nabla_xh_{\xi }(x-z)\right\|dz\\
        \overset{\text{(i)}}{\leq}&\sqrt{2}L_f\frac{\Gamma((n+1)/2)}{\Gamma(n/2)}\frac{\rho}{\xi }+2\overline{f}\lambda(\delta)\frac{e^{-1/2}}{(2\pi)^{n/2}{\xi }^{n+1}}\\
        \overset{\text{(ii)}}{\leq}&\sqrt{2}L_f\frac{\Gamma((n+1)/2)}{\Gamma(n/2)}\frac{\rho}{\xi }+\frac{2\overline{f} C e^{-1/2}}{(2\pi)^{n/2}\xi ^{n+1}}\delta^{(n-d)/2}\\
        =:&C_1(n,\xi)\rho+C_2(n,\xi)\delta^{(n-d)/2}.
\end{align*}
In (i) we use the following results:
\begin{align*}
    \int_{\mathbb{R}^n}\left\|\nabla_xh_{\xi }(x-z)\right\|dz=&\int_{\mathbb{R}^n}\frac{1}{\xi^2}\left\|u\right\|h_\xi(u)du=\frac{1}{\xi^2}\E_{W\sim\mathcal{N}(0,\xi^2I_n)}[\|W\|]=\frac{1}{\xi^2}\left(\sqrt{2}\xi\frac{\Gamma(\frac{n+1}{2})}{\Gamma(\frac{n}{2})}\right)=\frac{1}{\xi}\sqrt{2}\frac{\Gamma(\frac{n+1}{2})}{\Gamma(\frac{n}{2})},
\end{align*}
and
\begin{align*}
    \int_{\widetilde{\mathcal{X}}_{\delta}}\left\|\nabla_xh_{\xi }(x-z)\right\|dz\leq\lambda(\delta)\cdot\max_{z\in\R^n}\left\|\nabla_xh_{\xi }(x-z)\right\|=\lambda(\delta)\cdot\left\|\nabla_xh_{\xi }(z)\right\|\Big|_{\|z\|=\xi}=\lambda(\delta)\cdot\frac{e^{-1/2}}{(2\pi)^{n/2}{\xi }^{n+1}}.
\end{align*}
In (ii) we apply the bound $\lambda(\delta)\leq C\delta^{(n-d)/2}$ from Lemma~\ref{lem:measureofneighborhood}.
\subsection{Proof of Theorem~\ref{thm:biased_sgd_conv}}\label{proof:thm:biased_sgd_conv}

We define
$$\gamma_t:=\hat{\nabla}^{K_t}_x\phi^{\text{cd}}_{\xi }(x_t)-\nabla_x\pcd_\xi(x_t)$$
$$\Phi_{\mathcal{X},t}:=\frac{x_t-\proj_{\mathcal{X}}(x_t-\beta_t \nabla_x\pcd_\xi(x_t))}{\beta_t}$$
$$\widetilde{\Phi}_{\mathcal{X},t}:=\frac{x_t-\proj_{\mathcal{X}}(x_t-\beta_t \hat{\nabla}^{K_t}_x\phi^{\text{cd}}_{\xi }(x_t))}{\beta_t}.$$ 
Then the update \eqref{eq:update} can be written as $x_{t+1}=x_t-\beta_t\widetilde{\Phi}_{\mathcal{X},t}$. Since $\pcd_\xi$ is Lipschitz smooth by Proposition~\ref{proposition:lipschitz}, we have

\begin{align}\label{eq:descent1}
\pcd_\xi(x_{t+1})&\leq\pcd_\xi(x_t)+\langle\nabla_x\pcd_\xi(x_t),x_{t+1}-x_t\rangle+\frac{\overline{L}_{\phi^{\text{cd}}_{\xi}}}{2}\|x_{t+1}-x_t\|^2\nonumber\\
&=\pcd_\xi(x_t)
-\beta_t\langle\nabla_x \pcd_\xi(x_t),\widetilde{\Phi}_{\mathcal{X},t}\rangle
+\frac{\overline{L}_{\phi^{\text{cd}}_{\xi}}}{2}\beta_t^{2}\|\widetilde{\Phi}_{\mathcal{X},t}\|^{2}\nonumber\\
&=\pcd_\xi(x_t)-\beta_t\langle\hat{\nabla}^{K_t}_x\phi^{\text{cd}}_{\xi }(x_t),\widetilde{\Phi}_{\mathcal{X},t}\rangle
+\frac{\overline{L}_{\phi^{\text{cd}}_{\xi}}}{2}\beta_t^2\|\widetilde{\Phi}_{\mathcal{X},t}\|^2+\beta_t\langle\gamma_t,\widetilde{\Phi}_{\mathcal{X},t}\rangle.
\end{align}
From \citet[Lemma 6.4]{lan2020first} we have
$$\langle\hat{\nabla}^{K_t}_x\phi^{\text{cd}}_{\xi }(x_t),\widetilde{\Phi}_{\mathcal{X},t}\rangle\geq\|\widetilde{\Phi}_{\mathcal{X},t}\|^2.$$
Plug this into \eqref{eq:descent1}, we obtain
\begin{align*}
    \pcd_\xi(x_{t+1})&\leq \pcd_\xi(x_{t})-\beta_t\|\widetilde{\Phi}_{\mathcal{X},t}\|^2+\frac{\overline{L}_{\phi^{\text{cd}}_{\xi}}}{2}\beta_t^2\|\widetilde{\Phi}_{\mathcal{X},t}\|^2+\beta_t\langle\gamma_t,\Phi_{\mathcal{X},t}\rangle+\beta_t\langle\gamma_t,\widetilde{\Phi}_{\mathcal{X},t}-\Phi_{\mathcal{X},t}\rangle
\end{align*}
Then we have
\begin{align}
    \pcd_\xi(x_{t+1})&\leq\pcd_\xi(x_{t})-(\beta_t-\frac{\overline{L}_{\phi^{\text{cd}}_{\xi}}}{2}\beta_t^2)\|\widetilde{\Phi}_{\mathcal{X},t}\|^2+\beta_t\langle\gamma_t,\Phi_{\mathcal{X},t}\rangle+\beta_t\|\gamma_t\|\|\widetilde{\Phi}_{\mathcal{X},t}-\Phi_{\mathcal{X},t}\|\nonumber\\
    &\overset{\text{(i)}}{\leq}\pcd_\xi(x_{t})-(\beta_t-\frac{\overline{L}_{\phi^{\text{cd}}_{\xi}}}{2}\beta_t^2)\|\widetilde{\Phi}_{\mathcal{X},t}\|^2+\beta_t\langle\gamma_t,\Phi_{\mathcal{X},t}\rangle+\beta_t\|\gamma_t\|^2.\label{eq:descent2}
\end{align}
In (i) we use the result that
\begin{align*}
    \|\widetilde{\Phi}_{\mathcal{X},t}-\Phi_{\mathcal{X},t}\|&=\left\|\frac{x_t-\proj_{\mathcal{X}}(x_t-\beta_t \hat{\nabla}^{K_t}_x\phi^{\text{cd}}_{\xi }(x_t))}{\beta_t}-\frac{x_t-\proj_{\mathcal{X}}(x_t-\beta_t \nabla_x\pcd_\xi(x_t))}{\beta_t}\right\|\\
    &=\left\|\frac{\proj_{\mathcal{X}}(x_t-\beta_t \hat{\nabla}^{K_t}_x\phi^{\text{cd}}_{\xi }(x_t))-\proj_{\mathcal{X}}(x_t-\beta_t \nabla_x\pcd_\xi(x_t))}{\beta_t}\right\|\\
    &\leq\left\|\frac{(x_t-\beta_t \hat{\nabla}^{K_t}_x\phi^{\text{cd}}_{\xi }(x_t))-(x_t-\beta_t \nabla_x\pcd_\xi(x_t))}{\beta_t}\right\|\\
    &=\|\hat{\nabla}^{K_t}_x\phi^{\text{cd}}_{\xi }(x_t)-\nabla_x\pcd_\xi(x_t)\|\\
    &=\|\gamma_t\|.
\end{align*}
Summing up the \eqref{eq:descent2} for $t=1,...,T$ and noticing that $\beta_t\leq 1/\overline{L}_{\phi^{\text{cd}}_{\xi}}$, we obtain
\begin{align*}
    \sum_{t=1}^T(\beta_t-\overline{L}_{\phi^{\text{cd}}_{\xi}}\beta_t^2)\|\widetilde{\Phi}_{\mathcal{X},t}\|^2&\leq\sum_{t=1}^T(\beta_t-\frac{\overline{L}_{\phi^{\text{cd}}_{\xi}}}{2}\beta_t^2)\|\widetilde{\Phi}_{\mathcal{X},t}\|^2\\
    &\leq\pcd_\xi(x_{1})-\pcd_\xi(x_{T+1})+\sum_{t=1}^T\{\beta_t\langle\gamma_t,\Phi_{\mathcal{X},t}\rangle+\beta_t\|\gamma_t\|^2\}.
\end{align*}
Note that we assume that the number of lower-level iterations satisfies $K\geq K_0$ in Assumption \ref{assumption:descent}. $y^{\text{cd}}(x)$ is always in the level set $\{y:g(x,y)\leq g(x,y_0)\}$. By Proposition~\ref{prop:general}(\ref{assumption:general1(8)}), we also have $-\overline{f}\leq\pcd_\xi(x)\leq \overline{f}.$ It follows
\begin{align*}
    \sum_{t=1}^T(\beta_t-\overline{L}_{\phi^{\text{cd}}_{\xi}}\beta_t^2)\|\widetilde{\Phi}_{\mathcal{X},t}\|^2
    &\leq\pcd_\xi(x_{1})-\pcd_\xi(x_{T+1})+\sum_{t=1}^T\{\beta_t\langle\gamma_t,\Phi_{\mathcal{X},t}\rangle+\beta_t\|\gamma_t\|^2\}\\
    &\leq 2\overline{f}+\sum_{t=1}^T\{\beta_t\langle\gamma_t,\Phi_{\mathcal{X},t}\rangle+\beta_t\|\gamma_t\|^2\}.
\end{align*} 
Notice that the iterate $x_t$ is a function of the history, denoted by $\zeta_{[t-1]}$, of the generated random process and hence is random. Recall the definition of bias and variance in \eqref{eq:bias} and \eqref{eq:variance}, we have
\begin{align*}
    \Delta(\rho_t,\delta_t)&\geq \left\|\E\left[\hat{\nabla}^K_x\phi^{\text{cd}}_{\xi }(x_t)-\nabla_x\pcd_{\xi }(x_t)|\zeta_{[t-1]}\right]\right\|=\|\E[\gamma_t|\zeta_{[t-1]}]\|\\
    \sigma^2(N_t)&\geq\E\left[\left\|\hat{\nabla}^K_x\phi^{\text{cd}}_{\xi }(x_t)-\nabla^K_x\phi^{\text{cd}}_\xi(x_t)\right\|^2\Big|\zeta_{[t-1]}\right]=\E[\|\gamma_t\|^2|\zeta_{[t-1]}].
\end{align*}
According to Proposition~\ref{proposition:boundedgradient}, the following holds:
$$|\E[\langle\gamma_t,\Phi_{\mathcal{X},t}\rangle|\zeta_{[t-1]}]|=\|\langle\E[\gamma_t|\zeta_{[t-1]}],\Phi_{\mathcal{X},t}\rangle\|\leq\|\E[\gamma_t|\zeta_{[t-1]}]\|\|\Phi_{\mathcal{X},t}\|\leq\sqrt{\frac{2}{\pi}}\frac{\overline{f}}{\xi}\Delta(\rho_t,\delta_t)$$
We also have
$$\E[\|\gamma_t\|^2|\zeta_{[t-1]}]=\E[\|\hat{\nabla}^{K_t}_x\phi^{\text{cd}}_{\xi }(x_t)-\nabla_x\pcd_\xi(x_t)\|^2|\zeta_{[t-1]}]\leq\sigma^2(N_t)+(\Delta(\rho_t,\delta_t))^2.$$
Thus, we obtain
$$\sum_{t=1}^T(\beta_t-\overline{L}_{\phi^{\text{cd}}_{\xi}}\beta_t^2)\E[\|\widetilde{\Phi}_{\mathcal{X},t}\|^2]\leq 2\overline{f}+\sum_{t=1}^T\sqrt{\frac{2}{\pi}}\frac{\overline{f}}{\xi}\Delta(\rho_t,\delta_t)\beta_t+\sum_{t=1}^T(\sigma^2(N_t)+(\Delta(\rho_t,\delta_t))^2)\beta_t.$$
Then we conclude
$$\E[\|\widetilde{\Phi}_{\mathcal{X},R}\|^2]=\frac{\sum_{t=1}^T(\beta_t-\overline{L}_{\phi^{\text{cd}}_{\xi}}\beta_t^2)\E[\|\widetilde{\Phi}_{\mathcal{X},t}\|^2]}{\sum_{t=1}^T(\beta_t-\overline{L}_{\phi^{\text{cd}}_{\xi}}\beta_t^2)}\leq\frac{2\overline{f}+\sum_{t=1}^T\sqrt{\frac{2}{\pi}}\frac{\overline{f}}{\xi}\Delta(\rho_t,\delta_t)\beta_t+\sum_{t=1}^T(\sigma^2(N_t)+(\Delta(\rho_t,\delta_t))^2)\beta_t}{\sum_{t=1}^T(\beta_t-\overline{L}_{\phi^{\text{cd}}_{\xi}}\beta_t^2)}.$$

\subsection{Proof of Theorem~\ref{thm:globalversionmainthm}}\label{proof:thm:globalversionmainthm}

\begin{lemma}\label{lemma:localversionmainthm}
    Suppose $g(x,y)$ is three times continuously differentiable, and Assumptions~\ref{assumption:general1} holds. We denote by
    $$\Gamma:=\{(x,y)\in\mathcal{X}\times\R^m:\nabla_y g(x,y)=0,\ \det(\nabla^2_{yy}g(x,y))=0\}$$
    the set of degenerate stationary points. For any fold bifurcation stationary point $(\bar{x},\bar{y})$ of $g(x,y)$ respect to $y$, there exists a neighborhood $W$ of $(\bar{x},\bar{y})$ and two constants $C_1$ and $C_2$ such that
    $$|\lambda(x,y)|\geq \min\{C_1(\mathrm{dist}(\proj_x(\Gamma\cap W),x))^{1/2},C_2\}$$
    holds for any stationary point $(x,y)$ in $W$ with respect to $y$ for $g$, where $\proj_x:\R^{n+m}\to\R^n$ denotes the projection onto the $x$ coordinates, and $\lambda(x,y)$ is defined as follows
    \begin{align}\label{eq:smallesteignvalue}
        \lambda(x,y)=\lambda_{i^\ast}(\nabla^2_{yy}g(x,y)),\quad\text{where}\ i^\ast=\arg\min_{i}|\lambda_i(\nabla^2_{yy}g(x,y))|,
    \end{align}
    i.e., the eigenvalue of smallest absolute value of the Hessian with respect to $y$.
\end{lemma}

\begin{proof}
    In what follows, all stationary points refer to those of $g$ with respect to the $y$ variables. Throughout, subscripts of $x$ and $y$ denote coordinate indices. The core idea of the proof is to perform some coordinate transformations in a neighborhood $W$ of $(\bar{x},\bar{y})$. After these transformations, we identify a scalar function $A(x)$ such that a stationary point is degenerate if and only if $A(x)=0$. We then establish a relationship between $A(x)$ and the distance from $x$ to the set $\Gamma\cap W$ at the stationary point $(x,y)$.
    
    Without loss of generality, we suppose $(\bar{x},\bar{y})=(0,0)$. We perform an orthogonal transformation of the $y$-coordinate such that $\p y_1$ is the only degenerate direction of $\nabla_{yy}g(0,0)$, and the restriction of $\nabla_{yy}g(0,0)$ at the subspace of the tangent space spanned by $\p y_2,...,\p y_m$ is non-degenerate and diagonal. By the splitting lemma \citep{poston2014catastrophe}, we can find a neighborhood $W$ and a coordinate map preserving $(x,y_1)$ such that under this new coordinate $(x,y)$, $g$ has the following expression in $W$:
    \begin{align}
        g(x,y)=g_1(x,y_1)+\sum_{i=2}^m\epsilon_i y_i^2,\quad\text{where}\ \epsilon_i=1\ \text{or}\ -1.
    \end{align}
    We expand $g_1(x,\cdot)$ in a Taylor series:
    \begin{align}
        g_1(x,y_1)=g_1(x,0)+\frac{\p g_1}{\p y_1}(x,0)y_1+\frac{1}{2}\frac{\p^2 g_1}{\p y_1^2}(x,0)y_1^2+\frac{1}{6}\frac{\p^3 g_1}{\p y_1^3}(x,0)y_1^3+r(x,y_1)
    \end{align}
    where $r(x,y_1)=\mathcal{O}(y_1^4)$ with respect to $y_1$. Note that $(0,0)$ is a fold bifurcation point. By denoting
    \begin{align}
        a(x):=\frac{\p g_1}{\p y_1}(x,0),\quad b(x):=\frac{\p^2 g_1}{\p y_1^2}(x,0),\quad c(x)=\frac{1}{2}\frac{\p^3 g_1}{\p y_1^3}(x,0),
    \end{align}
    we have $a(0)=b(0)=0$ and $c(0)\ne 0$. By the second property of Definition~\ref{def:fold}, we can perform an orthogonal transformation of the $x$-coordinate such that
    \begin{align*}
        \frac{\p a}{\p x_1}(0)\ne 0,\quad \frac{\p a}{\p x_i}(0)=0\quad\text{for any}\ i\ne 1.
    \end{align*}
    We investigate the stationary equations under this new coordinate: 
    \begin{align}\label{eq:stationaryequation}
        \begin{cases}
            \frac{\p g}{\p y_1}(x,y)&=a(x)+b(x)y_1+c(x)y_1^2+r'_{y_1}(x,y_1)=0\\
            \frac{\p g}{\p y_2}(x,y)&=2\epsilon_2 y_2=0\\
            &\vdots\\
            \frac{\p g}{\p y_m}(x,y)&=2\epsilon_m y_m=0
        \end{cases}
    \end{align}
    where $r'_{y_1}(x,y_1)=\mathcal{O}(y_1^3)$ with respect to $y_1$. It is clear that under this new coordinate, the stationary point must satisfy $y_2=\cdots=y_m=0$. We next eliminate the linear term in the first equation of \eqref{eq:stationaryequation} to analyze the solution structure of \eqref{eq:stationaryequation}. Consider the following map:
    \begin{align}
        F(x,y_1)=b(x)+2c(x)y_1+r''_{y_1y_1}(x,y_1).
    \end{align}
    Noting that
    \begin{align*}
        F(0,0)=b(0)=0,\quad \frac{\p F}{\p y_1}(0,0)=2c(0)\ne 0,
    \end{align*}
    by the implicit function theorem, we obtain that there exists a continuously differentiable function $Y_1(x)$ satisfies $Y_1(0)=0$ and $F(x,Y_1(x))=0$. Therefore, the first equation of \eqref{eq:stationaryequation} can be written as
    \begin{align}\label{eq:uniform}
        \frac{\p g}{\p y_1}(x,y)=A(x)+C(x)(y_1-Y_1(x))^2+R(x,y_1-Y_1(x))=0,
    \end{align}
    for some functions $A$, $C$, $R$, where $R(x,y_1-Y_1(x))=\mathcal{O}((y_1-Y_1(x))^3)$ with respect to $y_1-Y_1(x)$. Note that
    \begin{align*}
        A(x)&=\frac{\p g}{\p y_1}(x,Y_1(x))=a(x)+b(x)Y_1(x)+c(x)(Y_1(x))^2+r'_{y_1}(x,Y_1(x))\\
        C(x)&=\frac{1}{2}\frac{\p^3 g}{\p y_1^3}(x,Y_1(x))=c(x)+\frac{1}{2}r'''_{y_1y_1y_1}(x,Y_1(x)).
    \end{align*}
    It is easy to see that $A(x)$ and $C(x)$ satisfy $A(0)=a(0)=0$, $C(0)=c(0)\ne 0$, and 
    \begin{align}\label{eq:propertyforwidetildeA}
        \frac{\p A}{\p x_1}(0)=\frac{\p a}{\p x_1}(0)\ne 0,\quad\text{and}\quad\frac{\p A}{\p x_i}(0)=\frac{\p a}{\p x_i}(0)=0\quad\text{for all }i\ne 1.
    \end{align}
    We then perform a translation in $y_1$ by $y_1\mapsto y_1-Y_1(x)$. Then the system \eqref{eq:stationaryequation} can be simplified to
    \begin{align}\label{eq:stationaryequation2}
        \begin{cases}
            \frac{\p g}{\p y_1}(x,y)&=A(x)+C(x)y_1^2+R(x,y_1)=0\\
            \frac{\p g}{\p y_2}(x,y)&=2\epsilon_2 y_2=0\\
            &\vdots\\
            \frac{\p g}{\p y_m}(x,y)&=2\epsilon_m y_m=0
        \end{cases}
    \end{align}
    where $R(x,y_1)=\mathcal{O}(y_1^3)$ with respect to $y_1$. We rewrite the first equation of \eqref{eq:stationaryequation2} by
    \begin{align}\label{eq:solutionofG}
        G(x,y_1):=A(x)+y_1^2\left(C(x)+\frac{R(x,y_1)}{y_1^2}\right).
    \end{align}
    Since $C(0)\ne 0$, we assume that the neighborhood $W$ is chosen appropriately, so that within which we can find two positive constants $0<\alpha<\beta$ such that
    \begin{align}\label{2cases}
        \alpha<C(x)+\frac{R(x,y_1)}{y_1^2}<\beta\quad\text{or}\quad -\beta<C(x)+\frac{R(x,y_1)}{y_1^2}<-\alpha.
    \end{align}
    The specific case depends on the sign of $C(0)$. Without loss of generality, we assume that $C(0)>0$, i.e., we are in the first case of \eqref{2cases}. From \eqref{eq:solutionofG} we can see 
    \begin{itemize}
        \item When $A(x)>0$, $G(x,y_1)$ has no solution with respect to $y_1$;
        \item When $A(x)=0$, $G(x,y_1)$ has only one solution $y_1=0$ with respect to $y_1$. In the new coordinate system, this is a degenerate stationary point;
        \item When $A(x)<0$, $G(x,y_1)$ has two solutions with respect to $y_1$. In the new coordinate system, both of these solutions are non-degenerate stationary points.
    \end{itemize}
    Therefore, whether $A(x)=0$ serves as a criterion for determining whether $x$ belongs to $\proj_x(\Gamma\cap W)$, i.e., the projection (in the $x$-direction) of the set of degenerate stationary points near $(0,0)$. Note that
    \begin{align*}
        &G(x,\sqrt{\frac{-A(x)}{\alpha}})>A(x)-\frac{A(x)}{\alpha}\cdot\alpha=0\quad G(x,\sqrt{\frac{-A(x)}{\beta}})<A(x)-\frac{A(x)}{\beta}\cdot\beta=0\\
        &G(x,-\sqrt{\frac{-A(x)}{\alpha}})>A(x)-\frac{A(x)}{\alpha}\cdot\alpha=0\quad G(x,-\sqrt{\frac{-A(x)}{\beta}})<A(x)-\frac{A(x)}{\beta}\cdot\beta=0.
    \end{align*}
    By the intermediate value theorem, for any $x$, there exists one root $y_1$ of $G(x,y_1)=0$ in each of the intervals
    $$\left(-\sqrt{\frac{-A(x)}{\alpha}},-\sqrt{\frac{-A(x)}{\beta}}\right)\quad\text{and}\quad\left(\sqrt{\frac{-A(x)}{\beta}},\sqrt{\frac{-A(x)}{\alpha}}\right).$$
    At these roots, we have
    \begin{align}\label{eq:hessianbound2}
        \left|\frac{\p^2 g}{\p y_1^2}\right|=2C(x)\left|\sqrt{\frac{-A(x)}{\beta}}\right|-\left|R(x,\sqrt{\frac{-A(x)}{\beta}}))\right|\geq M\sqrt{-A(x)}.
    \end{align}
    for some constant $M>0$. According to the system \eqref{eq:stationaryequation2}, in this new coordinate, $|\lambda(x,y)|$ defined in \eqref{eq:smallesteignvalue} can be computed as follows:
    \begin{align}
        |\lambda(x,y)|=\min_{i}\left|\frac{\p^2 g}{\p y_i^2}(x,y)\right|\geq\min\{M\sqrt{-A(x)},2\}.
    \end{align}
    
    We proceed to prove that these coordinate transformations preserve the stationary points, and then investigate bounds on $|\lambda(x,y)|$ in the original coordinates. We denote the original coordinate system by $(x,y)$, the new coordinate system by $(x',y')$. Note that, for the $x$-component, we only applied a coordinate transformation independent of $y$. Therefore, the coordinate transformation can be written as follows:
    \begin{align}\label{eq:coordinatetransformation}
        \phi:(x,y)\mapsto(x',y')=(\phi_1(x),\phi_2(x,y)), 
    \end{align}
    $$\phi^{-1}:(x',y')\mapsto(x,y)=(\phi^{-1}_1(x'),\phi^{-1}_2(x',y')).$$ 
    The function $g$ under the new coordinate system can be written as
    $$g'(x',y'):=g(x,y),\quad\text{where}\ (x',y')=\phi(x,y).$$
    By computation, we have 
    $$\nabla_{y'} g'(x',y')=\nabla_x g(x,y)\nabla_{y'}\phi_1^{-1}(x')+\nabla_y g(x,y)\nabla_{y'}\phi_2^{-1}(x',y')=\nabla_y g(x,y)\nabla_{y'}\phi_2^{-1}(x',y').$$
    Note that the full Jacobian of the coordinate transformation $\phi^{-1}$ is invertible, i.e., the following matrix is invertible:
    \begin{align*}
        \nabla_{(x',y')}\phi^{-1}(x',y')=\begin{pmatrix}
            \nabla_{x'} \phi^{-1}_1(x') &  \nabla_{y'} \phi^{-1}_1(x')\\
            \nabla_{x'} \phi^{-1}_2(x',y') &  \nabla_{y'} \phi^{-1}_2(x',y')
        \end{pmatrix}=\begin{pmatrix}
            \nabla_{x'} \phi^{-1}_1(x') &  0\\
            \nabla_{x'} \phi^{-1}_2(x',y') &  \nabla_{y'} \phi^{-1}_2(x',y')
        \end{pmatrix}.
    \end{align*}
    Thus, $\nabla_{y'}\phi_2^{-1}(x',y')$ is invertible. This implies $\nabla_{y'} g'(x',y')=0$ is equivalent to $\nabla_y g(x,y)=0$. By the chain rule, when $(x,y)$ is a stationary point, or equivalently $(x',y')$ is a stationary point, we have
     \begin{align}
         \nabla^2_{y'y'}g'(x',y')=&\nabla_{y'} \phi^{-1}_2(x',y')\nabla^2_{yy}g(x,y)(\nabla_{y'} \phi^{-1}_2(x',y'))^\top+\left[\sum_{k=1}^m\frac{\p g'}{\p y_k'}(x',y')\frac{\p^2 (\phi_2^{-1})_k}{\p y'_i \p y'_j}(x',y')\right]_{i,j=1}^m\nonumber\\
        =&\nabla_{y'} \phi^{-1}_2(x',y')\nabla^2_{yy}g(x,y)(\nabla_{y'} \phi^{-1}_2(x',y'))^\top.
    \end{align}
    Since the full Jacobian $\nabla_{(x',y')}\phi^{-1}(x',y')$ is bounded, it is clear that $\nabla_{y'} \phi^{-1}_2(x',y')$ is also bounded, i.e., $\|\nabla_{y'} \phi^{-1}_2(x',y')\|\leq \overline{M}$ for some constant $\overline{M}$. Thus, we obtain the following singular value inequality:
    $$\sigma_{\text{min}}(\nabla^2_{y'y'}g'(x',y'))\leq \overline{M}^2\sigma_{\text{min}}(\nabla^2_{yy}g(x,y)).$$ 
    Note that for a symmetric matrix, the singular values are precisely the absolute values of its eigenvalues. Hence, in the original coordinate system we still have the bound
    \begin{align}\label{eq:hessianbound3}
        |\lambda(x,y)|\geq\frac{1}{\overline{M}^2}\min\{M\sqrt{-A(x)},2\}.
    \end{align}

    We next study the relationship between $A(x)$ and the distance from $x$ to $\proj_x(\Gamma\cap W)$. Consider the following mapping
    $$(x_1,...,x_n)\mapsto(x_1',x_2',...,x_n')=(A(x),x_2,...,x_n).$$
    The Jacobian of this mapping at $x=0$ is
    \begin{align}\label{eq:jacobianofA}
        \begin{pmatrix}
        \frac{\p A}{\p x_1}(0) & 0 & \cdots & 0\\
        0 & 1 & \cdots & 0\\
        \vdots & \vdots & \ddots & \vdots \\
        0 & 0 & \cdots & 1
    \end{pmatrix}
    \end{align}
    which is invertible. By the inverse function theorem, this mapping defines a local coordinate transformation, and is therefore bi-Lipschitz. By shrinking $W$ if necessary, we may assume that the mapping is bi-Lipschitz on the entire set $W$. In the new coordinate system, the set $\proj_x(\Gamma\cap W)$ corresponds exactly to the hyperplane $x_1'=0$. Consequently, the distance from a point $x'$ to this set is simply $|x_1'| = |A(x)|$. Invoking the bi-Lipschitz property, the distance in the original coordinates is bounded by $\mathrm{dist}(x, \text{Proj}_x(\Gamma\cap W)) \le L|A(x)|$ for some constant $L>0$. Combining with \eqref{eq:hessianbound3}, we obtain that there exists two constants $C_1$ and $C_2$ such that the following holds for any stationary point $(x,y)$ in $W$
    $$|\lambda(x,y)|\geq \min\{C_1(\mathrm{dist}(\proj_x(\Gamma\cap W),x))^{1/2},C_2\}.$$

\end{proof}

\noindent\textbf{Detailed proof of Theorem~\ref{thm:globalversionmainthm}:} In what follows, all gradient norms and stationary points refer to those of $g$ with respect to the $y$ variables. Throughout, subscripts of $x$ and $y$ denote coordinate indices.\\

\noindent\textbf{Proof of (i):} Since every degenerate stationary point is a fold bifurcation point, points in $\Gamma\cap\mathcal{Y}'$ are also fold bifurcation points as well. We can apply Lemma~\ref{lemma:localversionmainthm} to obtain a corresponding neighborhood $W$ around each of them. Note that $\Gamma\cap\mathcal{Y}'\subset\mathcal{X}\times\mathcal{Y}'$ is closed and bounded, and thus compact, we can extract a finite collection of neighborhoods $\{W_i\}_{i=1}^I$, each satisfying the conditions of Lemma~\ref{lemma:localversionmainthm}, that together cover $\Gamma\cap\mathcal{Y}'$. For any stationary point $(x,y)\in W_i$, we observe that
$$\mathrm{dist}(\proj_x(\Gamma\cap W_i),x)\geq\mathrm{dist}(\proj_x(\Gamma),x),$$
and hence Lemma~\ref{lemma:localversionmainthm} implies
$$|\lambda(x,y)|\geq \min\{C_1(\mathrm{dist}(\proj_x\Gamma,x))^{1/2},C_2\},$$
where $\lambda(x,y)$ is defined in \eqref{eq:smallesteignvalue}. For stationary points within $\mathcal{X}\times\mathcal{Y}'$ but outside $\{W_i\}_{i=1}^I$, $|\lambda(x,y)|$ is uniformly bounded below by a positive constant. Note that $x\notin\widetilde{\mathcal{X}}_\delta$ implies that $\mathrm{dist}(\proj_x(\Gamma),x)\geq\delta$. Recalling the definition of $\mu(\delta)$ in Lemma~\ref{lemma:eigenvalue} and combining the two cases, we conclude that there exists two constants $D_1$ and $D_2$ such that for any $\delta>0$ the following holds
    \begin{align}\label{eq:finalresult1}
        \mu(\delta)&\geq\min\{D_1\sqrt{\delta},D_2\}.
    \end{align}

\noindent\textbf{Proof of (ii):} We define
$$K(\delta,r):=\{(x,y)\in\overline{\mathcal{X}\setminus\widetilde{\mathcal{X}}_\delta}\times\mathcal{Y}:y\in\overline{\mathcal{Y}\setminus \mathrm{Crit}_r(x)}\}$$
Note that the rate of $\alpha(\delta,\mu(\delta)/(2\overline{\overline{L}}_g))$ with respect to $\delta$ is exactly the rate, in terms of $\delta$, of the minimal gradient norm of $g$ with respect to $y$ over the set $K(\delta,\mu(\delta)/(2\overline{\overline{L}}_g))$. Our approach is as follows: for each stationary point in $\mathcal{X}\times\mathcal{Y}$, we construct an open neighborhood $W$ that is independent of $\delta$. Since the set of stationary points in $\mathcal{X}\times\mathcal{Y}$ is closed, we can find finitely many such open neighborhoods $\{W_i\}_{i=1}^I$ that cover the entire set of stationary points. Then, on $(\mathcal{X}\times\mathcal{Y})\setminus(\cup_{i=1}^I W_i)$, the gradient norm of $g$ with respect to $y$ naturally admits a positive lower bound independent of $\delta$. Therefore, it suffices to analyze the lower bound of the gradient norm within each set $W_i \cap K(\delta, \mu(\delta)/(2\overline{\overline{L}}_g))$. We construct open neighborhoods separately according to the following cases:

\noindent\textbf{The first case:} The pair $(\bar{x},\bar{y})$ is a stationary point with $\bar{x} \notin \widetilde{\mathcal{X}}$. This means that $\nabla_{yy}^2 g(\bar{x},\bar{y})$ is non-degenerate. By the implicit function theorem, $\bar{y}$ can be locally extended to a continuous function $\bar{y}(x)$ for $x \in B_{\bar{x}}(\Delta_x)$, where $B_{\bar{x}}(\Delta_x)$ denotes the ball centered at $\bar{x}$ with radius $\Delta_x$. Here $\Delta_x$ is chosen sufficiently small such that $\nabla_{yy}^2 g(x,\bar{y}(x))$ remains non-degenerate for all $x \in B_{\bar{x}}(\Delta_x)$. Then the absolute values of all eigenvalues of $\nabla_{yy}^2 g(x,\bar{y}(x))$ admit a uniform positive lower bound, denoted by $\lambda$. We then define the open set corresponding to $(\bar{x},\bar{y})$ as
    $$
    W=\{(x,y): \; x \in B_{\bar{x}}(\Delta_x), \; y \in B_{\bar{y}(x)}(\lambda/(2\overline{\overline{L}}_g))\}.
    $$
    Next, we consider the intersection of $W$ and $K(\delta, \mu(\delta)/(2\overline{\overline{L}}_g))$, and estimate the lower bound of the gradient norm of $g$ with respect to $y$ on this intersection. To ensure that the intersection is non-empty, we may assume $\delta<\Delta_x$ and $\mu(\delta)/(2\overline{\overline{L}}_g))<\lambda/(2\overline{\overline{L}}_g))$. Fix $x \in B_{\bar{x}}(\Delta_x)$, and consider the slice $\{x\}\times B_{\bar{y}(x)}(\lambda/(2\overline{\overline{L}}_g))$. 
    By Lipschitz continuity of Hessian matrix (from Assumption~\ref{assumption:general1}(\ref{assumption:general1(11)})), we know that for $y \in B_{\bar{y}(x)}(\lambda/2\overline{\overline{L}}_g)$, it holds that
    $$
    \lambda(x,y) \geq \lambda(x,\bar{y}(x))-\overline{\overline{L}}_g\times\frac{\lambda}{2\overline{\overline{L}}_g}\geq\lambda-\frac{\lambda}{2}=\frac{\lambda}{2},
    $$
    where $\lambda(x,y)$ denotes the minimum absolute value of the eigenvalues of $\nabla_{yy}^2g(x,y)$. Consequently, on the slice $\{x\}\times B_{\bar{y}(x)}(\lambda/(2\overline{\overline{L}}_g))$, the gradient norm satisfies
    $$
    \|\nabla_y g(x,y)\| \geq \frac{\lambda}{2}\|y-\bar{y}(x)\|,
    $$
    and the minimal value of $\|\nabla_y g(x,y)\|$ at $W\cap K(\delta, \mu(\delta)/(2\overline{\overline{L}}_g))$ is attained at the boundary of $B_{\bar{y}(x)}(\mu(\delta)/(2\overline{\overline{L}}_g))$, yielding
    $$
    \|\nabla_y g(x,y)\| \geq \frac{\lambda\,\mu(\delta)}{4\overline{\overline{L}}_g}.
    $$
    Figure~\ref{fig:case1} provides an illustrative depiction.

    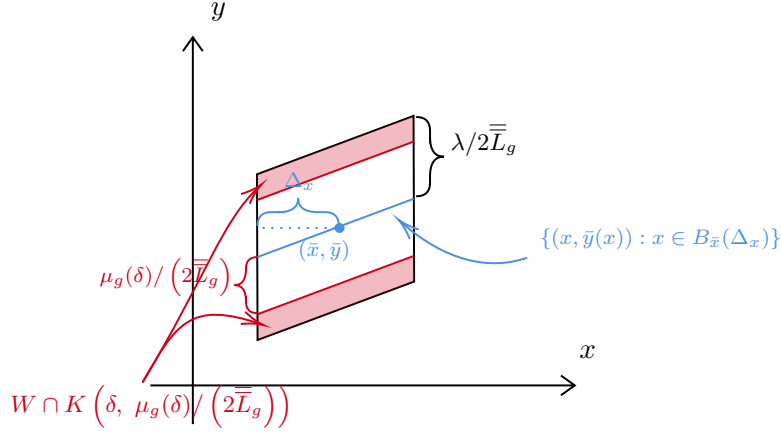
\begin{figure}[htbp]
    \centering
    \tikzset{every picture/.style={line width=0.75pt}} 

\begin{tikzpicture}[x=0.75pt,y=0.75pt,yscale=-1,xscale=1]

\draw   (157.05,201.8) -- (156.74,118.19) -- (235.71,88.69) -- (236.01,172.3) -- cycle ;
\draw [color={rgb, 255:red, 74; green, 144; blue, 226 }  ,draw opacity=1 ]   (156.89,159.99) -- (235.86,130.5) ;
\draw  (103,224.73) -- (317.03,224.73)(124.4,49) -- (124.4,244.26) (310.03,219.73) -- (317.03,224.73) -- (310.03,229.73) (119.4,56) -- (124.4,49) -- (129.4,56)  ;
\draw [color={rgb, 255:red, 74; green, 144; blue, 226 }  ,draw opacity=1 ]   (293.03,160.26) .. controls (273.43,167.12) and (250.95,161.49) .. (229.32,142.2) ;
\draw [shift={(228,141)}, rotate = 42.6] [color={rgb, 255:red, 74; green, 144; blue, 226 }  ,draw opacity=1 ][line width=0.75]    (10.93,-3.29) .. controls (6.95,-1.4) and (3.31,-0.3) .. (0,0) .. controls (3.31,0.3) and (6.95,1.4) .. (10.93,3.29)   ;
\draw   (237.03,129.26) .. controls (241.7,129.26) and (244.03,126.93) .. (244.03,122.26) -- (244.03,119.26) .. controls (244.03,112.59) and (246.36,109.26) .. (251.03,109.26) .. controls (246.36,109.26) and (244.03,105.93) .. (244.03,99.26)(244.03,102.26) -- (244.03,96.26) .. controls (244.03,91.59) and (241.7,89.26) .. (237.03,89.26) ;
\draw [color={rgb, 255:red, 74; green, 144; blue, 226 }  ,draw opacity=1 ] [dash pattern={on 0.84pt off 2.51pt}]  (196.38,145.23) -- (157.03,145.23) ;
\draw  [color={rgb, 255:red, 74; green, 144; blue, 226 }  ,draw opacity=1 ][fill={rgb, 255:red, 74; green, 144; blue, 226 }  ,fill opacity=1 ] (196.38,145.25) .. controls (196.38,144.13) and (197.28,143.23) .. (198.39,143.23) .. controls (199.5,143.23) and (200.41,144.13) .. (200.41,145.25) .. controls (200.41,146.36) and (199.5,147.26) .. (198.39,147.26) .. controls (197.28,147.26) and (196.38,146.36) .. (196.38,145.25) -- cycle ;
\draw  [color={rgb, 255:red, 74; green, 144; blue, 226 }  ,draw opacity=1 ][line width=0.75]  (198.03,144.26) .. controls (198.03,139.59) and (195.7,137.26) .. (191.03,137.26) -- (187.53,137.26) .. controls (180.86,137.26) and (177.53,134.93) .. (177.53,130.26) .. controls (177.53,134.93) and (174.2,137.26) .. (167.53,137.26)(170.53,137.26) -- (164.03,137.26) .. controls (159.36,137.26) and (157.03,139.59) .. (157.03,144.26) ;
\draw [color={rgb, 255:red, 208; green, 2; blue, 27 }  ,draw opacity=1 ]   (157.05,188.8) -- (236.01,159.3) ;
\draw [color={rgb, 255:red, 208; green, 2; blue, 27 }  ,draw opacity=1 ]   (156.74,131.19) -- (235.71,101.69) ;
\draw  [color={rgb, 255:red, 255; green, 255; blue, 255 }  ,draw opacity=0 ][fill={rgb, 255:red, 208; green, 2; blue, 27 }  ,fill opacity=0.27 ] (156.79,131.02) -- (156.74,118.18) -- (235.71,88.69) -- (235.75,101.52) -- cycle ;
\draw  [color={rgb, 255:red, 255; green, 255; blue, 255 }  ,draw opacity=0 ][fill={rgb, 255:red, 208; green, 2; blue, 27 }  ,fill opacity=0.27 ] (157.05,201.8) -- (157,188.97) -- (235.97,159.48) -- (236.01,172.31) -- cycle ;
\draw [color={rgb, 255:red, 208; green, 2; blue, 27 }  ,draw opacity=1 ]   (99.03,223.26) .. controls (131.37,169.36) and (138.74,142.35) .. (157.84,124.35) ;
\draw [shift={(159.03,123.26)}, rotate = 138.01] [color={rgb, 255:red, 208; green, 2; blue, 27 }  ,draw opacity=1 ][line width=0.75]    (10.93,-3.29) .. controls (6.95,-1.4) and (3.31,-0.3) .. (0,0) .. controls (3.31,0.3) and (6.95,1.4) .. (10.93,3.29)   ;
\draw [color={rgb, 255:red, 208; green, 2; blue, 27 }  ,draw opacity=1 ]   (99.03,223.26) .. controls (115.77,193.71) and (122.82,181.62) .. (158.38,193.69) ;
\draw [shift={(160.03,194.26)}, rotate = 199.36] [color={rgb, 255:red, 208; green, 2; blue, 27 }  ,draw opacity=1 ][line width=0.75]    (10.93,-3.29) .. controls (6.95,-1.4) and (3.31,-0.3) .. (0,0) .. controls (3.31,0.3) and (6.95,1.4) .. (10.93,3.29)   ;
\draw  [color={rgb, 255:red, 208; green, 2; blue, 27 }  ,draw opacity=1 ] (157.03,159.97) .. controls (153.15,159.97) and (151.21,161.91) .. (151.21,165.79) -- (151.21,165.79) .. controls (151.21,171.34) and (149.27,174.11) .. (145.38,174.11) .. controls (149.27,174.11) and (151.21,176.88) .. (151.21,182.43)(151.21,179.94) -- (151.21,182.43) .. controls (151.21,186.32) and (153.15,188.26) .. (157.03,188.26) ;

\draw (176,150) node [anchor=north west][inner sep=0.75pt]  [font=\scriptsize,color={rgb, 255:red, 74; green, 144; blue, 226 }  ,opacity=1 ]  {$( \bar{x} ,\bar{y})$};
\draw (318,203) node [anchor=north west][inner sep=0.75pt]    {$x$};
\draw (132,30) node [anchor=north west][inner sep=0.75pt]    {$y$};
\draw (298,144) node [anchor=north west][inner sep=0.75pt]  [font=\scriptsize,color={rgb, 255:red, 74; green, 144; blue, 226 }  ,opacity=1 ]  {$\{( x,\bar{y}( x)) :x\in B_{\bar{x}}( \Delta _{x})\}$};
\draw (253,91) node [anchor=north west][inner sep=0.75pt]  [font=\footnotesize]  {$\lambda /2\overline{\overline{L}}_{g}$};
\draw (169,116) node [anchor=north west][inner sep=0.75pt]  [font=\scriptsize,color={rgb, 255:red, 74; green, 144; blue, 226 }  ,opacity=1 ]  {$\Delta _{x}$};
\draw (31,223) node [anchor=north west][inner sep=0.75pt]  [font=\footnotesize,color={rgb, 255:red, 208; green, 2; blue, 27 }  ,opacity=1 ]  {$W\cap K\left( \delta ,\ \mu _{g}( \delta ) /\left( 2\overline{\overline{L}}_{g}\right)\right)$};
\draw (76,159) node [anchor=north west][inner sep=0.75pt]  [font=\scriptsize,color={rgb, 255:red, 208; green, 2; blue, 27 }  ,opacity=1 ]  {$\mu _{g}( \delta ) /\left( 2\overline{\overline{L}}_{g}\right)$};

\end{tikzpicture}

    \caption{Illustration of the construction of $W$ in the case where $(\bar{x},\bar{y})$ is a stationary point with $\bar{x} \notin \widetilde{\mathcal{X}}$; the red region indicates the intersection of $W$ and $K(\delta, \mu(\delta)/(2\overline{\overline{L}}_g))$.}
    \label{fig:case1}
\end{figure}
    
\noindent\textbf{The second case:}  The pair $(\bar{x},\bar{y})$ is a stationary point with $\bar{x} \in \widetilde{\mathcal{X}}$, but it is a non-degenerate stationary point, i.e., $\nabla_{yy} g(\bar{x},\bar{y})$ is non-degenerate. This case is almost the same as Case 1. Since $(\bar{x},\bar{y})$ is a non-degenerate stationary point, we can also construct a locally defined continuous curve of stationary points $\bar{y}(x)$. Along this curve, the absolute values of the eigenvalues of the Hessian admit a uniform positive lower bound $\lambda$, and thus we can construct $W$ in the same way as in Case 1. The only difference is that now the intersection of $W$ and $K(\delta, \mu(\delta)/(2\overline{\overline{L}}_g))$ is further intersected with $(\mathcal{X}\setminus\widetilde{\mathcal{X}}_\delta)\times \mathcal{Y}$ compared to Case 1. However, taking this additional intersection can only increase the minimal value of the gradient norm, and hence, we still obtain
    $$
    \|\nabla_y g(x,y)\| \geq \frac{\lambda\,\mu(\delta)}{4\overline{\overline{L}}_g}.
    $$

\noindent\textbf{The third case:}  The pair $(\bar{x},\bar{y})$ is a stationary point with $\bar{x} \in \widetilde{\mathcal{X}}$, and it is a degenerate stationary point, i.e., $\nabla_{yy} g(\bar{x},\bar{y})$ is degenerate. We first perform the coordinate transformation~\eqref{eq:coordinatetransformation} in the proof of Lemma~\ref{lemma:localversionmainthm}. In this coordinate, the equation system of stationary points has the form~\eqref{eq:stationaryequation2}. Recall the functions $A$, $C$, $R$ established in \eqref{eq:stationaryequation2}, we suppose that $C(0)>0$. We further perform the following two coordinate transformations
    \begin{align}
        \Phi_{1}:((x_1,...,x_n),y)\mapsto((A(x),x_2,...,x_n),y).
    \end{align}
    \begin{align}
        \Phi_{2}:(x,(y_1,...,y_m))\mapsto(x,(y_1\sqrt{C(x)+\frac{R(x,y_1)}{y_1^2}},y_2,...,y_m)).
    \end{align}
    We denote the original coordinate system by $(x,y)$, and the new system after performing the coordinate transformation~\eqref{eq:coordinatetransformation} and $\Phi_1$, $\Phi_2$ by $(x',y')$. The coordinate mapping is denoted by $\Psi:(x,y)\mapsto (x',y')$. Under this new coordinate, the equation system of stationary points has the form:
    \begin{align}\label{eq:stationarysystem3}
        \begin{cases}
            x_1'+y_1'^2=0\\
            2\epsilon_2 y'_2=0\\
            &\vdots\\
            2\epsilon_m y'_m=0
        \end{cases}
    \end{align}
    
    Without loss of generality, we suppose $(\bar{x},\bar{y})=(0,0)$. It is clear that $(\bar{x}',\bar{y}')=\Psi(x,y)=(0,0)$. We then choose the open neighborhood $W$ of $(\bar{x},\bar{y})=(0,0)$ such that $\Psi(W)$ is a box neighborhood of $(\bar{x}',\bar{y}')=(0,0)$ of the form $\{(x',y'):|x'_i|<r,\ |y'_j|<r\text{ for all } i,j\}$ for some constant $r>0$, i.e.,
    $$W=\Psi^{-1}(\{(x',y'):|x'_i|<R,\ |y'_j|<r\text{ for all } i,j\})$$
    After applying the coordinate transformations $\Phi_1$ and $\Phi_2$, the set $\Psi(\widetilde{X})$ is locally given by the $(n-1)$-dimensional manifold defined by $x'_1=0$.
    
    Next, we analyze the minimum of the gradient norm over the set $W \cap K(\delta, \mu(\delta)/(2\overline{\overline{L}}_g))$. Note that the set $W\cap K(\delta, \mu(\delta)/(2\overline{\overline{L}}_g))$ can essentially be viewed as obtained from $W$ by removing two subsets:
    \begin{itemize}
        \item \textbf{Set 1:} the points belonging to $\widetilde{\mathcal{X}}_\delta\times\mathcal{Y}$;
        \item \textbf{Set 2:} the $\mu(\delta)/(2\overline{\overline{L}}_g)$-neighborhoods of the set of stationary points in the fiber sense, i.e., for each stationary point $(x,y)$, we remove the set $\{x\}\times B_y(\mu(\delta)/(2\overline{\overline{L}}_g))$.
    \end{itemize}  
    We then turn to examining the properties of these two sets in the new coordinate system $(x',y')$.\\

    By the design of the coordinate transformation~\eqref{eq:coordinatetransformation} in the proof of Lemma~\ref{lemma:localversionmainthm}, and $\Phi_1$,$\Phi_2$, the overall coordinate transformation $\Psi$ takes the form
    $$\Psi:(x,y)\mapsto (x',y')=(\Psi_1(x),\Psi_2(x,y)),$$
    where both $\Psi$ and its $x$-component $\Psi_1(x)$ are bi-Lipschitz. \textbf{Set 1} consists of all points whose $x$-component lies within a distance at most $\delta$ from $\widetilde{\mathcal{X}}$. By the bi-Lipschitz property of $\Psi_1$, there exists a constant $D_1>0$ such that the image of this set under $\Psi$ contains all points whose $x'$-component lies within distance at most $D_1\delta$ from $\Psi_1(\widetilde{\mathcal{X}})$. Since $\Psi_1(\widetilde{\mathcal{X}})$ is precisely the $(n-1)$-dimensional manifold defined by $x_1'=0$, it follows that the image of the first set under $\Psi$ contains the set of points satisfying $|x_1'|\leq D_1\delta$. \textbf{Set 2} removes $\{x\}\times B_{y}(\mu(\delta)/(2\overline{\overline{L}}_g))$ for each stationary points $(x,y)$. Since $\Psi$ is fiber-preserving, whenever two points $(x,y)$ and $(x,\hat y)$ lie in the same fiber $\{x\}\times\mathcal{Y}$, their images $(x',y')$ and $(x',\hat y')$ also lie in the same fiber $\{x'\}\times\mathcal{Y}$. This property allows us to restrict the global bi-Lipschitz continuity of $\Psi$ to each fiber: there exists a uniform constant $D_2>0$, independent of $x$, such that for all $x$ and $y,\hat y$
    $$D_2\|y-\hat{y}\|\leq\|\Psi_2(x,y)-\Psi_2(x,\hat{y})\|\leq \frac{1}{D_2}\|y-\hat{y}\|.$$
    
    Thus, for every stationary point $(x',y')$, the image of the removed neighborhood under $\Psi$ contains the fiberwise ball $\{x'\}\times B_{y'}(D_2\mu(\delta)/(2\overline{\overline{L}}_g))$. Thus, in the new coordinates $(x',y')$, it suffices to consider the set $\Psi(W)$ after removing two regions: (i) the strip $\{|x_1'|\leq D_1\delta\}$, and (ii) the $D_2\mu(\delta)/(2\overline{\overline{L}}_g)$-neighborhoods of the stationary points in the fiber sense. Denote this remaining set by $Z$. By construction, we have  
    $$
    Z \supseteq \Psi\!\left(W \cap K\left(\delta,\tfrac{\mu(\delta)}{2\overline{\overline{L}}_g}\right)\right).
    $$  
    Therefore, obtaining a lower bound for the gradient norm over $Z$ immediately yields a lower bound for the gradient norm over $\Psi(W \cap K(\delta,\mu(\delta)/(2\overline{\overline{L}}_g)))$. Since the coordinate transformation $\Psi$ is bi-Lipschitz, the same bound converts into a gradient norm bound on the set $W \cap K(\delta,\mu(\delta)/(2\overline{\overline{L}}_g))$ in the original coordinates, up to a constant factor depending only on the bi-Lipschitz constant of $\Psi$.
    
    We identify the point in $Z$ where the gradient norm attains its minimum according to system~\eqref{eq:stationarysystem3}. For clarity, we consider the one-dimensional case for both $x'$ and $y'$. In fact, the higher-dimensional case is completely analogous, since only $x'_1$ and $y'_1$ are nontrivial in this system. This system of stationary points reduces to
    \begin{align}
        \frac{\partial g}{\partial y'}=x'+{y'}^2=0.
    \end{align}
    We distinguish two cases: $x'>\delta$ and $x'<-\delta$. For $x'>\delta$, the situation is relatively simple, since we directly obtain $|\partial g/\partial y'|\geq\delta.$ For $x'<-\delta$, we consider the intersection of the fiber $\{x'\}\times \mathcal{Y}$ with $Z$. This intersection is a union of intervals. Specifically, when  $\bar{r}:=D_2\mu(\delta)/2\overline{\overline{L}}_g\leq\sqrt{-x'}$, the set $\{x'\}\times \mathcal{Y}$ intersected with $Z$ is given by  
    $$
    \Psi(W)\cap\Bigl(\{x'\}\times\bigl((-\infty,\,-\sqrt{-x'}-\bar{r})\cup(-\sqrt{-x'}+\bar{r},\,\sqrt{-x'}-\bar{r})\cup(\sqrt{-x'}+\bar{r},+\infty)\bigr)\Bigr),
    $$
    and when $D_2\mu(\delta)/2\overline{\overline{L}}_g>\sqrt{-x'}$, the set $\{x'\}\times \mathcal{Y}$ intersected with $Z$ is given by  
    $$
    \Psi(W)\cap\Bigl(\{x'\}\times\bigl((-\infty,\,-\sqrt{-x'}-\bar{r})\cup(\sqrt{-x'}+\bar{r},+\infty)\bigr)\Bigr).
    $$
    By carrying out the computations on each of the above intervals and applying \eqref{eq:finalresult1}, it is easy to see that there exist constants $D_3,D_4>0$ such that, for every point $(x',y')$ in $Z$, we have
    $$
    \|\nabla_y g(x',y')\|\geq\min\{D_3\delta,D_4\}.
    $$
    Since the coordinate transformation $\Psi$ is bi-Lipschitz, the same conclusion holds (up to a constant factor) in the original coordinates. In particular, we obtain that for each point $(x,y)$ in $W \cap K(\delta,\mu(\delta)/(2\overline{\overline{L}}_g))$, it holds that  
    $$
    \|\nabla_y g(x,y)\|\geq D_2\min\{D_3\delta,D_4\}.
    $$

By combining the above three cases and applying \eqref{eq:finalresult1}, and noting that on $(\mathcal{X}\times\mathcal{Y})\setminus(\cup_{i=1}^I W_i)$ the gradient norm of $g$ with respect to $y$ naturally admits a positive lower bound independent of $\delta$, we conclude that there exist constants $C_3,C_4>0$ such that
    $$\alpha\left(\delta,\frac{\mu(\delta)}{2\overline{\overline{L}}_g}\right)\geq\min\{C_3\delta,C_4\}.$$

\end{document}